\documentclass[letterpaper, 11pt]{amsart}

\usepackage{fullpage}
\usepackage{graphicx}
\usepackage{amsmath, amssymb, amsfonts, amsthm}
\usepackage{url}
\usepackage{slashbox}

\usepackage{tikz} 

\usepackage[ruled,vlined]{algorithm2e}

\usepackage{color}
\definecolor{darkblue}{rgb}{0,0,0.4}
\usepackage[plainpages=false,colorlinks=true,citecolor=blue,filecolor=black,linkcolor=red,urlcolor=darkblue]{hyperref}

\newtheorem{thm}{Theorem}[section]

\newtheorem{prop}[thm]{Proposition}
\newtheorem{lem}[thm]{Lemma}
\theoremstyle{remark}
\newtheorem{rem}[thm]{Remark}

\newcommand{\x}{{\mathbf x}}

\graphicspath{{./figs/}} 

\usepackage[backend=bibtex,firstinits=true,maxcitenames=2,maxbibnames=10,style=alphabetic,natbib=true,url=false,sorting=nyt,doi=true,backref=false]{biblatex}
\renewbibmacro{in:}{\ifentrytype{article}{}{\printtext{\bibstring{in}\intitlepunct}}}

\begin{document}

\title{A generalized MBO diffusion generated motion \\ for orthogonal matrix-valued fields}
\author{Braxton Osting}
\thanks{B. Osting is partially supported by NSF DMS 16-19755.}
\author{Dong Wang}
\address{Department of Mathematics, University of Utah, Salt Lake City, UT}
\email{\{osting,dwang\}@math.utah.edu}

\subjclass[2010]{35K93, 
35K05, 
65M12} 

\keywords{Allen-Cahn equation, Ginzburg-Landau equation, Merriman-Bence-Osher (MBO) diffusion generated motion, Lyapunov functional, orthogonal matrix-valued field}

\date{\today}

\begin{abstract} 
We consider the problem of finding stationary points of the Dirichlet energy for orthogonal matrix-valued fields. Following the Ginzburg-Landau approach, this energy is relaxed by penalizing the matrix-valued field when it does not take orthogonal matrix values. A generalization of the MBO diffusion generated motion is introduced that effectively finds local minimizers of this energy by iterating two steps until convergence. In the first step, as in the original method, the current matrix-valued field is evolved by the diffusion equation. In the second step, the field is pointwise reassigned to the closest orthogonal matrix, which can be computed via the singular value decomposition. We extend the Lyapunov function of Esedoglu and Otto to show that the method is non-increasing on iterates and hence, unconditionally stable. We also prove that spatially discretized iterates converge to a stationary solution in a finite number of iterations. The algorithm is implemented using the closest point method and non-uniform fast Fourier transform. We conclude with several numerical experiments on flat tori and closed surfaces, which, unsurprisingly, exhibit classical behavior from the Allen-Cahn and complex Ginzburg Landau equations, but also new phenomena.
\end{abstract}

\maketitle

\section{Introduction} \label{s:intro}
In a variety of settings, it is of interest to find a matrix-valued field which is smooth, perhaps away from a singularity set, and which best describes some observations or satisfies given boundary conditions. 
We have in mind 
(i) the study of polycrystals, where the matrix-valued field describes the local crystal orientation and should allow for crystal defects, both dislocations and grain boundaries \cite{Berkels_2007,Elsey_2013,Elsey_2014};
(ii) directional field synthesis problems arising in geometry processing and computer graphics, which can sometimes be formulated as finding matrix-valued fields of a certain class \cite{Viertel2017,Vaxman_2016}; and 
(iii) inverse problems in image analysis, {\it e.g.}, diffusion tensor MRI or fiber tractography, where it is of interest to estimate a matrix- or orientation-valued function \cite{Batard_2014,Ba_k_2016,Rosman_2014}. 
In this paper, we consider a model problem for these applications, and study the Dirichlet energy for matrix-valued functions on a given closed surface which take values in the class of orthogonal matrices. 

Let $\Omega$ to be a flat torus, or more generally, a closed manifold. Let $O_n \subset  M_n = \mathbb R^{n\times n} $ be the group of orthogonal matrices.
Let $H^1(\Omega,M_n)$ and $H^1(\Omega,O_n) \subset H^1(\Omega,M_n)$ denote the matrix-valued Sobolev spaces. 
We consider the problem
\begin{equation} \label{eq:GL}
 \min_{A \in H^1(\Omega,O_n)} \ E(A), 
\qquad \qquad \textrm{where} \ \  
E(A) := \frac{1}{2} \int_\Omega \| \nabla A \|_F^2 \ dx.  
\end{equation}
Here, $\| \cdot \|_F$ denotes the Frobenius norm, induced by the Frobenius inner product, $\langle A, B \rangle_F = \sum_{i,j} A_{i,j} B_{i,j}$. 
The gradient of $A$ is understood in the sense that $A$ takes values in the Euclidean embedding space, $M_n$, \cite{Rosman_2014} 
and \emph{not} the covariant derivative sense as pursued in \cite{Batard_2014,Ba_k_2016} and from a statistical standpoint in \cite{Fletcher_2007}. 
The Dirichlet energy, $E$, is non-negative and the minimum is attained by any constant function. In this paper, we're generally interested in approximating stationary points of the energy in \eqref{eq:GL} and the related gradient flow. 

Since the constraint $A \in H^1(\Omega,O_n)$ is non-trivial to enforce, as a first attempt,  we might try to add a term to the energy that pointwise penalizes a candidate field by its distance to $O_n$, 
\begin{equation} \label{eq:GL-relax1}
\min_{A \in H^1(\Omega,M_n)} \ E_{1,\varepsilon}(A), 
\qquad \qquad \textrm{where} \ \  
E_{1,\varepsilon}(A) :=  \int_\Omega \frac{1}{2} \| \nabla A \|_F^2  + \frac{1}{2\varepsilon^2} \textrm{dist}^2(O_n, A) \ dx . 
\end{equation}
Here, the metric is induced by the Frobenius norm and satisfies the following Lemma. 
\begin{lem} \label{l:proj}
Let $A \in \mathbb R^{n \times n}$ have a singular value decomposition, $A = U \Sigma V^t$. Then $B^\star =  A (A^t A )^{ - \frac 1 2}  = U V^t$ attains the minimum in 
\begin{equation} \label{e:Frob}
\textrm{dist}^2(O_n, A) = \min_{B \in O_n} \ \| B - A \|_F^2 = \sum_{i=1}^n (\sigma_i(A) - 1 )^2.
\end{equation}
Furthermore, if $A$ is non-singular, the sign of the determinant of $A$ and $B^\star$ are the same. The minimum, $B^\star$ is unique if the singular values of $A$ are distinct. 
\end{lem}
Lemma \ref{l:proj} is well known, but we include its proof  in Section \ref{s:s1proofs} for completeness. 
It follows from Lemma \ref{l:proj}  that the penalty term in $E_{1,\varepsilon}$ can be written 
\[  \frac{1}{ 2 \varepsilon^2} \sum_{i=1}^n \left(\sigma_i(A) -1 \right)^2. \]
In Theorem \ref{t:GradFlow}, we show that gradient of the energy in \eqref{eq:GL-relax1} exists at each nonsingular matrix and that the gradient flow is given by 
\begin{equation} \label{eq:GD-1}
\partial_t A = -\nabla_A E_{1,\varepsilon}(A) = \Delta A - \varepsilon^{-2}  U( \Sigma - I_n ) V^t, \qquad \textrm{where} \ \ A = U\Sigma V^t \textrm{ is nonsingular}.
\end{equation}

For a gradient flow that is defined for all matrices, we slightly modify the penalty term. Following the Ginzburg-Landau theory \cite{Bethuel_1994}, we introduce the following relaxation 
\begin{equation} \label{eq:GL-relax2}
\min_{A \in H^1(\Omega,M_n)} \ E_{2,\varepsilon} (A), 
\qquad \qquad \textrm{where} \ \  
E_{2,\varepsilon}(A) :=  \int_\Omega \frac{1}{2} \| \nabla A \|_F^2  + \frac{1}{4\varepsilon^2} \| A^t A - I_n \|_F^2 \ dx . 
\end{equation}
It is not difficult to see that the   penalty term in $E_{2,\varepsilon}$ is equivalent to the sum of a ``double well'' potential, $W\colon \mathbb R \to \mathbb R$, evaluated at the singular values of $A$,  {\it i.e.},
\[   \frac{1}{4\varepsilon^2}  \| A^t A - I_n \|_F^2 =  \frac{1}{\varepsilon^2}  \sum_{i=1}^n W\left(\sigma_i(A) \right), 
\qquad \qquad \textrm{where} \ \ 
W(x) = \frac{1}{4} \left(x^2 - 1 \right)^2. \]
In Theorem  \ref{t:GradFlow}, we show that the gradient flow for the energy in  
\eqref{eq:GL-relax2} is given by 
\begin{equation} \label{eq:GD-2}
\partial_t A = -\nabla_A E_{2,\varepsilon}(A) = \Delta A - \varepsilon^{-2}  U( \Sigma^2 - I_n) \Sigma  V^t, \qquad \textrm{where} \ \ A = U\Sigma V^t.
\end{equation}
An elementary calculation allows us to rewrite the second term on the right hand side of \eqref{eq:GD-2} to obtain
\begin{equation} \label{eq:GD-3}
\partial_t A = -\nabla_A E_{2,\varepsilon}(A) = \Delta A - \varepsilon^{-2} A (A^tA-I_n).
\end{equation}

\begin{rem} \label{r:Simp} For $n=1$, $O_n = \{ \pm 1\}$ and \eqref{eq:GD-3} is the classical Allen-Cahn equation. 
\end{rem}

To consider what the behavior  in higher dimensions, we recall that $O(n) = SO(n) \cup SO^-(n)$ is the disconnected union of  the special orthogonal group and the set of matrices with determinant equal to $-1$, 
$$
SO(n) = \{ A \in O_n \colon \textrm{det}(A) = 1 \}  
\quad \textrm{and} \quad  
SO^-(n) =  \{ A \in O_n \colon \textrm{det}(A) = -1 \}.
$$ 
In the relaxed energy given in \eqref{eq:GL-relax2}, the first term can be interpreted as a smoothing term  while the second term penalizes when the matrix-valued field is not an orthogonal matrix. Consider a field that takes values in $SO(n)$ on one set and values in $SO^-(n)$ in the complement. Intuitively, the small parameter  $\varepsilon$ should govern the length scale for which the solution smoothly transitions between values in $SO(n)$  and $SO^-(n)$. 

Indeed, consider the solution to the Cauchy problem to \eqref{eq:GD-3} with  initial condition $A(0,x) = A_0(x)$. We assume $A_0(x) \equiv A_+ \in SO(n)$ for $x$ on a subset of $\Omega$ and $A_0(x) = A_- \in SO^-(n)$ on the complement. Let $\Gamma(t)$ denote the zero level set of $\textrm{det} \left(A(t,x) \right)$ in \eqref{eq:GD-3}, which will serve to mark the interface between  $SO(n)$  and $SO^-(n)$ and let $d(x,t)$ be the signed distance function to $\Gamma(t)$. 
Denote $z= \varepsilon^{-\beta}d(x,t)$ where $\varepsilon^{\beta}$ is the thickness of the interface, and take the asymptotic expansion of \eqref{eq:GD-3} as follows:
\[\partial_t A_{\varepsilon}+\varepsilon^{-1} d_t \partial_z A_{\varepsilon} = \varepsilon^{-2\beta} \partial_{zz}A_{\varepsilon} + \varepsilon^{-\beta} \Delta d \partial_z A_{\varepsilon}  - \varepsilon^{-2} A_{\varepsilon} \left( A_{\varepsilon}^t A_{\varepsilon} - I_n \right),  
\qquad  \text{where} \ \ \ A_{\varepsilon} =\sum_{k=0}^{\infty}\varepsilon^{k} A_k. \]
To get a nontrivial solution, we require $\beta = 1$ and obtain the leading order equation at $O(\varepsilon^{-2})$,
\begin{align} \label{eq:leadingorder}
\partial_{zz}A_0 = A_0(A_0^tA_0-I_n).
\end{align}
It follows that the width of the interface between the sets where the determinant takes values $+1$ and $-1$ is $O(\varepsilon)$.
At  order $O(\varepsilon^{-1})$, we obtain
\begin{align} \label{eq:secondorder}
d_t \partial_z A_0 = \Delta d \partial_z A_0  +\partial_{zz}A_1 - (A_1A_0^tA_0+A_0A_1^tA_0+A_0A_0^tA_1-A_1)
\end{align}
where the last term comes from the direct expansion of $A_{\varepsilon}( A_{\varepsilon}^t A_{\varepsilon} - I_n )$.
Taking the Frobenius inner product with $\partial_z A_0$ on both sides of  \eqref{eq:secondorder} yields:
\begin{align*} 
d_t \langle \partial_z A_0 , &  \partial_z A_0\rangle_F -\Delta d \langle \partial_z A_0, \partial_z A_0 \rangle_F    \nonumber  \\
= & \langle \partial_{zz} A_1, \partial_z A_0 \rangle_F  - \langle A_1 A_0^tA_0+A_0A_1^tA_0 +A_0A_0^tA_1 -A_1, \partial_z A_0 \rangle_F  \\
=&  \langle \partial_{zz} A_1, \partial_z A_0 \rangle_F  - \langle A_1,   \partial_z A_0 A_0^tA_0+A_0\partial_z A_0^tA_0 +A_0A_0^t\partial_zA_0 - \partial_zA_0 \rangle _F  \nonumber
\end{align*}
We integrate the above equation with respect to $z$ from $-\infty$ to $\infty$. Using integration by parts, we can rewrite the first term on the right hand side as 
\begin{align*}
\int_{-\infty}^{\infty}\langle \partial_{zz} A_1 ,\partial_z A_0 \rangle_F dz = \int_{-\infty}^{\infty} \langle A_1 ,\partial_{zzz} A_0 \rangle_F dz
\end{align*}
where the boundary terms vanish because of the assumption that $A_0$ is constant away from $\Gamma(t)$.
We obtain 
\begin{align} \label{eq:secondorder4}
(d_t-\Delta d )\int_{-\infty}^{\infty}\langle  \partial_z A_0, \partial_z A_0 \rangle_F dz &=   \int_{-\infty}^{\infty} \left\langle A_1, \partial_{z}\left(\partial_{zz} A_0- (A_0A_0^t A_0-A_0)\right) \right\rangle_F dz . 
\end{align}
Using \eqref{eq:leadingorder}, the right hand side of \eqref{eq:secondorder4} vanishes and we have
\[ d_t = \Delta d, \]
unless $\|\partial_z A_0\|_F = 0 $, which is impossible since we assume $A_0 \in SO(n)$ on  one side of $\Gamma(t)$ and $A_0 \in SO^-(n)$ on the other.  It follows that the interface evolves according to mean curvature flow if $A_0$ is constant away from the interface. We will demonstrate through numerical experiments that when this assumption does not hold, the interface $\Gamma(t)$ evolves according to a different governing equation (see the example in Figure~\ref{fig:flat_torus_dynamic_winding_1}).

For $n=2$, if the initial condition takes values in $SO(2) \cong \mathbb S^1$, then \eqref{eq:GD-3} is equivalent to the complex Ginzburg-Landau energy as the following Lemma shows. 

\begin{lem} \label{l:n=2Equiv}
Let $n=2$ and consider the Cauchy initial value problem for \eqref{eq:GD-3} with  $A(0,x) = A_0(x)\in SO(2)$.  
Write the solution as $A(t,x)=[ \vec{u}_1(t,x),\vec{u}_2(t,x)]$ where $\vec{u}_1(t,x)$ and $\vec{u}_2(t,x)$ are two column vectors. Then $\vec{u}_1(t,x)$ satisfies the complex Ginzburg-Landau equation  
\[  \partial_t \vec{u}_1(t,x) = \Delta \vec{u}_1(t,x) - \varepsilon^{-2} \vec{u}_1( | \vec{u}_1 |^2-1), 
\qquad \qquad  x\in \Omega, \ t \geq 0. \]
\end{lem}
A proof of Lemma \ref{l:n=2Equiv} is given in Section \ref{s:s1proofs}.

\begin{algorithm}[t!]
\DontPrintSemicolon
 \KwIn{Let $\Omega$ be a closed surface, $\tau > 0$, and $A_0 \in H^1(\Omega,O_n)$.}
 \KwOut{ A matrix-valued function $A_n \in H^1(\Omega,O_n)$ that approximately minimizes   \eqref{eq:GL}.}
 Set $s=1$\;
 \While{not converged}{
{\bf 1.  Diffusion Step.} Solve the initial value problem for the diffusion equation  until time $\tau$ with initial value given by $A_{s-1}(x)$:
\begin{align*}
&\partial_t A(t,x) = \Delta A (t,x) \\
&A(0,x) = A_{s-1}(x).
\end{align*}
Let $\tilde A(x) = A(\tau,x)$\;
{\bf 2. Projection Step.} If $\tilde A(x) = U(x) \Sigma(x) V^t(x)$, set $A_s(x) = U(x) V^t(x)$\;
Set $s = s+1$\;
 }
\caption{A diffusion generated motion algorithm for approximating minimizers of the energy in  \eqref{eq:GL}. } 
\label{a:MBO}
\end{algorithm}

\subsection{Results.} As in \eqref{eq:GL}, the minimum in \eqref{eq:GL-relax1} and \eqref{eq:GL-relax2}  are  attained by any constant orthogonal matrix-valued function, so we are interested in  computing non-trivial stationary solutions. In this paper, we introduce an energy splitting method for finding local minima of \eqref{eq:GL-relax1} and \eqref{eq:GL-relax2}, which can also be interpreted as a split-step method for the evolutions in  \eqref{eq:GD-1} and \eqref{eq:GD-2}. The method, which we refer to as a generalized MBO method,  is given in Algorithm \ref{a:MBO}. The first step of  Algorithm \ref{a:MBO} is the gradient flow until a time $\tau$  of the first term in either \eqref{eq:GL-relax1} or \eqref{eq:GL-relax2}, which is equivalent to heat diffusion of the matrix-valued field. The second step is   gradient flow until infinite time of the second term in either \eqref{eq:GL-relax1} or \eqref{eq:GL-relax2}.  The second step can also be viewed as the point-wise projection onto $O_n$, as Lemma \ref{l:proj} shows.

The generalized MBO method (Algorithm \ref{a:MBO}) reduces to the original one for $n=1$ \cite{merriman1992diffusion,MBO1993,merriman1994motion} and to the generalized method in \cite{Ruuth_2001} if the initial condition is restricted to  $SO(2) \subset O(2)$; see Lemma \ref{l:n=2Equiv}. We'll review these connections and the literature on computational methods for the MBO method in Section \ref{s:PW}.

In Section \ref{s:s1proofs}, we prove Lemma \ref{l:proj} and verify that  \eqref{eq:GD-1} and \eqref{eq:GD-2} are the gradient flows of  \eqref{eq:GL-relax1} and \eqref{eq:GL-relax2}, respectively.  We also prove a maximum principle for the matrix heat equation (see Proposition \ref{p:MaxPrinc}). 
In Section \ref{s:stab}, we extend the Lyapunov function of Esedoglu and Otto \cite{esedoglu2015threshold} to the generalized MBO method (Algorithm \ref{a:MBO}) considered here.  This is equivalent to proving the stability of  Algorithm \ref{a:MBO}. In Section \ref{s:stab}, we also prove convergence of a discretized problem, which we believe is a new result, even for $n=1$.  
In Section \ref{s:Imp}, we discuss an implementation of the method on closed surfaces based on the non-uniform FFT. 
In Section \ref{s:CompExp}, we present a variety of computational experiments. Since, as described above, \eqref{eq:GD-3} can be seen to reduce to the Allen-Cahn and complex Ginzburg-Landau equations in special cases, many numerical experiments exhibit behavior of either or a combination of these two equations. For some initial conditions, new phenomena and stationary states are also observed. In particular, a non-trivial field index can cause an interface that would be stationary under the Allen-Cahn equation to be non-stationary. We also show experiments where the volume of the set $\{x \in \Omega \colon \det A(x) > 0 \}$ is constrained. 
We conclude in Section \ref{s:disc} with a discussion and future directions.

\section{Previous work} \label{s:PW}
In 1992, Merriman, Bence, and Osher (MBO) \cite{merriman1992diffusion,MBO1993,merriman1994motion}
developed a threshold dynamics method for the motion of the interface driven
by the mean curvature. 
To be more precise, let $D \subset \mathbb{R}^n$ be a domain where its boundary
$\Gamma= \partial D$ is to be evolved via motion by mean curvature.
The MBO method is an iterative method, and at each time step, generates a new interface, $\Gamma_{\text{new}}$ (or equivalently, $D_{\text{new}}$) via the following two steps: \\


\noindent {\bf Step 1.} Solve the Cauchy initial value problem for the heat diffusion equation until time $t = \tau$, 
\begin{align*}
& u_t = \Delta u , \\
& u(t=0,\cdot) = \chi_{D},
\end{align*}
where $\chi_D$ is the indicator function of domain $D$. Let $\tilde u(x) = u(\tau,x)$. \\

\noindent  {\bf Step 2.} Obtain a new domain $D_{\text{new}}$ with boundary $\Gamma_{\text{new}} = \partial D_{\text{new}}$  by
\begin{align*}
D_{\text{new}} = \left\{ x\colon \tilde u (x) \geq \frac{1}{2} \right\}.
\end{align*} 

The MBO method has been shown to converge to the continuous motion by mean curvature 
\cite{barles1995simple,chambolle2006convergence,evans1993convergence,swartz2017convergence}. 
Esedoglu and Otto  generalize this type of method to multiphase flow with general mobility \cite{esedoglu2015threshold}. The method has attracted much attention due to its simplicity and unconditional stability. It has  subsequently been extended to deal with many other applications. These applications include the multi-phase problems with arbitrary surface tensions \cite{esedoglu2015threshold}, 
the problem of area or volume preserving interface motion \cite{ruuth2003simple}, 
image processing \cite{wang2016efficient,esedog2006threshold,merkurjev2013mbo}, 
problems of  anisotropic interface motions \cite{merriman2000convolution,ruuth2001convolution,bonnetier2010consistency,elsey2016threshold}, 
the wetting problem on solid surfaces \cite{xu2016efficient}, 
generating quadrilateral meshes \cite{Viertel2017}, 
graph partitioning and data clustering  \cite{Gennip2013}, 
and auction dynamics \cite{jacobsauction}. 
Various algorithms and rigorous error analysis have been introduced to refine and extend the original MBO method and related methods for the aforementioned problems; see, {\it e.g.}, \cite{deckelnick2005computation,esedoglu2008threshold,merriman1994motion,
ilmanen1998lectures,ishii2005optimal,mascarenhas1992diffusion,
ruuth1998diffusion,ruuth1998efficient}.
Adaptive methods have also been used to accelerate this type of method \cite{jiang2016nufft} based on non-uniform fast Fourier transform (NUFFT) \cite{nufft2,nufft6}. Laux et al. \cite{laux2016convergence,laux2016convergence2} rigorously proved the convergence of the method proposed in \cite{esedoglu2015threshold}.  

In Section \ref{ss:closedsurface}, we'll briefly review numerical methods for surface diffusion. 

\section{Properties of gradient flows for the relaxed energies} \label{s:s1proofs}
We first  prove  Lemma \ref{l:proj} and then discuss the gradient flows of $E_{1,\varepsilon}$ and $E_{2,\varepsilon}$ in  \eqref{eq:GL-relax1} and \eqref{eq:GL-relax2}. In Section \ref{s:MaxPrinc}, we also give a maximum norm principle for the matrix diffusion equation.

\begin{proof}[Proof of Lemma \ref{l:proj}]
Introducing the dual variable $\Lambda \in \mathbb S^n$, we form the Lagrangian 
\[ L(B; \Lambda ) = \| B - A \|_F^2 + \langle \Lambda, B^tB - I_n \rangle_F. \]
Setting the derivative with respect to $B$ to zero yields 
\[ -2A + 2B + 2 B \Lambda = 0 \quad \implies \quad B(\Lambda + I_n) =  A . \]
We choose $\Lambda$ so that the constraint $B^t B = I_n$ is satisfied. We compute 
\[ A^t A = (\Lambda + I_n) B^t B  (\Lambda + I_n) =  (\Lambda + I_n)^2 
\quad \implies \quad 
\Lambda + I_n  = (A^t A )^{\frac 1 2}, \] 
to obtain $B^\star = A (\Lambda + I_n)^{-1}  = A (A^t A )^{ - \frac 1 2}$. If   $A = U \Sigma V^t$, then we have that 
\[ B^\star =  A (A^t A )^{ - \frac 1 2} = U \Sigma V^t ( V \Sigma^{2} V^t )^{-\frac 1 2} = U V^t . \]
Finally, we compute 
$  \| B^\star - A \|_F^2 = \| U (\Sigma - I_n) V \|_F^2 = \textrm{tr}( \Sigma - I_n)^2 = \sum_{i=1}^n (\sigma_i(A) - 1)^2 $.
\end{proof}

\begin{thm} \label{t:GradFlow} The gradient  flow for  $E_{1,\varepsilon}$ in  \eqref{eq:GL-relax1} is given in \eqref{eq:GD-1} and defined provided $A$ is non-singular almost everywhere. The gradient flow for  $E_{2,\varepsilon}$ in  \eqref{eq:GL-relax2} is given in \eqref{eq:GD-2}. 
\end{thm}
\begin{proof}
We consider the function $W_1 \colon \mathbb R^n \to \mathbb R$ defined by $W_1(x) = \frac{1}{2} \sum_{i=1}^n (|x| - 1)^2$, which is absolutely symmetric in the sense that it is invariant under permutations and sign changes of its arguments. 
For $x\succ 0$, $W_1(x)$ is differentiable and $\nabla_x  W_1(x) = x - 1$.  By  \cite[Corollary 7.4 ]{Lewis_2005}, the function $W_1 \circ \sigma (A)$ is Fr\'echet differentiable for all non-singular $A$ and the derivative is given by 
$$
\nabla_A W_1\circ \sigma (A) = U ( \Sigma - I_n) V^t, 
\qquad \qquad \textrm{where} \ \ 
A = U \Sigma V^t  \ \ \textrm{is nonsingular}. 
$$
The gradient of the first term of $E_{1,\varepsilon}$ in  \eqref{eq:GL-relax1}  is computed as usual, so this justifies  \eqref{eq:GD-1}. 

We consider the function $W_2 \colon \mathbb R^n \to \mathbb R$ defined by $W_2(x) = \frac{1}{4} \sum_{i=1}^n (x^2 - 1)^2$, which is absolutely symmetric. 
For $x\in \mathbb R^n$, $W_2(x)$ is differentiable and $\nabla_x  W_2(x) = x^3 - x $.  By  \cite[Corollary 7.4 ]{Lewis_2005}, the function $W_2 \circ \sigma (A)$ is Fr\'echet differentiable for all $A \in \mathbb R^{n \times n}$ and the derivative is given by 
$$
\nabla_A W_2\circ \sigma (A) = U ( \Sigma^3 - \Sigma) V^t, 
\qquad \qquad \textrm{where} \ \ 
A = U \Sigma V^t.  
$$
This justifies \eqref{eq:GD-2}. 
\end{proof}

\begin{proof}[Proof of Lemma \ref{l:n=2Equiv}.]
Write the solution to \eqref{eq:GD-3} as $A(t,x)=[\vec{u}_1(t,x),\vec{u}_2(t,x)]$ where $\vec{u}_1(t,x)$ and $\vec{u}_2(t,x)$ are two column vectors. Equation \eqref{eq:GD-3} can be written as the coupled system for the column vectors, 
\begin{subequations} \label{eq:GD-4}
\begin{align}
& \partial_t \vec{u}_1(t,x) = \Delta \vec{u}_1(t,x)-\varepsilon^{-2}(\vec{u}_1(\|\vec{u}_1\|^2-1)+\vec{u}_2(\vec{u}_1\cdot\vec{u}_2)) && x\in \Omega, \ t \geq 0,    \label{eq:GD-4-1}  \\
& \partial_t \vec{u}_2(t,x) = \Delta \vec{u}_2(t,x)-\varepsilon^{-2}(\vec{u}_2(\|\vec{u}_2\|^2-1)+\vec{u}_1(\vec{u}_1\cdot\vec{u}_2)) && x\in \Omega, \ t \geq 0. \label{eq:GD-4-2}
\end{align}
\end{subequations}
Taking the inner product of $\vec{u}_2$ with \eqref{eq:GD-4-1} and $\vec{u}_1$ with \eqref{eq:GD-4-2} and adding the results gives
\begin{align}\label{eq:SO(2)}
\partial_t( \vec{u}_1\cdot \vec{u}_2) = \Delta \vec{u}_1\cdot \vec{u}_2+\Delta \vec{u}_2\cdot \vec{u}_1-\varepsilon^{-2}((\vec{u}_1\cdot\vec{u}_2)(\|\vec{u}_1\|^2+\|\vec{u}_2\|^2-2)+2(\vec{u}_1\cdot\vec{u}_2)^2).
\end{align}

Note that $A(0,x)\in SO(2)$ implies $\vec{u}_1(0,x)\cdot\vec{u}_2(0,x)=0$, $\|u_1(0,x)\|^2=1$, $\|u_2(0,x)\|^2=1$, $u_{11}(0,x) = u_{22}(0,x)$, and $u_{12}(0,x)=-u_{21}(0,x)$ if we denote $u_1(t,x) = (u_{11}(t,x),u_{12}(t,x))^t$ and $u_2(t,x) = (u_{21}(t,x),u_{22}(t,x))^t$. Then, we have 
\begin{align*}
\Delta \vec{u}_1(0,x) & \cdot \vec{u}_2(0,x)+\Delta \vec{u}_2(0,x)\cdot \vec{u}_1(0,x) \\
=&  \Delta u_{11}(0,x)u_{21}(0,x)+ \Delta u_{12}(0,x)u_{22}(0,x)+ \Delta u_{21}(0,x)u_{11}(0,x) +\Delta u_{22}(0,x)u_{12}(0,x) \\
=& -\Delta  u_{11}(0,x) u_{12}(0,x)+ \Delta u_{12}(0,x)u_{11}(0,x) -\Delta u_{12}(0,x)u_{11}(0,x) +\Delta u_{11}(0,x)u_{12}(0,x)  \\
=&0
\end{align*}
which implies $\partial_t( \vec{u}_1\cdot \vec{u}_2)  = 0 $ at $t=0$, $\forall x\in \Omega$. Combining with $\vec{u}_1(0,x)\cdot\vec{u}_2(0,x)=0$ indicates that $\vec{u}_1\cdot \vec{u}_2 = 0$, $\forall t\geq 0 $, $\forall x\in \Omega$. Then, \eqref{eq:GD-4-1} gives us the classical complex Ginzburg-Landau equation as in the statement of Lemma \ref{l:n=2Equiv}.
\end{proof}

\subsection{Maximum norm principle for the matrix-valued heat equation} \label{s:MaxPrinc}

The following is a maximum principle for the solution to the matrix-valued heat equation. 
The proof can also be derived from \cite[Proposition 10.3]{kresin2012maximum}, but we include a simpler proof. 

\begin{prop} \label{p:MaxPrinc}
Let $A \colon \Omega \times [0,\infty) \rightarrow \mathbb{R}^{n\times n}$ satisfy the Cauchy initial value problem, 
\begin{align*}
&\partial_t  A = \Delta A \\
&A(0,x) = A_0.
\end{align*}
Assume that  $\| A_0(x) \|_F = 1$ for every $x\in \Omega$. 
Then  $ \| A(t,x) \|_F \leq 1$ for all $t\geq 0$ and $x \in \Omega$.  This implies that $\| A(t,x) \|_2 \leq 1$ and $ A(t,\cdot) \in \mathcal{K}_n$ for $t \geq 0$. Furthermore, if $A_0(x) \in O(n)$, then 
\[ | \det A(t,x) |  \leq 1, \qquad \qquad \textrm{for all} \ x\in \Omega, \ t\geq 0. \] 

\end{prop}
\begin{proof}
First consider $u\colon \Omega \times [0,\infty)\rightarrow \mathbb{R}^n$ satisfying
\begin{align*}
&u_t = \Delta u \\
&u(0,x) = u_0,
\end{align*}
with $|u_0|_2= 1$. Taking the vector inner product with $u$ on both sides of the equation gives  
\begin{align*}
&\partial_t |u|_2^2  = \Delta |u|_2^2 - \|\nabla u\|_F^2, \\
&|u_0(x) |_2^2 = 1.
\end{align*}
Here $\nabla u$ is the Jacobian of $u$. 

We prove the following claim: \\
\noindent {\it Claim.} 
If $u\colon \Omega \times [0,\infty)\rightarrow \mathbb{R}^n$ satisfies 
\begin{align*}
&\partial_t |u|_2^2  \leq \Delta |u|_2^2 - \|\nabla u\|_F^2, \\
&|u_0(x) |_2^2 = 1.
\end{align*}
then $|u(t,x)|_2\leq 1$ for all $t\geq 0$ and $x \in \Omega$. That is the maximum value must be attained at $t=0$.

\medskip

We divide the proof into two steps as follows.\\

\medskip

{\it Step 1: Assume $\partial_t |u|_2^2  < \Delta |u|_2^2 - \|\nabla u\|_F^2$.} We prove the result by contradiction. 
Assume $|u(t,x)|_2^2$ attains a maximum value at $(x^*,t^*)$ with $t^*>0$. 
On one hand, we have $\partial_t |u(x^*,t^*)|_2^2 \geq 0$. (Otherwise, there exists $\delta > 0$ and $t' \in (t-\delta,t)$ such that $|u(x^*,t')|_2^2 > |u(x^*,t^*)|_2^2$.)
On the other hand, we have $\Delta |u (x^*,t^*)|_2^2 \leq 0$ and $-|\nabla u|_2^2 \leq 0$. Then, we have $\partial_t |u|_2^2 \geq 0 \geq  \Delta |u|_2^2 - \|\nabla u\|_F^2$ at $(t^*,x^*)$ which is contradictory to the assumption.\\

\medskip

{\it Step 2: Assume $\partial_t |u|_2^2  \leq \Delta |u|_2^2 - \|\nabla u\|_F^2$.}
Define $\tilde{u} \colon \Omega \times [0,\infty)\rightarrow \mathbb{R}^{n+1}$  by 
$$\tilde{u}(t,x) = \left( u(t,x) ,\varepsilon^{1/2} (1+t)^{-1/4} \right),$$ 
for any $\varepsilon>0$. 
Then, we check that 
$\partial_t |\tilde{u}|_2^2 = \partial_t |u|_2^2-\frac{1}{2}\varepsilon (1+t)^{-3/2}$, 
$\Delta |\tilde{u}|_2^2=\Delta |u|_2^2$, and $\|\nabla \tilde{u}\|_F^2=\|\nabla u\|_F^2$. Hence, we have $\partial_t |\tilde{u}|_2^2  < \Delta |\tilde{u}|_2^2 - \|\nabla \tilde{u}\|_F^2$.  According to Step 1, we have that $|\tilde{u}|_2^2$ attains the maximum value at $t=0$. 
Then, we have for all $x \in \Omega$ and $t\geq 0$, 
\begin{equation}
|u(t,x)|_2^2 < | \tilde{u}(t,x) |_2^2 
\leq \sup\limits_{\substack{x \in \Omega\\ t \in [0,\infty)} } |\tilde{u}(t,x) |_2^2 
= \sup\limits_{x \in \Omega} |\tilde{u}(0,x) |_2^2 
=  \sup\limits_{x \in \Omega}  |u(0,x)|_2^2+\varepsilon 
= 1+\varepsilon. 
\end{equation}
Taking $\varepsilon \searrow 0$, we have $|u(t,x)|_2^2 \leq 1$ for all $t\geq 0$ and 
$x \in \Omega$. 

\medskip

As for the proposition statement, from the claim, by unfolding the entries of $A$ into a vector, we have 
$ \| A(t,x) \|_F \leq 1$ for all $x \in \Omega$ and $t\geq 0$. 

For $A_0(x) \in O(n)$, we compute
\begin{align*} 
| \det A(t,x) |^{\frac 1 n} &= \left( \prod_{i=1}^n \sigma_i\left(A(t,x) \right) \right)^{\frac 1 n} 
\leq  \frac{1}{n} \sum_{i=1}^n  \sigma_i\left(A(t,x) \right) \\
&\leq \frac{1}{\sqrt{n}} \left( \sum_{i=1}^n  \sigma_i^2\left(A(t,x) \right) \right)^{\frac{1}{2}} 
= \frac{1}{\sqrt{n}} \| A(t,x) \|_F 
\leq 1, 
\end{align*} 
where we used the AM-GM inequality and the inequality $|x|_1 \leq \sqrt{n} | x |_2$ for $x \in \mathbb R^n$. 
\end{proof}

\section{Stability of the generalized MBO diffusion generated motion} \label{s:stab}
Following  \cite{esedoglu2015threshold}, a natural candidate for a Lyapunov function for the generalized MBO diffusion generated motion (Algorithm \ref{a:MBO}) is the functional $E^\tau \colon H^1(\Omega, M_n) \to \mathbb R$, given by  
\begin{equation} \label{eq:Lyapunov} 
E^\tau(A) :=  \frac{1}{\tau} \int_\Omega n - \langle   A, e^{\Delta \tau} A \rangle_F  \ dx. 
\end{equation}
Here we use the notation $e^{\tau \Delta} A$ to denote the solution to the heat equation at time $\tau$ with initial condition at time $t=0$ given by $A = A(x)$. 

Denoting the spectral norm by $\| A \|_2 = \sigma_{\max}(A)$, the convex hull of $O_n$ is 
\[ K_n = \textrm{conv} \ O_n = \{ A \in M_n \colon \| A \|_2 \leq 1 \}; \]  
see \cite{Saunderson_2015}. 

\begin{lem} \label{l:Jprops}
The functional $E^\tau$ defined in \eqref{eq:Lyapunov} has the following elementary properties. 
\begin{enumerate}
\item[(i)]  For $A \in L^2(\Omega, O_n)$, $E^\tau(A)= E(A) + O(\tau)$.
\item[(ii)] $E^\tau(A)$ is concave. 
\item[(iii)] We have 
\[ \min_{A \in L^2(\Omega,O_n)} \ E^\tau(A) = \min_{A \in L^2(\Omega, K_n) } \ E^\tau(A).  \]
\item[(iv)] $E^\tau(A)$ is Fr\'echet differentiable with derivative $L^\tau_A \colon L^\infty(\Omega, M_n) \to \mathbb R$  at $A$ in the direction $B$ given by 
\[ L^\tau_A(B) = - \frac{2}{\tau} \int_\Omega \langle  e^{\Delta \tau} A, B \rangle_F  \ dx. \]
\end{enumerate}
\end{lem}
\begin{proof}
(i) Since for any $A \in O_n$, we have $\| A \|_F = \sqrt{n}$, it follows that  for $A \in L^2(\Omega, O_n)$, 
\begin{align*} 
E^\tau(A) &=  \frac{1}{\tau} \int_\Omega n - \langle   A, e^{\Delta \tau} A \rangle_F  \ dx \\
&= \frac{1}{\tau} \int_\Omega n - \langle   A, \left( I + \tau \Delta + O(\tau^2) \right) A \rangle_F  \ dx \\
&= E(A) + O(\tau). 
\end{align*}

(ii) A computation shows that for a symmetric operator $S \colon L^2(\Omega) \to L^2(\Omega)$, $\theta \in [0,1]$, and the quadratic form $f(u) = \langle u, S u\rangle$ we have
\begin{equation} \label{eq:QuadId} 
f\left( \theta u + (1-\theta) v \right)  \ - \ \theta f(u)  \ -  \ (1-\theta) f(v)  \ = \ - \theta (1-\theta)  \langle u - v, S (u - v) \rangle.
\end{equation}
For $A_1,A_2 \in L^2(\Omega, M_n)$ and $\theta \in [0,1]$, writing $A_\theta = \theta A_1 + (1-\theta) A_2$, we then compute 
\begin{align} \label{e:ConvexForm}
& \int_\Omega \left\langle  A_\theta, S A_\theta \right\rangle_F \ dx  
- \theta \int_\Omega \langle A_1 , S A_1 \rangle_F \ dx  
- (1- \theta) \int_\Omega \langle A_2 , S A_2 \rangle_F \ dx  \\
\nonumber
& \qquad  = - \theta (1-\theta) \int_\Omega \langle  A_1 - A_2, S (A_1 - A_2) \rangle_F \ dx. 
\end{align}
For $S = e^{\Delta \tau}$, by the semi-group property, we have that 
$$\int_\Omega \langle  A_1 - A_2, S (A_1 - A_2) \rangle_F \ dx = \int_\Omega \| e^{\Delta \tau / 2} (A_1 - A_2) \|_F^2  \ dx \geq 0. $$ 
This shows that $E^\tau $ is concave.

(iii) The problem on the right is the minimization of a concave  functional on a convex set, $\L^\infty(\Omega,K_n)$, so the minimum is attained on the boundary, $L^\infty(\Omega,O_n)$. 

(iv) The Fr\'echet derivative is computed 
\begin{align*}
L^\tau_A(B)  =  \left\langle \nabla_A E^{\tau}(A), B \right\rangle 
=  - \frac{1}{\tau} \int_\Omega \langle e^{\Delta \tau}  A, B \rangle_F + \langle e^{\Delta \tau} B , A \rangle_F  \ dx 
= - \frac{2}{\tau} \int_\Omega \langle  e^{\Delta \tau} A, B \rangle_F  \ dx, 
\end{align*}
where we have used that $e^{\Delta \tau}$ is a symmetric operator on a closed manifold. 
\end{proof}

Thus, $E^\tau(A)$ is an approximation of $E(A)$ for small $\tau$, with linearization, $L_A^\tau$. 
We consider the linear optimization problem,
\begin{equation}
\label{e:LinProb}
\min_{B \in  L^\infty(\Omega, K_n)} \ L^\tau_A(B). 
\end{equation}

\begin{lem}\label{l:LinProbSol}
If $e^{\Delta \tau} A = U\Sigma V^t$, the solution to \eqref{e:LinProb} is attained by the  function $B^\star = U V^t \in  L^\infty(\Omega, O_n)$ with value 
$L_A^\tau(B^\star) = - \frac{2}{\tau} \sum_{i=1}^n \int_\Omega \sigma_i\left( e^{\Delta \tau} A \right) \ dx $. 
\end{lem}
\begin{proof} 
Since $L^\infty(\Omega, K_n)$ is a compact and convex set, and $L_A^\tau$ is a linear function,  the minimum is in $L^\infty(\Omega, O_n)$. The result then follows from Lemma \ref{l:proj}. 
\end{proof}

The sequential linear programming approach to minimizing $E^\tau(A)$ in \eqref{eq:Lyapunov} subject to $A \in L^\infty(\Omega, O_n)$
is to consider 
a  sequence of functions $\{A_s\}_{s=0}^\infty$ which satisfies
\begin{equation}
\label{eq:SLP}
A_{s+1} = \arg \min_{A \in  L^\infty(\Omega, K_n)} \ L^\tau_{A_s} ( A), \qquad A_0 \in L^\infty(\Omega, O_n) \textrm{ given}.
\end{equation}
The solution to the optimization problem in \eqref{eq:SLP} is given in Lemma \ref{l:LinProbSol}. The elements of the resulting sequence $\{A_s\}_{s=0}^\infty$ are elements of $L^\infty(\Omega, O_n)$ for all $s\geq 0$.  
The following proposition shows the in \eqref{eq:SLP} is the same as the sequence in Algorithm \ref{a:MBO} and that $E^\tau$ is a  \emph{Lyapunov function} for this sequence.  This shows that the generalized MBO algorithm (Algorithm \ref{a:MBO}) is unconditionally stable. 

\begin{prop} \label{prop:Lyapunov}
The iterations defined in  Algorithm \ref{a:MBO} and \eqref{eq:SLP} are equivalent. The functional $E^\tau$, defined in \eqref{eq:Lyapunov}, is non-increasing on the iterates $\{ A_s \}_{s=1}^\infty$, {\it i.e.}, $E^\tau(A_{s+1}) \leq E^\tau(A_s)$. 
\end{prop}

\begin{proof} The first claim follows from Lemma \ref{l:LinProbSol}. By the  concavity of $E^\tau$ (see Lemma \ref{l:Jprops}) and linearity of $L^\tau_{A_s}$,
$$
E^\tau(A_{s+1})- E^\tau(A_s) \leq  L^\tau_{A_s} (A_{s+1} - A_s) =  L^\tau_{A_s} (A_{s+1}) -  L^\tau_{A_s} (A_s). 
$$
Since $A_s \in L^\infty(\Omega, O_n) \subset L^\infty(\Omega, K_n)$, $L^\tau_{A_s} (A_{s+1}) \leq  L^\tau_{A_s} (A_s)$ which implies $E^\tau(A_{s+1}) \leq E^\tau(A_s)$. 
\end{proof}

\subsection{Convergence of Generalized MBO in a discrete setting} \label{s:ConvProof}
We consider a discrete grid $\tilde \Omega = \{ x_i \}_{i=1}^{|\tilde \Omega |} \subset \Omega$ and a standard finite difference approximation of the Laplacian, $\tilde \Delta$, on $\tilde \Omega$. We first consider the case $n=1$. 
Define the discrete functional $\tilde E^\tau \colon \ell^2(\tilde \Omega) \to \mathbb R$ by 
\[ \tilde E^\tau(f) = \frac{1}{\tau} \sum_{x_i \in \tilde \Omega}  1 - f_i  (e^{\tilde \Delta \tau} f)_i  
\qquad \qquad \qquad \qquad  \qquad \qquad (n=1) \]
and its linearization by 
\[ \tilde L^\tau_f(g) = - \frac{2}{\tau} \sum_{x_i \in \tilde \Omega}   g_i  (e^{\tilde \Delta \tau} f)_i  
\qquad \qquad \qquad \qquad  \qquad \qquad (n=1). \]
\begin{prop} \label{prop:MBOconverges} Let $n=1$. 
Non-stationary iterations in Algorithm \ref{a:MBO} strictly decrease the value of $\tilde E^\tau$ and since the state space is finite, $\{\pm 1\}^{|\tilde \Omega|}$, the algorithm converges in a finite number of iterations. Furthermore, for $m := e^{- \| \tilde \Delta \| \tau}$, each iteration reduces the value of $J$ by at least $2m$,  so the total  number of iterations is less than 
$ \tilde E^\tau (A_0) / 2m$.
\end{prop}

\begin{proof}
Since $\tilde \Delta$ is a bounded operator, $e^{\tilde \Delta \tau}$ is positive definite, satisfying 
\[  \sum_{x_i \in \tilde \Omega}  A_i  (e^{\tilde \Delta \tau} A)_i  \geq m  \| A \|_{\ell^2(\tilde \Omega)}^2 .\]
This gives the strong concavity of the quadratic form, $\tilde E^\tau - |\tilde \Omega | /\tau$, {\it i.e.}, 
\begin{equation}  \label{e:StrConv1} 
\tilde E^\tau (A) -  |\tilde \Omega | /\tau \leq - m \| A \|_{\ell^2(\tilde \Omega)}^2. 
\end{equation} 
We now consider the iterations $\{ A_s \}_s$ of Algorithm \ref{a:MBO}. 
Using \eqref{e:StrConv1} and the  linearity of $\tilde L_{A_s}$, for $A_{s+1} \neq A_s$,
\begin{align*}
\tilde E^\tau(A_{s+1})- \tilde E^\tau(A_s) &\leq  \tilde L_{A_s} (A_{s+1} - A_s) - \frac{m}{2} \| A_{s+1} - A_{s} \|_{\ell^2(\tilde \Omega)}^2 \\
&=  \tilde L_{A_s} (A_{s+1}) -  \tilde L_{A_s} (A_s) - \frac{m}{2} \| A_{s+1} - A_{s} \|_{\ell^2(\tilde \Omega)}^2.
\end{align*}
Since $A_s \in \ell^2(\tilde \Omega)$, $\tilde L_{A_s} (A_{s+1}) \leq  \tilde L_{A_s} (A_s)$ by construction which implies 
\[\tilde E^\tau(A_{s+1}) \leq \tilde E^\tau(A_s) - \frac{m}{2}  \| A_{s+1} - A_s \|_{\ell^2(\tilde \Omega)}^2 < \tilde E^\tau(A_s) \] 
for all non-stationary iterations. The convergence of the algorithm in a finite number of iterations then follows from the fact that $\{\pm 1\}^{|\tilde \Omega|}$,  is a finite (though exponentially large) set. 

Since each non-stationary iteration satisfies $ \| A_{s+1} - A_{s} \|_{\ell^2(\tilde \Omega)}^2 \geq 4$, we have that each iteration decreases the objective by at least $2m$. If the algorithm converges in $N$ iterations, it follows that 
$$
N \cdot 2m \leq \tilde E^\tau(A_0) - \tilde E^\tau(A_N) \leq \tilde E^\tau(A_0).
$$ 
\end{proof}

We now consider the case $n\geq 2$. 
For $A \colon \tilde \Omega \to O_n$, define the discrete functional 
\[ \tilde E^\tau(A) = \frac{1}{\tau} \sum_{x_i \in \tilde \Omega}  1 - \langle A_i,   (e^{\tilde \Delta \tau} A)_i \rangle_F  
\qquad \qquad \qquad \qquad  \qquad \qquad (n\geq 2) \]
and its linearization by 
\[ \tilde L^\tau_A(B) = - \frac{2}{\tau} \sum_{x_i \in \tilde \Omega}   \langle B_i,  (e^{\tilde \Delta \tau} A)_i   \rangle_F
\qquad \qquad \qquad \qquad  \qquad \qquad (n\geq 2). \]

 We will use the following Lemma. 
\begin{lem} \label{l:SOdist}
$\textrm{dist}\left(SO(n), \ SO^-(n) \right) = 2$. 
\end{lem}
\begin{proof}
We first claim that 
\[ 
\min_{\substack{A \in SO(n) \\ B \in SO^{-}(n)}}  \| A - B \|_F^2   =   
\min_{C \in SO^{-}(n) }  \| I_n - C \|_F^2. 
\]
This follows from the facts that $\| A - B \|_F^2 = \| I_n - A^t B \|_F^2$ and  $A^t B \in SO^-(n)$. 
But, for $C \in SO^{-}(n)$, we have 
\[ \| I_n - C \|_F^2  
\ =  \ 2n - 2\textrm{tr}(C) 
\ =  \ 2n - 2 \sum_{i \in [n]} \lambda_i(C). \] 
Since at least two eigenvalues cancel for every $C \in SO^{-}(n)$, we have $\sum_{i \in [n]} \lambda_i(C) \leq n-2$. This implies that 
$\textrm{dist}^2\left(SO(n), \ SO^-(n) \right)  \geq 2n - 2(n-2) = 4$. This distance is attained by $I_n \in SO(n)$ and $\textrm{diag}(-1,1,\cdots) \in SO^-(n)$, so the result follows. 
\end{proof}
\begin{prop} \label{prop:MBOconverges2} Let $n\geq 2$. The non-stationary iterations in Algorithm \ref{a:MBO} strictly decrease the value of $\tilde E^\tau$. 
For a given  initial condition $A_0 \colon  \tilde \Omega \to O_n$, there exists a partition $\tilde \Omega = \tilde \Omega_+ \amalg \tilde \Omega_-$ and an $S \in \mathbb N$ such that for $s \geq S$, 
\[ \det A_s( x_i) = \begin{cases}
+1 & x_i \in \tilde \Omega_+ \\
-1 & x_i \in \tilde \Omega_-
\end{cases}. \] 
\end{prop}
\begin{proof} For $A \colon \tilde \Omega \to O_n$, we obtain 
\[ \sum_{x_i \in \tilde \Omega} \langle A_i, e^{\tilde \Delta \tau} A_i \rangle_F \geq 
m \sum_{x_i \in \tilde \Omega} \| A_i \|_F^2, \]
which, as in \eqref{e:StrConv1}, gives the strong concavity condition  
\[ \tilde E^\tau (A) -  |\tilde \Omega | /\tau \leq - m  \sum_{x_i \in \tilde \Omega} \| A_i \|_F^2. \]
For the iterates $\{A_s\}_s$ of Algorithm \ref{a:MBO}, this implies that 
\[\tilde E^\tau(A_{s+1}) \leq \tilde E^\tau(A_s) - \frac{m}{2}  \sum_{x_i \in \tilde \Omega} \| A_{s+1} - A_s \|_F^2 < \tilde E^\tau(A_s) \] 
for all non-stationary iterations. 

For $n\geq 2$, the state space is infinite, so we cannot immediately guarantee convergence in a finite number of iterations as in Proposition \ref{prop:MBOconverges}. However, if the determinant of at least one vertex switches signs, by Lemma \ref{l:SOdist}, we have 
\[  \sum_{x_i \in \tilde \Omega} \| A_{s+1} - A_s \|_F^2  \geq \textrm{dist}^2 \left(SO(n), SO^-(n) \right) = 4.\] It follows that the sign of the determinant can only change signs at  $ \tilde E^\tau (A_0) / 2 m$ iterations. 
\end{proof}

\begin{rem}
Proposition \ref{prop:MBOconverges}, which applies when $n=1$, is stronger than \ref{prop:MBOconverges2}, which applies when $n\geq2$. For $n=1$, the field is constant away from the interface where the sign changes. For $n\geq 2$, our numerical experiments indicate that the field is not necessarily constant away from the interface where the sign of the determinant changes; see the numerical experiments in Section \ref{s:FT} and especially Figure \ref{fig:flat_torus_dynamic_winding_1}. 
\end{rem}

\section{Computational Methods} \label{s:Imp}
In this section, we describe a numerical implementation of Algorithm \ref{a:MBO} on a flat torus and closed embedded surface. 

\subsection{Implementation on flat torus}\label{ss:flat}
A \emph{flat torus} is a torus with the metric inherited from its representation as the quotient, $\mathbb R^n/\mathbb L$, where $\mathbb L$ is a Bravais lattice. For $n=2$ and $\mathbb L = \mathbb Z^2$, we obtain the square  flat torus, $\mathbb R^2 / \mathbb Z^2$, which is just the Cartesian plane with the identifications $(x,y) \sim (x+1,y)\sim (x,y+1)$.
Thus, we consider the computational domain $\Omega = [-\frac{1}{2},\frac{1}{2}] ^2 $ with edges identified (periodic boundary conditions). 

The diffusion step in Algorithm~\ref{a:MBO} is to solve
\begin{subequations} \label{e:PeriodicHeat}
\begin{align}
&\partial_t A(t,x) = \Delta A (t,x)  && x\in \Omega, \ t\geq 0,\\
&A(0,x) = A_{s-1}(x) && x \in \Omega \\
& A \textrm{ satisfies periodic boundary conditions on } \partial \Omega. 
\end{align}
\end{subequations}
It is well-known that the solution for the diffusion equation for a scalar function on $\mathbb R^d$ at time $t=\tau$ can be expressed as the convolution of the  heat kernel, 
\begin{equation} \label{e:HeatKer}
G_{\tau}^d(x) = ( 4 \pi \tau )^{-d/2} \exp(-\frac{|x|^2}{4 \tau}), 
\end{equation}
 and the initial condition, $A_{s-1}(x)$. For our periodic domain, $\Omega \subset \mathbb R^2$, we denote by $G_{p,\tau}$ the periodic heat kernel, given by 
  \[ G_{p,\tau}(x) = \sum_{\alpha \in \mathbb Z^2 }  G_\tau^2(x- \alpha). \]
 The solution, $\tilde{A}(x) = A(\tau,x)$ to \eqref{e:PeriodicHeat} at time  $t=\tau$ has matrix components given by $\tilde{A}_{ij}=G_{p,\tau} * (A_{s-1})_{ij}$, where $*$ denotes the convolution. 

We denote the Fourier transform and its inverse by $\mathcal{F}$ and $\mathcal{F}^{-1}$, respectively. Using the convolution property that 
$\mathcal{F}(G_{p,\tau}*(A_{s-1})_{ij}) =\mathcal{F}(G_{p,\tau}) \ \mathcal{F}((A_{s-1})_{ij}) $, we can express the solution to \eqref{e:PeriodicHeat} as 
\[ 
\tilde{A}_{ij} = \mathcal{F}^{-1} \left( \ \mathcal{F}(G_{p,\tau}) \ \mathcal{F}( (A_{s-1})_{ij}) \ \right). 
\] 

In our numerical implementation, we discretize $\Omega$ using $1024 \times 1024$ grid points and, due to the periodic boundary condition, can efficiently compute an approximation to the Fourier transform and its inverse using the  fast Fourier transform (FFT) and inverse fast Fourier transform (iFFT). For $\Omega$ a flat torus, the diffusion step in Algorithm~\ref{a:MBO}, is evaluated via 
\begin{align*}
\tilde{A}_{ij} =  \textrm{iFFT} \left( \ \textrm{FFT}(G_{p,\tau}) \ \textrm{FFT}((A_{s-1})_{ij}) \ \right).
\end{align*}

\subsection{Implementation on closed surfaces} \label{ss:closedsurface}
We consider the diffusion step in Algorithm~\ref{a:MBO} in the case when $\Omega$ is a smooth closed surface. Here, diffusion is taken to mean surface diffusion or evolution with respect to the Laplace-Beltrami operator, $\Delta_{\mathcal{S}}$, {\it i.e.}, the solution at time $t=\tau$ to the equation 
\begin{subequations} \label{eq:surface_diffusion}
\begin{align}
& \partial_t A(t,x) = \Delta_{\mathcal{S}} A(t,x) && x\in \Omega, \ t \geq 0,    \\
& A(0,x) = A_{s-1}(x) && x \in \Omega.
\end{align}
\end{subequations}
 For simplicity, we assume the surface is smoothly embedded in $\mathbb{R}^3$. 
In this section, we propose a method based on the closed point method (CPM) \cite{ruuth2008simple} and non-uniform fast Fourier transform (NUFFT) \cite{nufft0} to efficiently solve \eqref{eq:surface_diffusion} with a relatively large time step.

There are several popular approaches for solving problems with surface diffusion.  One can solve the surface diffusion based on a triangulation of the surface using the finite element method. Also, one can impose a smooth coordinate system or parameterization on the surface, express the Laplace-Beltrami operator within these coordinates, and discretize the resulting equations; see, for example, \cite{floater2005surface} and references therein.  An alternative approach for embedded surfaces, is to extend the surface diffusion equation to an equation on all of the embedded space, $\mathbb R^3$. Then, the solution of the embedded equation provides the solution to the original problem when restricted to the surface.  
This embedding method was introduced in  \cite{bertalmio2001variational,bertalmio2003variational} for solving variational problems and the resulting Euler-Lagrange evolution PDEs on surfaces based on level set methods and has since been further developed; see, {\it e.g.}, \cite{greer2006improvement}. Recently, a new embedding procedure, called the \emph{Closest Point Method}, was presented in \cite{ruuth2008simple} by discretizing the surface diffusion equation using a fixed Cartesian grid. A function on the surface is represented on the grid using the value at the closest point on the surface and the PDE of interest is then solved on the Cartesian grid. See \cite{merriman2007diffusion,Macdonald/Brandman/Ruuth:eigen,luke:segment,cbm:lscpm,cbm:icpm} for applications and extensions of the Closest Point Method.

In the Closest Point Method \cite{ruuth2008simple}, for any $x\in \mathbb{R}^3$, we denote $x_{cp} \in \Omega$ as the closest point to $x$ on the surface. Then, we extend the initial condition $A(x_{cp},0)$ from the surface to $\mathbb{R}^3$ to have $\tilde{A}(0,x)$ by constant extension, i.e., $\tilde{A}(0,x)=A(x_{cp},0)$. Then, instead of solving equation \eqref{eq:surface_diffusion}, we solve a free space heat diffusion equation until time $t=\tau$:
\begin{subequations} \label{eq:surface_diffusion_extended}
\begin{align}
& \partial_t \tilde{A}(t,x) = \Delta \tilde{A}(t,x) && x\in \mathbb{R}^3, \ t \geq 0,    \\
& \tilde{A}(0,x) = A_{s-1}(x_{cp}) && x \in \Omega.
\end{align}
\end{subequations}
The solution to this free-space problem, $\tilde{A}(\tau,x)$, has matrix entries 
\begin{equation} \label{e:freeSpace} 
\tilde{A}_{ij}(\tau,x)=G^3_{\tau}*\tilde{A}_{ij}(0,x).
\end{equation}
The projection of $\tilde{A}(\tau,x)$ onto the surface is an approximation to $A(x_{cp},\tau)$.

\begin{rem}
Note that the closest point $x_{cp}$ for some $x$ may not be unique for some surfaces. In this case, we choose one randomly. 
\end{rem}

Since the Gaussian kernel is localized (in a manner dependent on $\tau$), the solution on the surface (i.e. $A(x_{cp},\tau)$) is only strongly affected by points in a small neighborhood of the surface. Hence, we only need to consider the extension from $A$ to $\tilde{A}$ in a relatively narrow band around $\Omega$, rather than the whole space $\mathbb{R}^3$. The following theorem shows the relationship between the time step $\tau$, truncation error $\varepsilon$, and width of the band, $w_b$. 

\begin{thm}\label{thm:truncation_band}
Let $T(x) =  \frac{2x}{\sqrt{\pi}} \exp(-x^2) + (1- \textrm{erf}(x))$, where $\textrm{erf}(x) = \frac{2}{\sqrt{\pi}} \int_0^x \exp(-t^2) dt$. 
Given a time parameter for the heat kernel, $\tau>0$, and $\varepsilon>0$,  for any $w_b$ satisfying 
\begin{align}
T(\frac{w_b}{2\sqrt{\tau}}) \leq \varepsilon, \label{est:w_b}
\end{align}
we have 
\begin{align}
\left|\left(G_{\tau}^3*\tilde{A}_{ij}\right)(x)  \ - \  \iiint_{|x- y|\leq w_b} G_{\tau}^3( x-y) \tilde{A}_{ij}(y) d {y}\right| \leq\varepsilon. \label{eq:HeatKernelTwoPart}
\end{align}
\begin{proof}
Given $\varepsilon$, we write $G_{\tau}^3*\tilde{A}_{ij}$ at any point $x$ on surface as
\begin{align}
\left(G_{\tau}^3*\tilde{A}_{ij}\right)(x) &= \iiint_{|x- y|\leq w_b} G_{\tau}^3( x-y) \tilde{A}_{ij}(y) d {y}+ \iiint_{|x- y|>w_b} G_{\tau}^3(x- y) \tilde{A}_{ij}( y) d  y . \label{eq:HeatKernelTwoPart}
\end{align}
Then we have the following estimate on the second integral on the right hand side of equation \eqref{eq:HeatKernelTwoPart}:
\begin{align*} 
\Big| & \iiint_{|x- y|>w_b}  G_{\tau}^3({x}-{y}) A_{ij}({y}) d {y} \Big| \\
& \leq \left|\iiint_{|x-y|>w_b} G_{\tau}^3(x- y) d y\right| \\
& =  \left| \iiint_{|x|>w_b} G_{\tau}^3({x}) d {x} \right| \\ 
& = \left|\int_{0}^{2\pi} \int_{0}^{\pi} \int_{w_b}^{+\infty}  \frac{1}{(4\pi \tau )^{\frac{3}{2}}} \exp(-\frac{r^2}{4\tau}) r^2 \sin\theta dr d\theta d\phi\right|\\
& \leq  \left|\int_{0}^{2\pi} \int_{0}^{\pi}\left. \left[ \frac{2\tau}{(4\pi \tau )^{\frac{3}{2}}} \exp(-\frac{r^2}{4\tau}) r\right]\right|_{w_b}^{+\infty} \sin\theta d\theta d\phi\right|+\left|\int_{0}^{2\pi} \int_{0}^{\pi} \int_{w_b}^{+\infty} \frac{2\tau}{(4\pi \tau )^{\frac{3}{2}}} \exp(-\frac{r^2}{4\tau}) dr \sin\theta d\theta d\phi\right|\\
& = T(\frac{w_b}{2\sqrt{\tau}}) \leq \varepsilon . 
\end{align*} 
\end{proof}
\end{thm}

In Table \ref{tab:Estimatebw}, we list the values of the band width, $w_b$, necessary in the approximation of the direct evaluation, $G_{\tau}^3*A_{ij}$, for various levels of precision $\varepsilon$ and time steps $\tau$. 
 
\begin{table}[t]
\begin{center}
\caption{Width of the band $w_b$ satisfying $T(\frac{w_b}{2\sqrt{\tau}}) = \varepsilon$ for different $\tau$ and $\varepsilon$. See Theorem \ref{thm:truncation_band}. \label{tab:Estimatebw}
}
\begin{tabular*}{0.6\textwidth}{@{\extracolsep{\fill}} | l|c|c|c|c|  }
\hline
\backslashbox{$\varepsilon$}{$\tau$}
&\makebox[3em]{$10^{-1}$ }&\makebox[3em]{$10^{-2}$}&\makebox[3em]{$10^{-3}$}
&\makebox[3em]{$10^{-4}$}\\
\hline
$10^{-3}$ &    1.796 &   0.5683 &  0.1796 & 0.05683\\ \hline
$10^{-6}$ &   2.474 &  0.7823 &  0.2474 &  0.07823\\ \hline
$10^{-9}$&   
   2.993  & 0.9465 &  0.2993 &  0.09465 \\ \hline
$10^{-12}$ &   3.432 &  1.085 &  0.3432 &  0.1085 \\ \hline
\end{tabular*}
\end{center}
\end{table}

\begin{rem} 	
Table \ref{tab:Estimatebw} indicates that $w_b$ increases as $\tau$ increases. That means if we choose large $\tau$ in the experiments,  $w_b$ is large which introduces large degrees of freedom for the same accuracy $\varepsilon$. However, in our experiments, the choice of $\tau$ can be relatively large at an acceptable degree of freedom achieving single precision accuracy. 
\end{rem}

Analogous to the case of a flat torus, we evaluate the convolution in \eqref{e:freeSpace} using the equality 
\begin{align}\label{e:fundamental_Fourier}
G_\tau^3*\tilde{A}_{ij}=\mathcal{F}^{-1}( \ \mathcal{F}(G_\tau^3) \ \mathcal{F}(\tilde{A}_{ij}) \  ) .
\end{align} 
We numerically evaluate \eqref{e:fundamental_Fourier}  as follows. 
We first select a relatively large box $B$ that contains the surface and then discretize the box with grid size $dx$ to obtain the set of grid points, $\mathcal{G} \subset B$. For each $x_g\in \mathcal{G}$, according to the parametrization of  
the surface, we find the distance  to the surface, $d_g = d(x_g,\Omega)$ and the corresponding closest point $x_g^c = \arg\min_{x \in \Omega} \ d(x,x_g)$ on the surface. Due to Theorem \ref{thm:truncation_band}, to approximately  evaluate \eqref{e:fundamental_Fourier},  we only need the value of $\tilde{A}_{ij}$ in a band of the surface with width $w_b$. Hence, we only keep  $x_g \in \mathcal{G}$ with $d_g<w_b$ denoted by 
\[\mathcal{G}^b = \{ x_g \in \mathcal{G} \colon d(x_g, \Omega) < w_b \}.\]
 To further improve the accuracy of the discretization of the integral in 
 equation \eqref{e:fundamental_Fourier}, for each $x_g \in \mathcal{G}^b $, we generate the $p$-th order scaled Gauss-Chebyshev quadrature points in the box  $x_g + [0,dx]^3$. 
 Let $\mathcal{Q}$ denote this set of these quadrature points. For each $x_q \in \mathcal{Q}$, we find the corresponding closest point, denoted  $x_{q}^c = \arg\min_{x \in \Omega} \ d(x,x_q)$.
 
An overview of our algorithm is  as follows, which we further detail below. 
(1) Initially, we have the value of $A_{ij}(x_{q}^c)$ on the surface. 
(2) We then constantly extend the value of $A_{ij}(x)$ to all  quadrature points in $\mathcal{Q}$ in the band of the surface, which we denote by $\tilde{A}_{ij}(x_q)$ for $x_q \in \mathcal{Q}$. 
(3) Using the value of $\tilde{A}_{ij}$ at $x_q$, we evaluate $\mathcal{F}(\tilde{A}_{ij})$. 
(4) Multiplying by $\mathcal{F}(G_\tau^3)$, we then use the inverse Fourier transform to evaluate the solution at $x_{q}^c$ on the surface. 

To evaluate the Fourier transform, it is too expensive to use the discrete Fourier transform (DFT). Also, because both the quadrature points $x_{q} \in \mathcal{Q}$ and closest points $x_{q}^c$ on the sureface are  non-uniform, it is not possible to simply use the fast Fourier transform (FFT) to evaluate the Fourier transform. In our setting, we need to evaluate the Fourier transform from non-uniform points in the physical domain to uniform points in the spectral domain and the inverse Fourier transform from uniform points in the spectral domain to non-uniform points in the physical domain. These two tasks can be efficiently accomplished using the type-1 non-uniform fast Fourier transform (NUFFT) and type-2 NUFFT, as described in  \cite{nufft0}.
The type-1 NUFFT evaluates sums of the form
\begin{equation}\label{4.1}
f(k)=\frac{1}{N}\sum_{j=1}^{N}c_j e^{\pm i k\cdot x_j},
\end{equation}
for ``targets" $k$ on a regular (equispaced) grid in $\mathbb{R}^d$
($d=3$ for our case), given function values $c_j$ prescribed at non-uniform 
points $x_j$ in physical space.
Here, $N$ denotes the total number of source points. The type-2 NUFFT evaluates sums of the form
\begin{equation}\label{4.2}
  F(x_n)=\sum_{m_1=-M_1}^{M_1-1}\cdots \sum_{m_d=-M_d}^{M_d-1}f(m)
  e^{\pm im \cdot x_n},
\end{equation}
where the ``targets" $ x_n$ are non-uniform points in $\mathbb{R}^d$ and the function $f$ is evaluated on a regular grid in the spectral domain. So, we use type-1 NUFFT to evaluate the Fourier transform and type-2 NUFFT to evaluate the inverse Fourier transform.

We recall the Fourier spectral approximation of the one-dimensional heat kernel for a fixed time has error which is characterized by the following Theorem. 
\begin{thm}  \cite[Theorem 1, Fourier spectral approximation of the one-dimensional heat kernel.]{jiang2016nufft}
\label{thm:approximation_heat_kernel}
Let $\varepsilon < 1/2$ be the prescribed accuracy,  
$\tau> 0$ be the time parameter for the heat kernel, 
and $R>0$ be a spatial radius. 
Define $h=\min\left(\frac{\pi}{R},\frac{\pi}{2\sqrt{\tau|\ln\varepsilon|}}\right)$ and 
$M = \frac{1}{h}\sqrt{\frac{\ln(\sqrt{\pi}\varepsilon/2h\sqrt{\tau})}{\tau}}$.
  Then for all $|x|\leq R$,
  \begin{equation}\label{2.7}
  \left|G_\tau(x)- \frac{h}{2\pi}\sum_{m=-M}^{M-1} e^{-m^2h^2 \tau+imhx}\right| \leq \frac{4\varepsilon}{\sqrt{4\pi\tau}}.
  \end{equation}
\end{thm}

\begin{table}[t!]
\begin{center}
\caption{Number of Fourier modes, $M$, needed to approximate the one-dimensional heat kernel
$G_\tau^1(x)$ for $x\in [-R,R]$ with $R = \pi$; see Theorem \ref{thm:approximation_heat_kernel}. \label{tab:Fourier_mode}}
\begin{tabular*}{0.6\textwidth}{@{\extracolsep{\fill}} | l|c|c|c|c|  }
\hline
\backslashbox{$\varepsilon$}{$\tau$}
&\makebox[3em]{$10^{-1}$ }&\makebox[3em]{$10^{-2}$}&\makebox[3em]{$10^{-3}$}
&\makebox[3em]{$10^{-4}$}\\
\hline
$10^{-3}$ &    8 &   21 &  55  & 136\\ \hline
$10^{-6}$ &   11 &  34 &  100 &  296\\ \hline
$10^{-9}$&   14  & 43 &  130 &  396 \\ \hline
$10^{-12}$ &   17 &   50 &  154 &  475 \\ \hline
\end{tabular*}
\end{center}
\end{table}

To illustrate Theorem \ref{thm:approximation_heat_kernel} and provide intuition for how many Fourier modes are  needed in practice, we list the values of $M$ in Table \ref{tab:Fourier_mode} for various levels of precision $\varepsilon$ and time step $\tau$ for the approximation of $G_\tau^1(x)$ for $|x|\leq R$ with $R = \pi$.

In \eqref{e:HeatKer}, since $G^d_\tau( x)=\prod_{i=1}^d G^1_\tau(x_i)$, we simply use the tensor product to evaluate the Fourier spectral approximation of the heat kernel in higher dimensions. That is,
\begin{equation}\label{eq:approximation_heat_kernel}
G_\tau^d(x)  \ \approx  \ 
\frac{h^d}{(2\pi)^d}\prod_{i=1}^d\left(
\sum_{m_i=-M}^{M-1}
  e^{ -m_i^2 h^2 \tau + i h m_i  x_i }\right), 
  \qquad \qquad x \in \mathbb{R}^d.
\end{equation}

We use the type-1 NUFFT \eqref{4.1} to evaluate the Fourier transform of $\tilde{A}_{ij}$ from $x_q \in \mathcal{Q}$ to a uniform spectral domain to have $\hat{A}_{ij}(m)$ where $m \in \{-Mh,(-M+1)h, \cdots, (M-2)h, (M-1)h\}^3 $. Here, the integer $M$ and $h$ are computed by Theorem~\ref{thm:approximation_heat_kernel} when $\varepsilon$ and $\tau$ are given (see also Table~\ref{tab:Fourier_mode}). 
Multiplying $\hat{A}_{ij}(m)$ by $n_q e^{-|m|^2h^2\tau}$ where $n_q$ is the number of quadrature points, we then apply the type-2 NUFFT \eqref{4.2} to evaluate the inverse Fourier transform of the product from $m$ to $x_{q}^c$.

Our algorithm for solving the surface diffusion equation \eqref{eq:surface_diffusion} is summarized in Algorithm \ref{a:nufft_surface}.

\begin{algorithm}[t!]
\DontPrintSemicolon
 \KwIn{Let $\Omega$ be a closed surface in $\mathbb{R}^3$, $u_0: \Omega \rightarrow \mathbb{R}$, $\tau > 0$, a grid size $dx>0$, and accuracy $\varepsilon >0$.}
 \KwOut{ A matrix-valued function $A \colon \Omega \rightarrow \mathbb{R}^{n\times n}$ that approximately solves the surface diffusion equation \eqref{eq:surface_diffusion} at time $\tau$. }
 
{\bf 1.} Compute $w_b$ via \eqref{est:w_b}. Select a box $B$ containing a $w_b$ neighborhood of $\Omega$ and construct a uniform grid with the grid size $dx$ in $B$. For each grid point $x_g$, find the corresponding closest point on the surface and compute the distance $d_g$ to the surface. Denote the set of grid points $x_g$ with $d_g<w_b$ as $\mathcal{G}^b$. For each $x_g\in \mathcal{G}^b$, generate the $p-$th order quadrature points $x_q$ ({\it e.g.}, Gauss-Chebyshev quadrature points) in  $x_g + [0,dx]^3$ and find the corresponding closest points $x_{q}^c$ on the surface. Denote the number of quadrature points as $n_q$. \;

{\bf 2.} Extend $A(x)$ from the surface to the quadrature points using the closest point
function, i.e., $\tilde{A}_{ij}(x_q) = A_{ij}(x_{q}^c)$. \;

{\bf 3.} For accuracy $\varepsilon >0$, define  $h$ and $M$ as in Theorem \ref{thm:approximation_heat_kernel}. Generate grid points in the spectral domain, $m \in \{-Mh,(-M+1)h, \cdots, (M-2)h, (M-1)h\}^3$. \;

{\bf 4.} Apply the type-1 NUFFT \eqref{4.1} to approximately compute the Fourier transform of $\tilde{A}_{ij}$ to have $\hat{A}_{ij}(m)$. Compute $\bar{A}_{ij}(m) = n_qe^{-|m|^2dt}\hat{A}_{ij}(m)$. \;

{\bf 5.} Apply the type-2 NUFFT \eqref{4.2} to approximately compute the inverse Fourier transform of $\bar{A}_{ij}$ at the closest points $x_{q}^c \in \Omega$ to obtain $A^{*}_{ij}$. The approximate solution to the surface diffusion equation \eqref{eq:surface_diffusion} at time $\tau$ is then given by 
$A_{ij}(x_{qc},\tau) = (2\pi)^{-3} h^{-3} A^{*}_{ij}$.
\caption{An NUFFT based solver for the surface diffusion equation \eqref{eq:surface_diffusion}. } 
\label{a:nufft_surface}
\end{algorithm}

\section{Computational Experiments} \label{s:CompExp}
In this section, we illustrate the performance of the algorithms described in Section \ref{s:Imp} via several numerical examples. 
We implemented the algorithms in MATLAB. 
For the closed surface computations described in Sections \ref{s:Sphere}, \ref{s:Peanut}, and \ref{s:VolConst}, 
we used the NUFFT library from \cite{nufftlib} and CPM library from \cite{ruuth2008simple}. 
All reported results were obtained on a laptop with a 2.7GHz Intel Core i5 processor and 8GB of RAM.  

\subsection{Flat torus} \label{s:FT} We consider the flat torus, $\Omega = [-\frac{1}{2},\frac{1}{2}]^2$, and $n=2$. We use $1024 \times 1024$ grid points to discretize $\Omega$.  We study various initial conditions in the following three examples.

\medskip

In the first example, we set the initial condition to be 
\begin{align*}
A(r,\theta)  = 
\begin{cases}
\begin{bmatrix}
\cos \alpha & -\sin \alpha \\
\sin \alpha & \cos \alpha
\end{bmatrix},   & \text{if} \  r < 0.3+0.06 \sin(6 \theta) \\
\begin{bmatrix}
\cos \alpha & \sin \alpha \\
\sin \alpha & -\cos \alpha
\end{bmatrix} & \text{otherwise} 
\end{cases},
\end{align*}
where $\alpha = \alpha(x,y) = \dfrac{\pi}{2}\sin(2\pi(x+y))$ and $(r, \theta)$ is the corresponding polar coordinate of $(x,y)$. 
This initial condition is plotted in the top left panel of Figure~\ref{fig:flat_torus_dynamic_shrinking}. 
In Figure \ref{fig:flat_torus_dynamic_shrinking}, and also in Figures \ref{fig:flat_torus_dynamic_para}--\ref{fig:sphere_volume_preserving}, the domain is colored by the sign of the determinant of the matrix. For a matrix field $A \in H^1\left(\Omega, O_n \right)$, we use the convention 
\begin{center}
\begin{tabular}{l c l c l }
$x$ is yellow &  $\iff$  &  $\textrm{det}(A(x)) =1$ & $\iff$ & $A(x) \in SO(n)$ \\
$x$ is blue  & $\iff$  & $\textrm{det}(A(x)) =-1$ & $\iff$ & $A(x) \in SO^-(n)$, 
\end{tabular}
\end{center}
The vector field in the figure is generated by the first column vector in $A$, initially given by $\begin{pmatrix} \cos \alpha \\ \sin \alpha \end{pmatrix}$.
We set $\tau = 8dx = 0.0078125$. Figure~\ref{fig:flat_torus_dynamic_shrinking} displays various snapshots of the time dynamics for this evolution. We observe that for this initial field with an star-shaped line defect, the region where $A \in SO(n)$ shrinks with the interface becoming a circle before the region where $A\in SO(n)$ vanishes.  The field converges to a uniform matrix field, which is easily seen to be a minimum of \eqref{eq:GL}.

\medskip

\begin{figure}[ht]
\centering
\includegraphics[scale=0.18,clip,trim= 7cm 1cm 7cm 1cm]{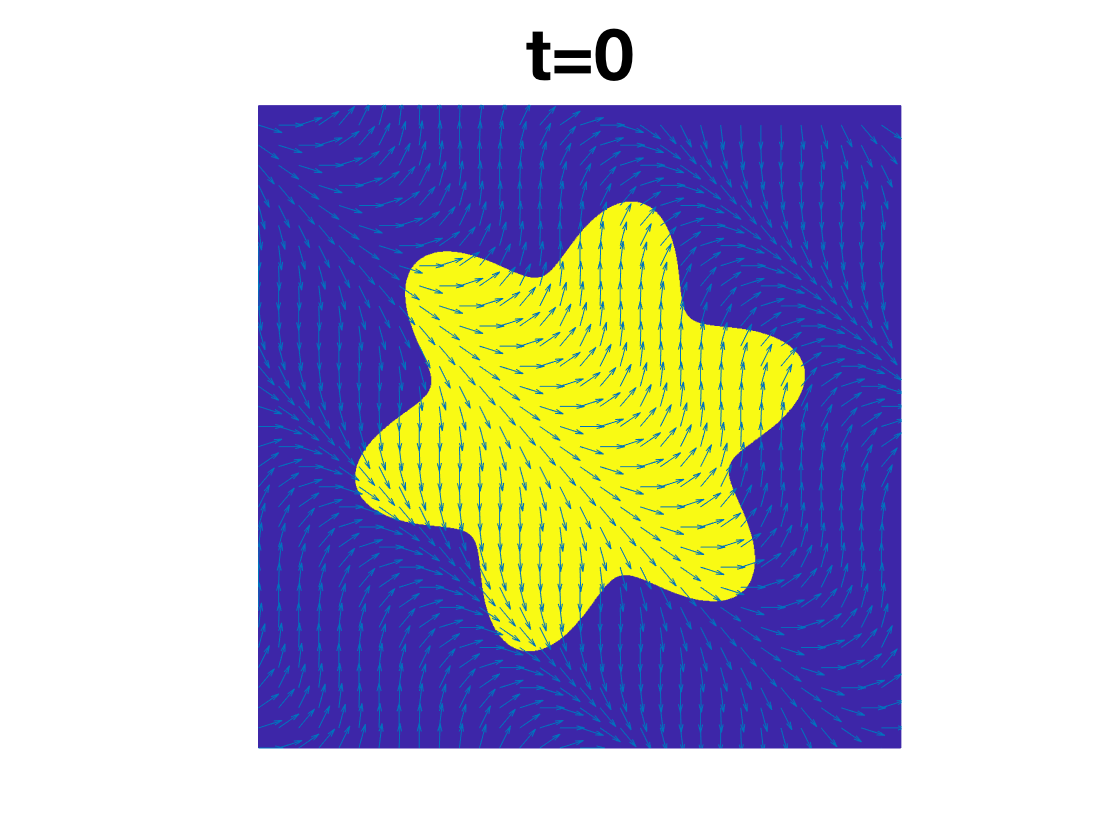}
\includegraphics[scale=0.18,clip,trim= 7cm 1cm 7cm 1cm]{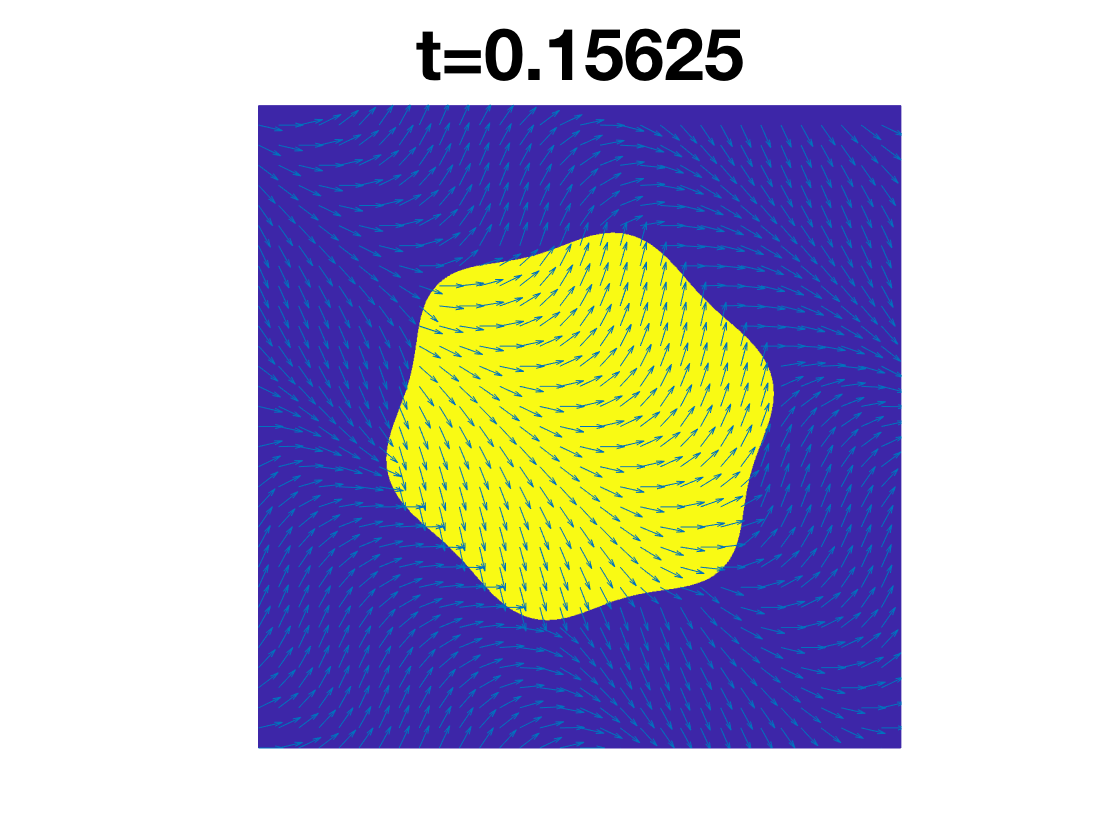}
\includegraphics[scale=0.18,clip,trim= 7cm 1cm 7cm 1cm]{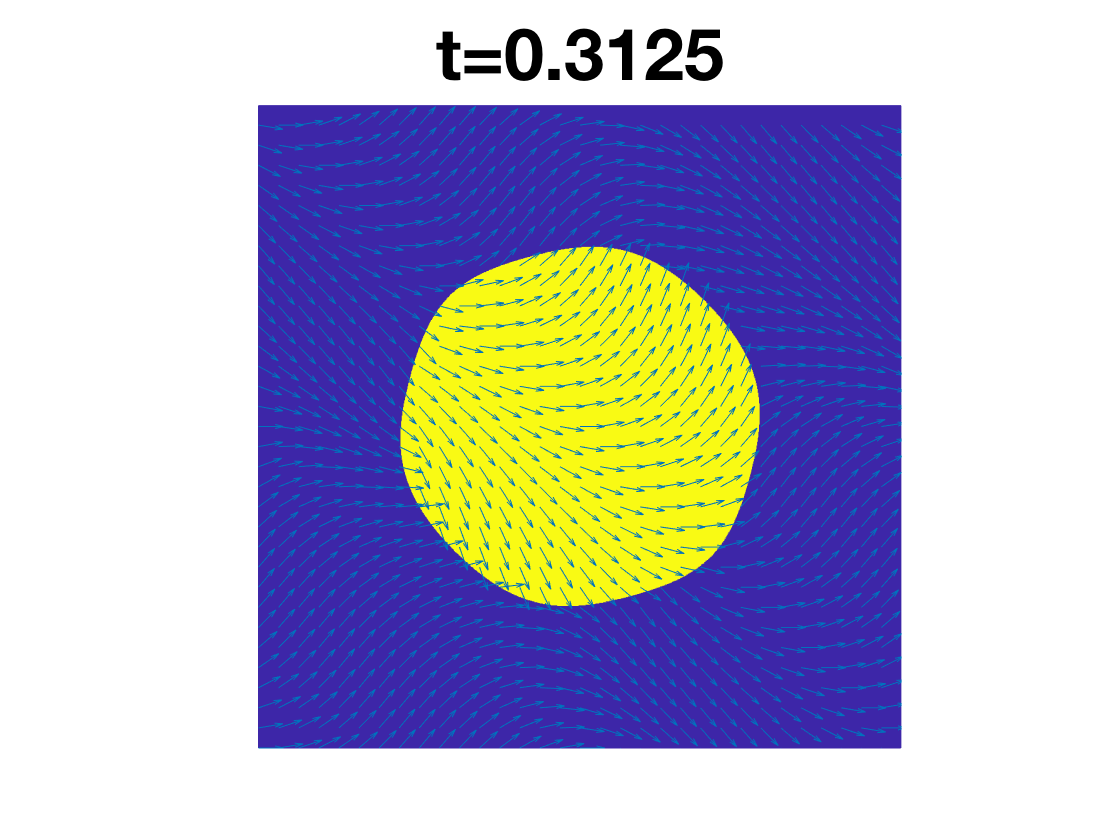}
\includegraphics[scale=0.18,clip,trim= 7cm 1cm 7cm 1cm]{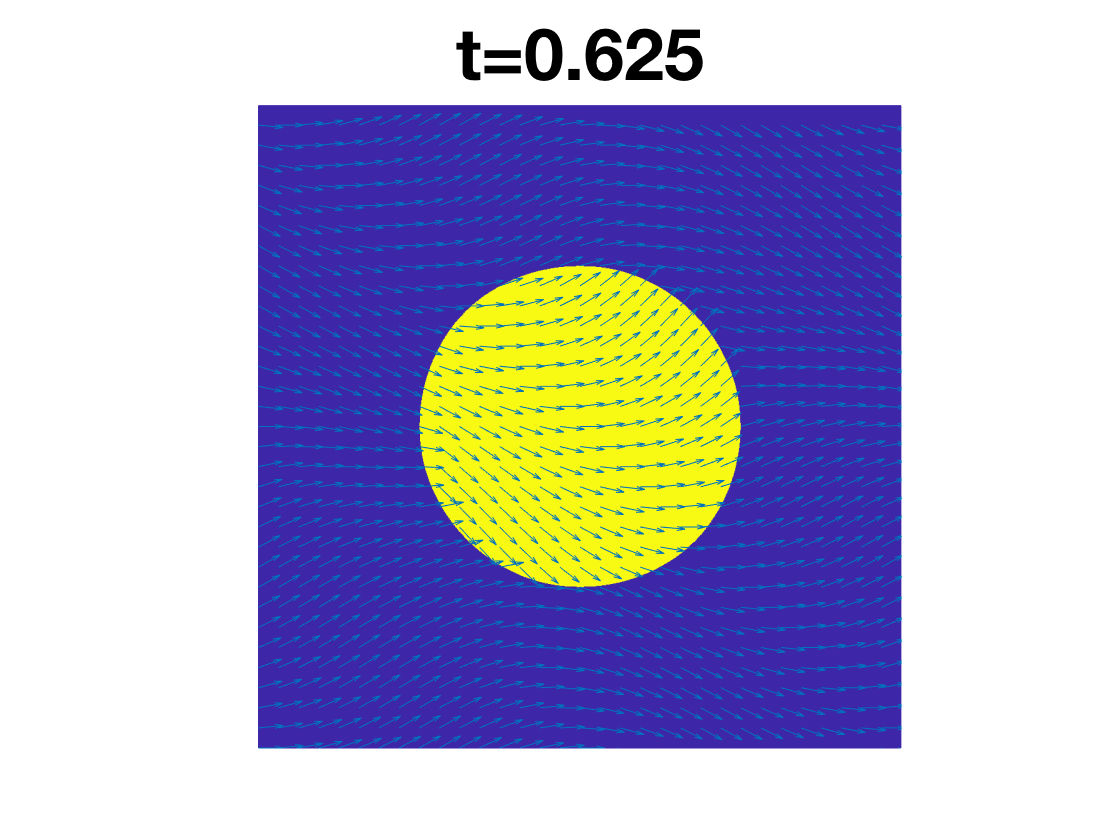}
\includegraphics[scale=0.18,clip,trim= 7cm 1cm 7cm 1cm]{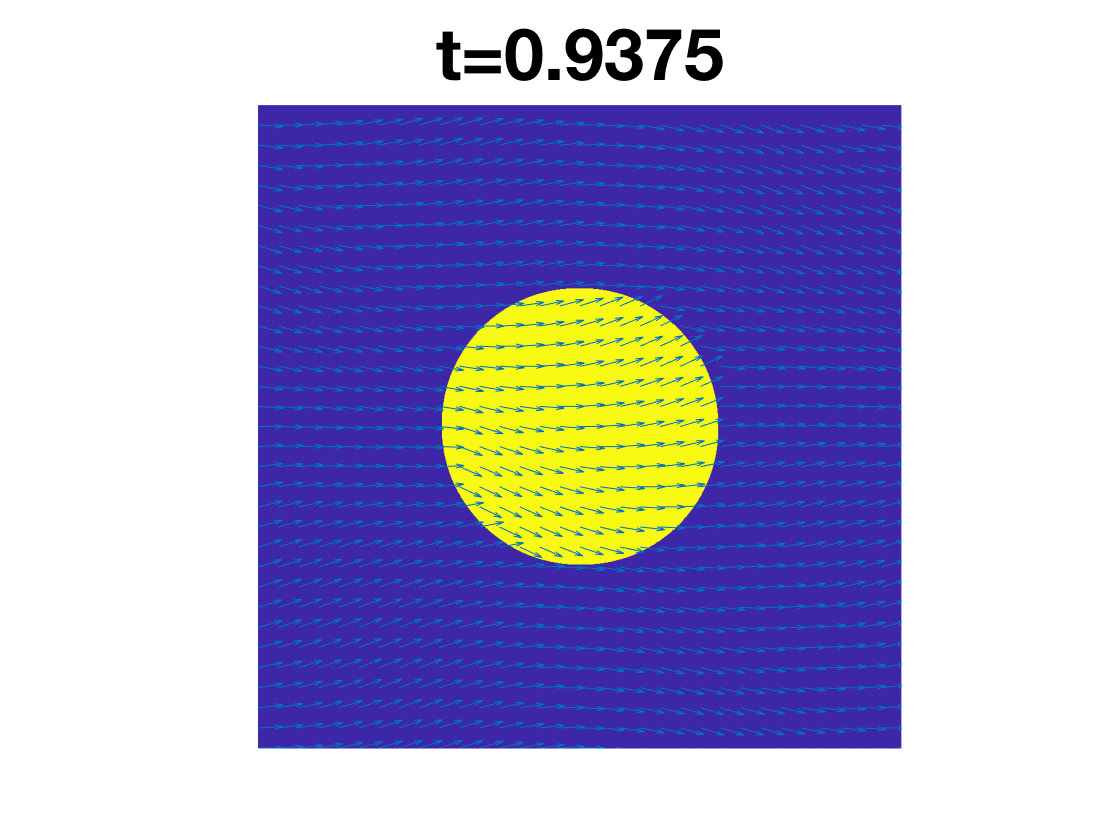}
\includegraphics[scale=0.18,clip,trim= 7cm 1cm 7cm 1cm]{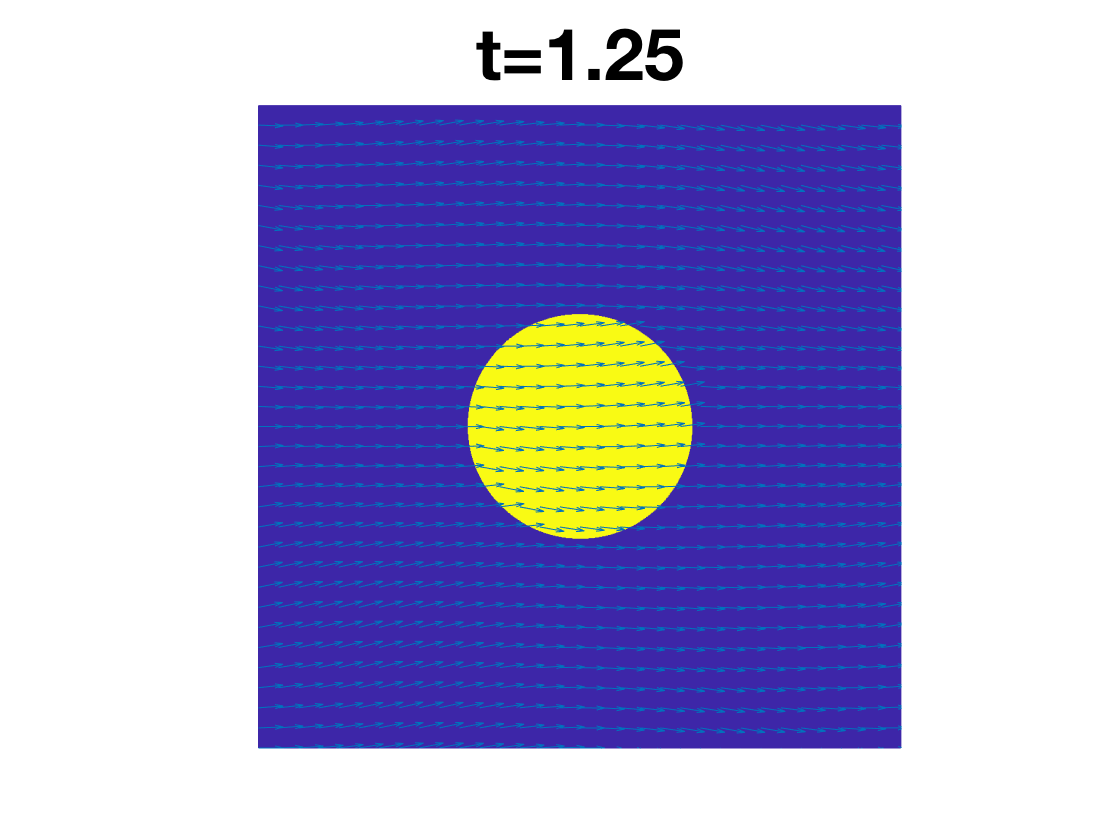}
\includegraphics[scale=0.18,clip,trim= 7cm 1cm 7cm 1cm]{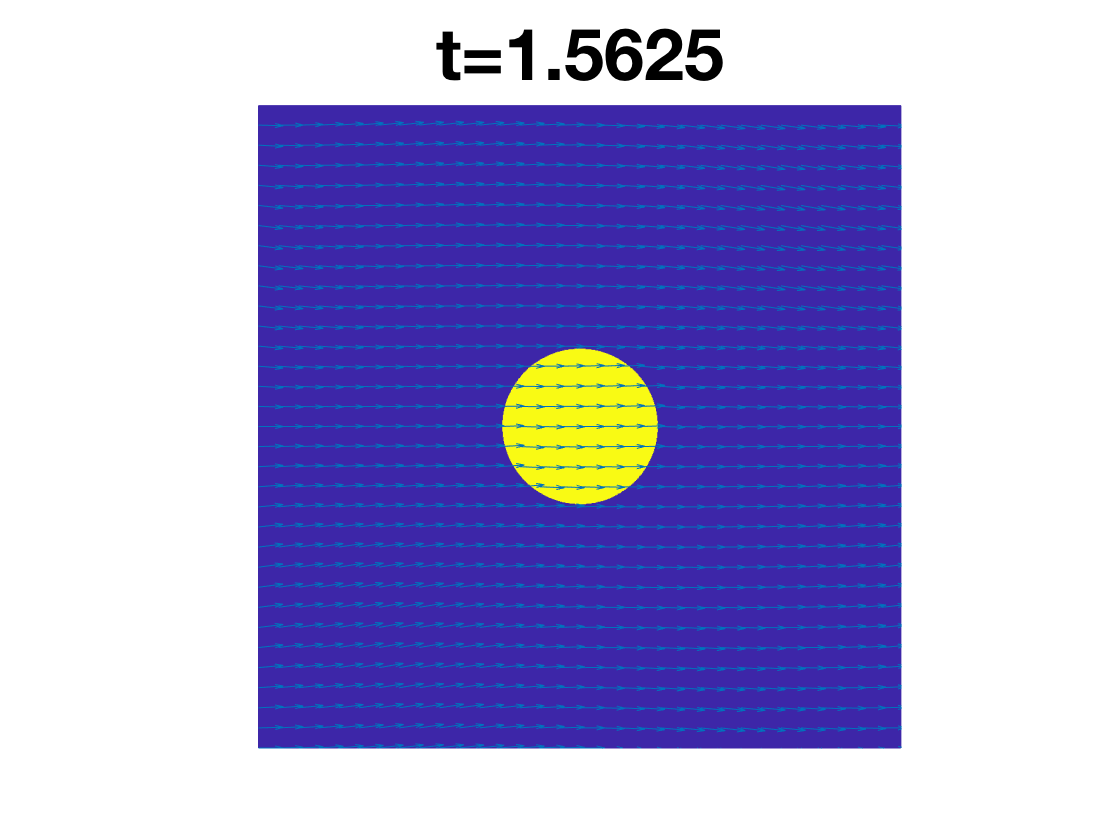}
\includegraphics[scale=0.18,clip,trim= 7cm 1cm 7cm 1cm]{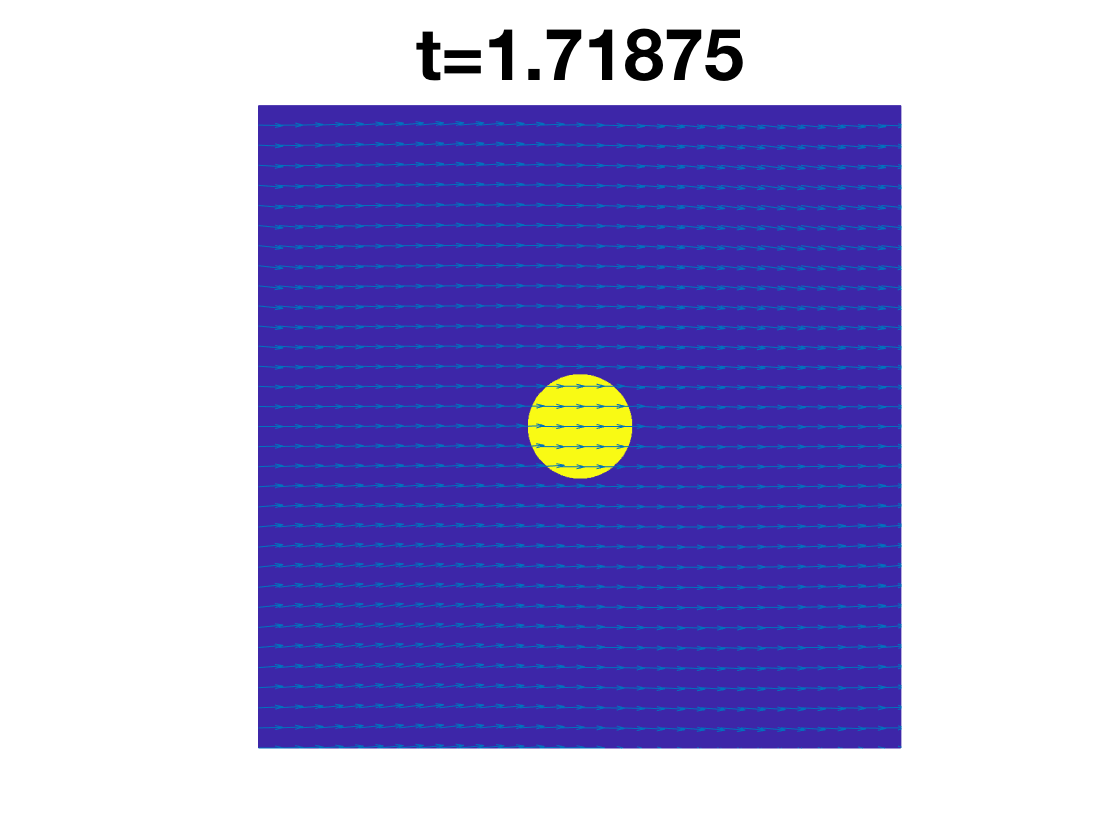}
\includegraphics[scale=0.18,clip,trim= 7cm 1cm 7cm 1cm]{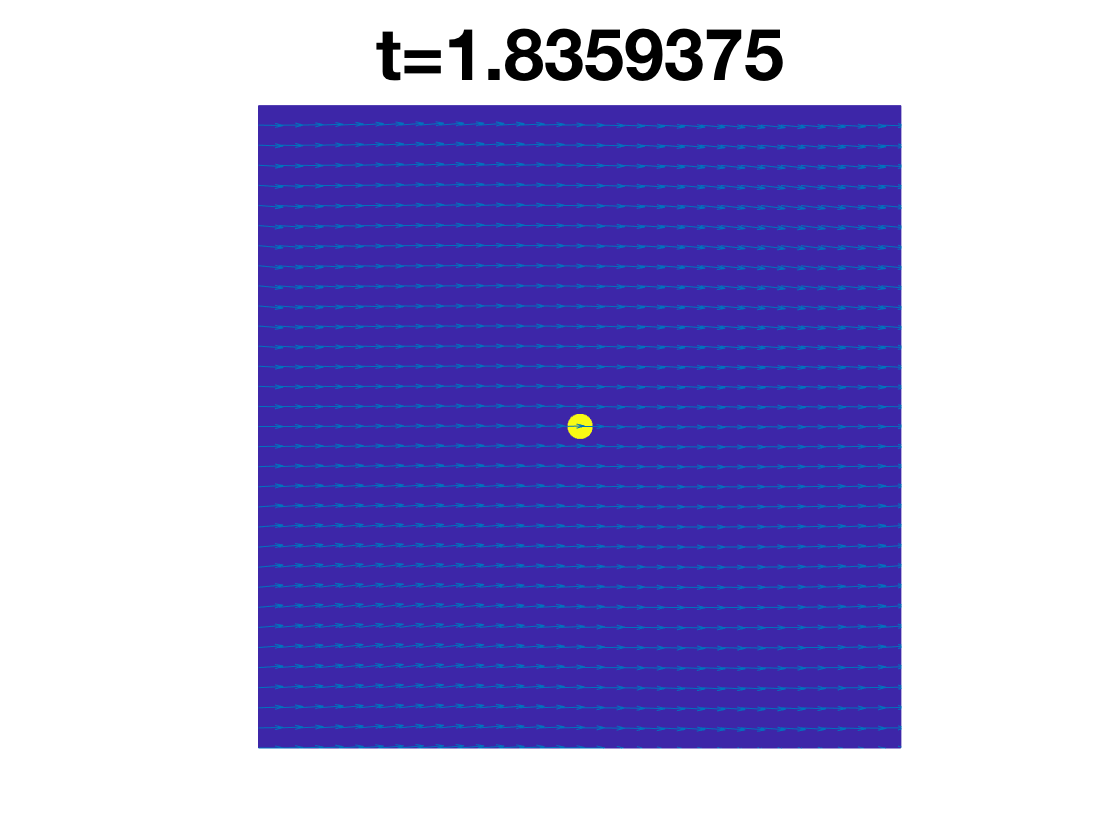}
\caption{Snapshots of the time dynamics of a closed line defect on the flat torus. See \S \ref{s:FT}.}\label{fig:flat_torus_dynamic_shrinking}
\end{figure}

In the second example, we take the initial condition to be 
\begin{align*}
A(x,y) = 
\begin{cases}
\begin{bmatrix}
\cos \alpha & -\sin \alpha \\
\sin \alpha & \cos \alpha
\end{bmatrix},   & \text{if}  \  x>0.25|\sin(2.5\pi y)|+0.2 \  \  \text{or}  \ \  x<-0.25|\sin(2.5\pi y)|-0.2 , \\
\begin{bmatrix}
\cos \alpha & \sin \alpha \\
\sin \alpha & -\cos \alpha
\end{bmatrix} & \text{otherwise} 
\end{cases}, 
\end{align*}
where $\alpha = \alpha(x,y) = \pi \sin(2\pi y)$  and  $\tau = 64dx = 0.0625$, where $dx$ is the grid size. 
This initial condition is plotted in the top left panel of Figure~\ref{fig:flat_torus_dynamic_para}. 

Figure~\ref{fig:flat_torus_dynamic_para} displays snapshots of the time evolution of the field. The  two initial line defects converge to two straight lines,  parallel to the $y-$axis. The fields on both sides of the defect are constant and aligned. {\it i.e.}, $\alpha$ is uniform.  This is a local minimum of the energy in \eqref{eq:GL}.  

\begin{figure}[ht]
\includegraphics[scale=0.18,clip,trim= 7cm 1cm 7cm 1cm]{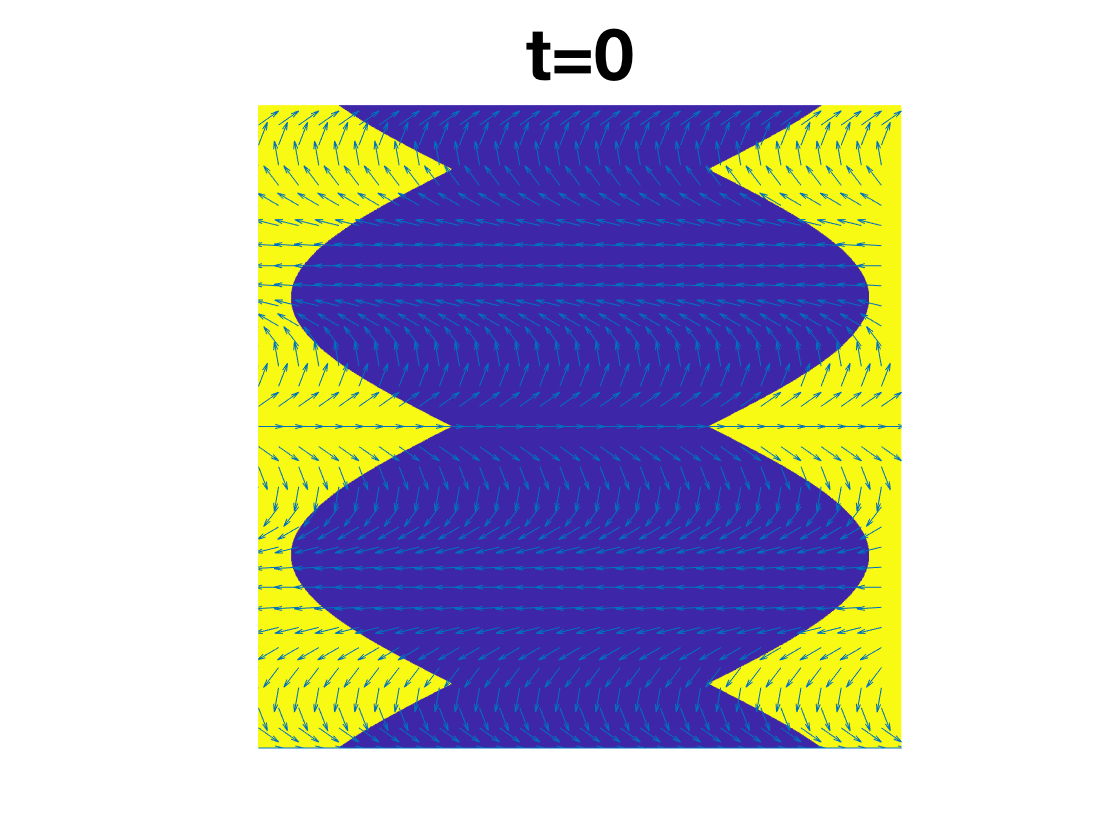}
\includegraphics[scale=0.18,clip,trim= 7cm 1cm 7cm 1cm]{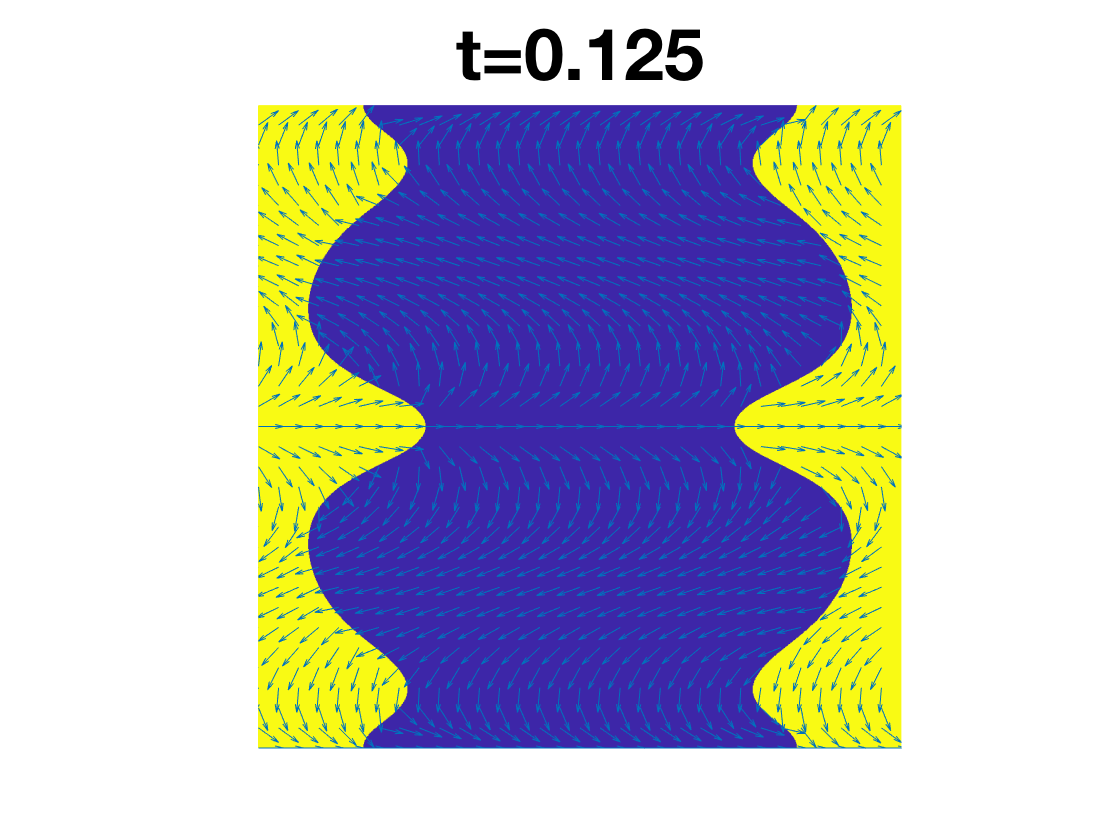}
\includegraphics[scale=0.18,clip,trim= 7cm 1cm 7cm 1cm]{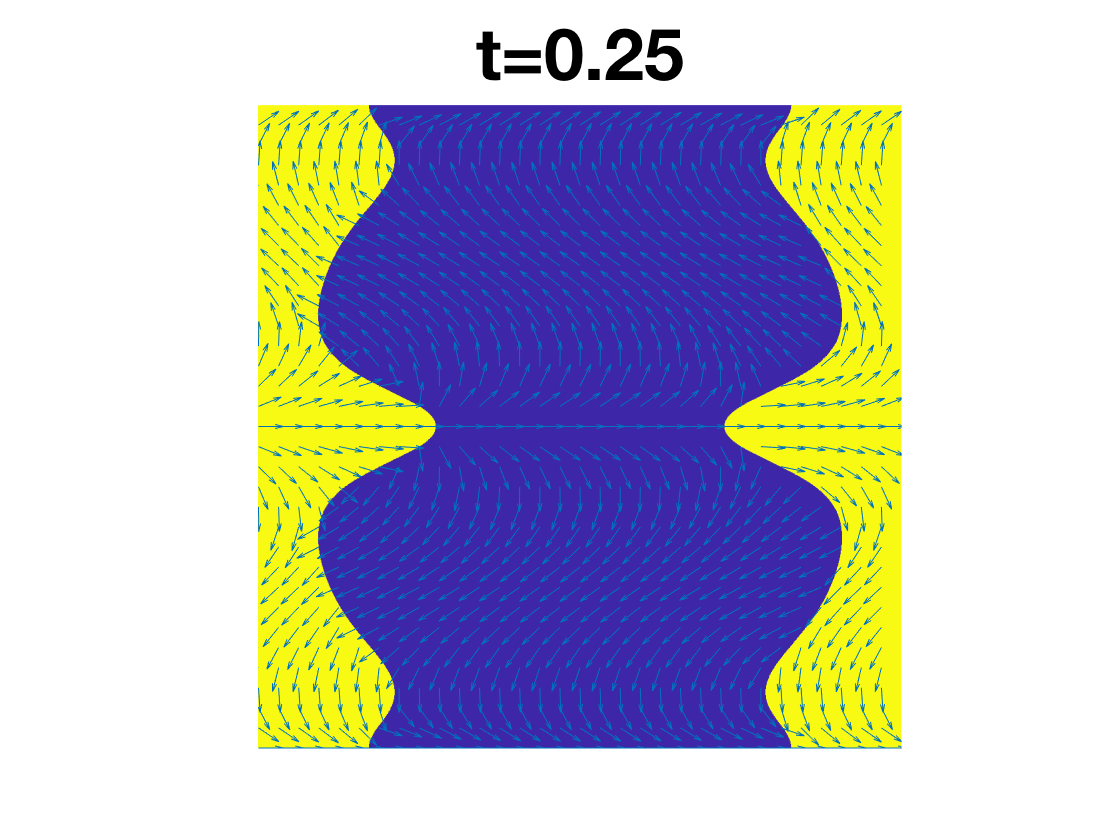}
\includegraphics[scale=0.18,clip,trim= 7cm 1cm 7cm 1cm]{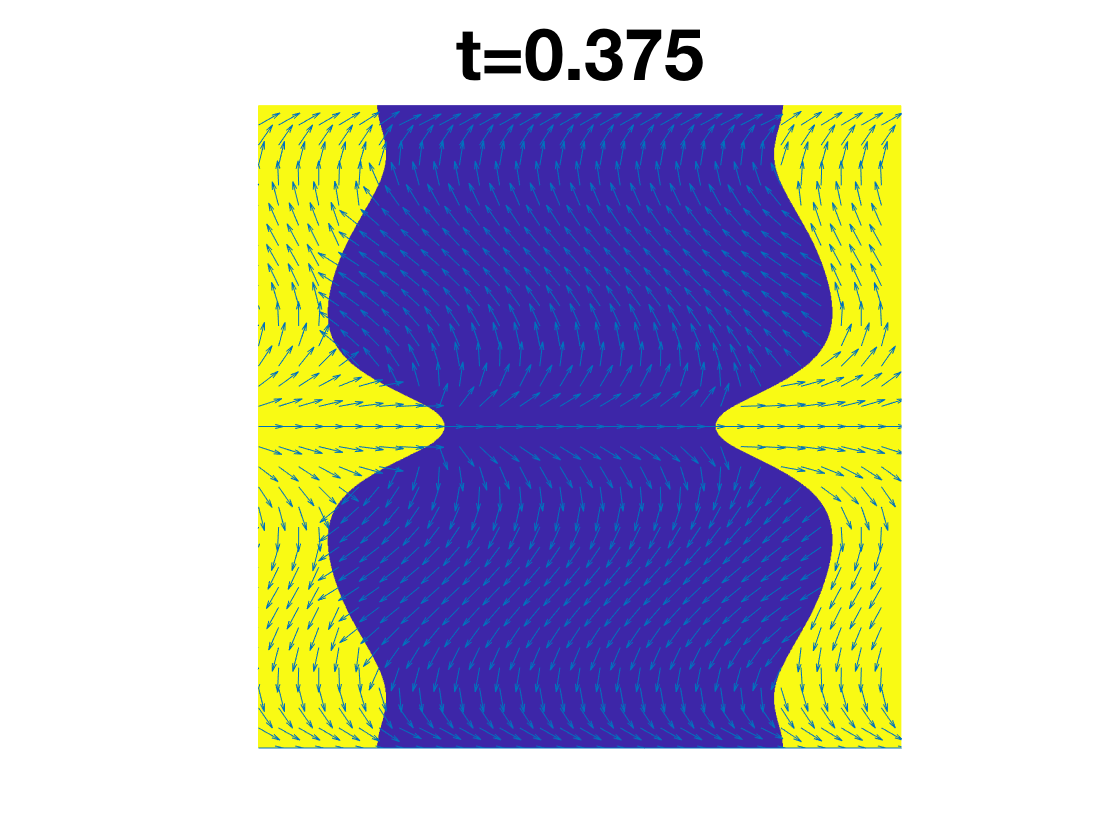}
\includegraphics[scale=0.18,clip,trim= 7cm 1cm 7cm 1cm]{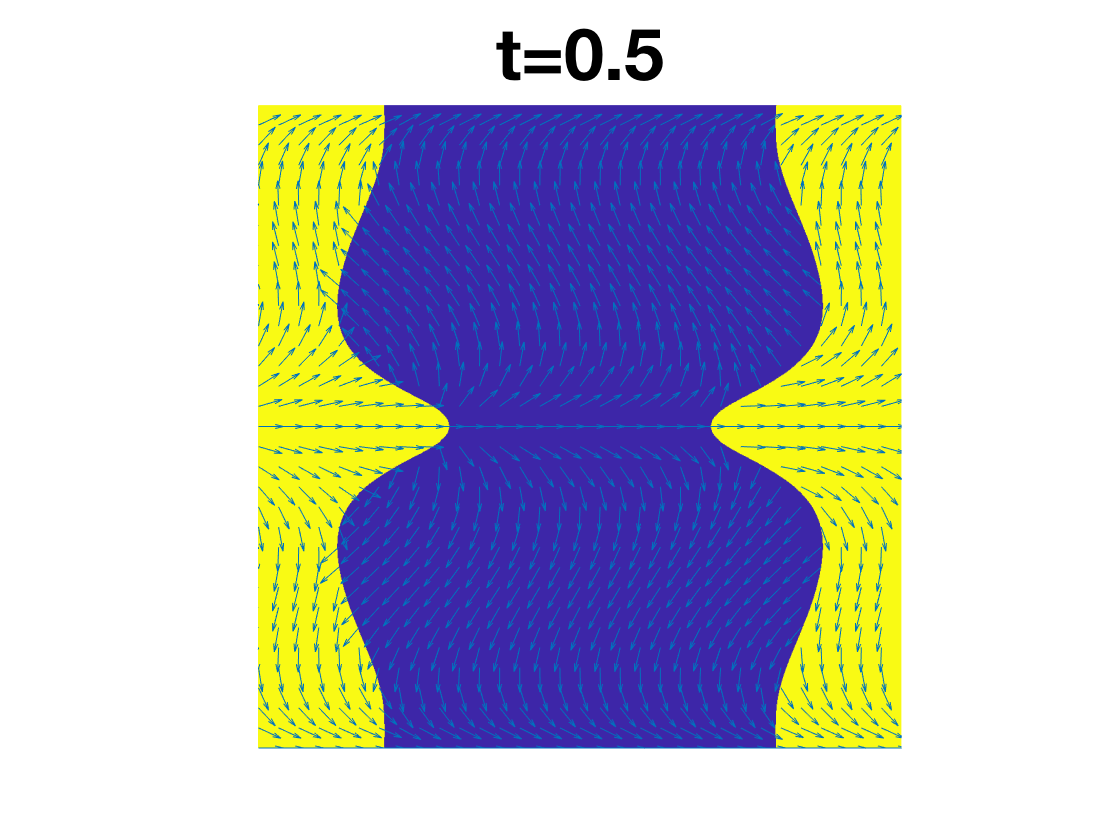}
\includegraphics[scale=0.18,clip,trim= 7cm 1cm 7cm 1cm]{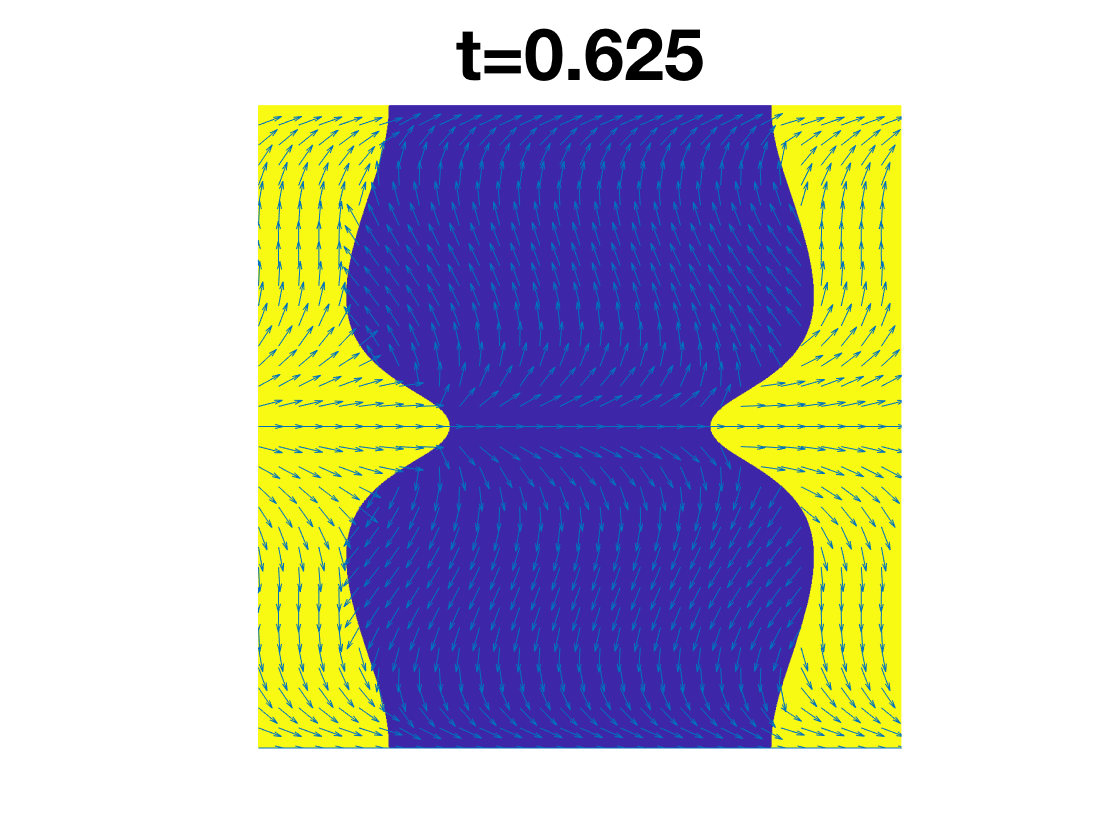}
\includegraphics[scale=0.18,clip,trim= 7cm 1cm 7cm 1cm]{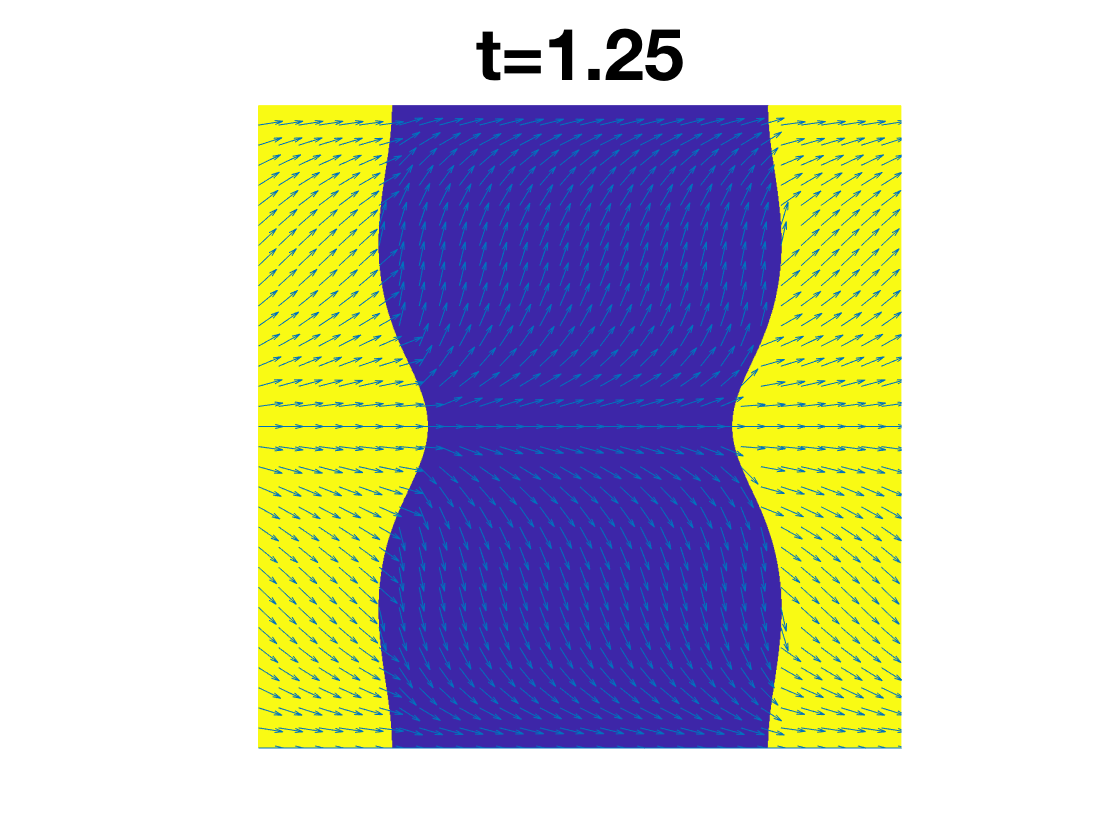}
\includegraphics[scale=0.18,clip,trim= 7cm 1cm 7cm 1cm]{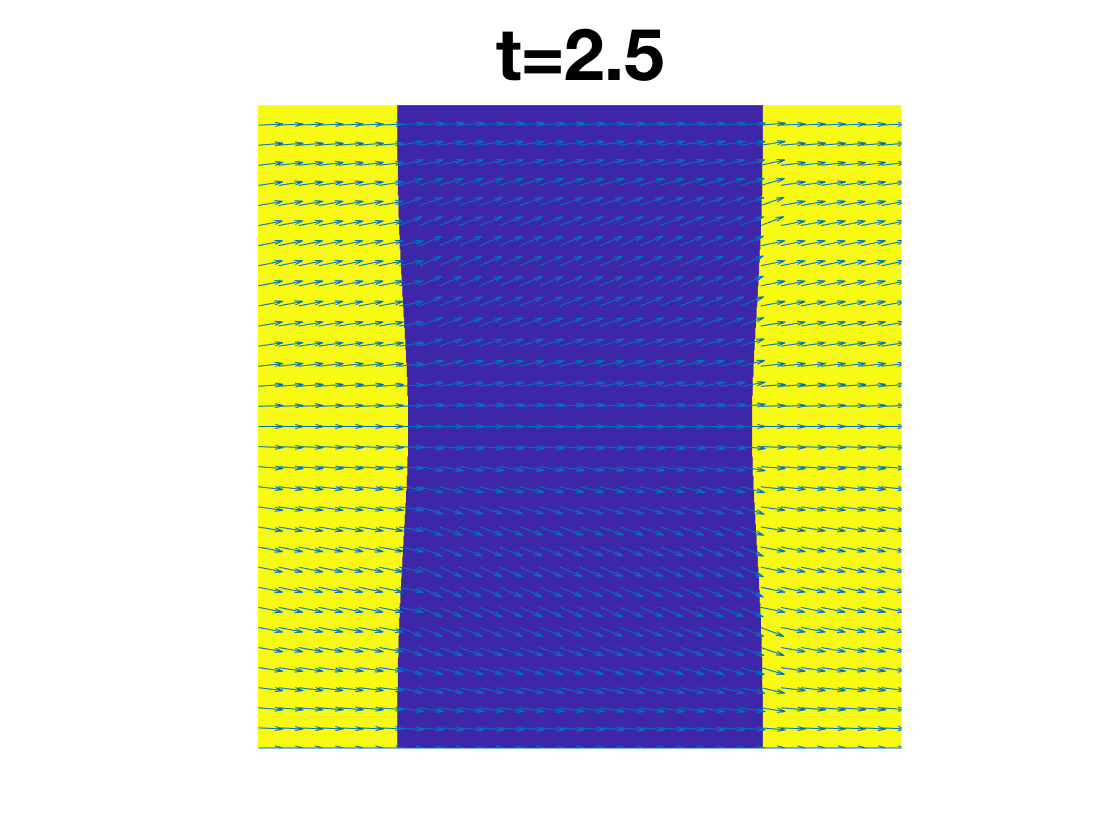}
\includegraphics[scale=0.18,clip,trim= 7cm 1cm 7cm 1cm]{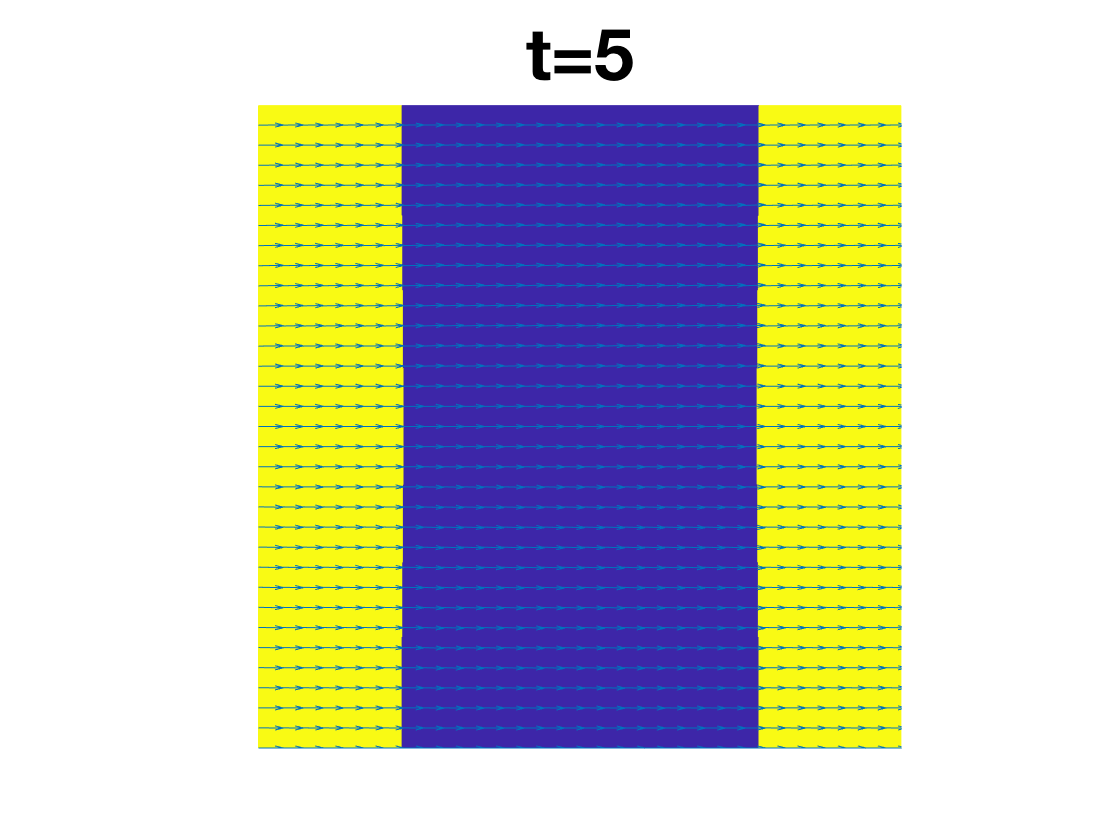}
\caption{Snapshots of the time dynamics of two closed line defects on the flat torus. See \S \ref{s:FT}.}\label{fig:flat_torus_dynamic_para}
\end{figure}

\medskip

As a third example, we take the initial condition to be 
\begin{align*}
A(x,y) = \begin{cases}
\begin{bmatrix}
\cos \alpha & -\sin \alpha \\
\sin \alpha & \cos \alpha
\end{bmatrix}   & \text{if} \  x>0.25|\sin(2.5\pi y)|+0.2 \  \  \text{or}  \ \  x<-0.25|\sin(2.5\pi y)|-0.2 , \\
\begin{bmatrix}
\cos \alpha & \sin \alpha \\
\sin \alpha & -\cos \alpha
\end{bmatrix}   & \text{otherwise} 
\end{cases} . 
\end{align*}
where $\alpha = \alpha(x,y) = 2\pi y$, $\tau = 64dx = 0.0625$ where $dx$ is the grid size and the winding number is 1. 
This initial condition is plotted in the top left panel of Figure~\ref{fig:flat_torus_dynamic_winding_1}. 
We emphasize that the initial interface here and in the second example (see Figures~\ref{fig:flat_torus_dynamic_para}) are the same; the only difference is the function that defines the angle $\alpha = \alpha(x,y)$.
Figure~\ref{fig:flat_torus_dynamic_winding_1} displays the time evolution of the field.  We observe that the initial line defect begins to straighten as in  Figure~\ref{fig:flat_torus_dynamic_para}, but then, as a consequence of the field, finds a way close and shrink to a point. The solution converges to a non-constant field, in contrast to the first example (see Figure~\ref{fig:flat_torus_dynamic_shrinking}).

To further describe this phenomena, we let $v\colon \Omega \to \mathbb C$ be a complex-valued field with no zeros. Let $\gamma \colon [0,1] \rightarrow \Omega$ be a closed curve. We define the \emph{index of $\gamma$ with respect to $v$} to  be
\[\mathrm{ind}_v(\gamma) := \frac{1}{2\pi} \left[ \arg{v(\gamma(1))} - \arg{ v(\gamma(0))} \right].  \]
Clearly the index of $\gamma$ is an integer and varies continuously with deformations to $\gamma$, so it depends only on the homotopy class of $\gamma$. For a torus, we can parameterize the homotopy classes by the number of times the curve wraps around $\Omega$ in the $x$- and $y$-directions. Furthermore, if we let $[\gamma]_{m,n}$ denote the equivalence class of curves that wraps around $\Omega$ $m$ times in the $x$-direction and $n$ times in the $y$-direction, then it is not difficult to see that 
\[\mathrm{ind}_v([\gamma]_{m,n}) = \mathrm{ind}_v([\gamma]_{1,0})^m + \mathrm{ind}_v([\gamma]_{0,1})^n. \]
So we can characterize the index of any curve in terms of the indices of $[\gamma]_{1,0}$ and $[\gamma]_{0,1}$. For a given field $v$, we let $I = \left( \mathrm{ind}_v([\gamma]_{1,0}), \mathrm{ind}_v([\gamma]_{1,0}) \right) $ be the pair of indices corresponding to curves that wrap around $\Omega$ once in the $x$- and $y$-directions. 

By identifying $SO(2)$ with $\mathbb S^1 \subset \mathbb C$, we can define the index of an $SO(2)$-valued field on the torus as above. The pair of indices of the uniform field in the first example is $I = (0,0)$. The pair of indices of the stationary field in third example is $I = (1,0)$. More generally, we could consider a field with pair of indices $I = (m,n)$  
of the form 
$$ 
A(x,y) =
\begin{bmatrix}
\cos \alpha & -\sin \alpha \\
\sin \alpha & \cos \alpha
\end{bmatrix}  
\qquad \qquad
\textrm{where } \ \alpha = \alpha(x,y) = \dfrac{\pi}{2}\sin(2\pi(mx+ny)).
$$
Further numerical examples suggest that, for all $m,n \in \mathbb N$, this field is stationary under iterations of the MBO algorithm.  

\begin{figure}[h!]
\centering
\includegraphics[scale=0.18,clip,trim= 7cm 1cm 7cm 1cm]{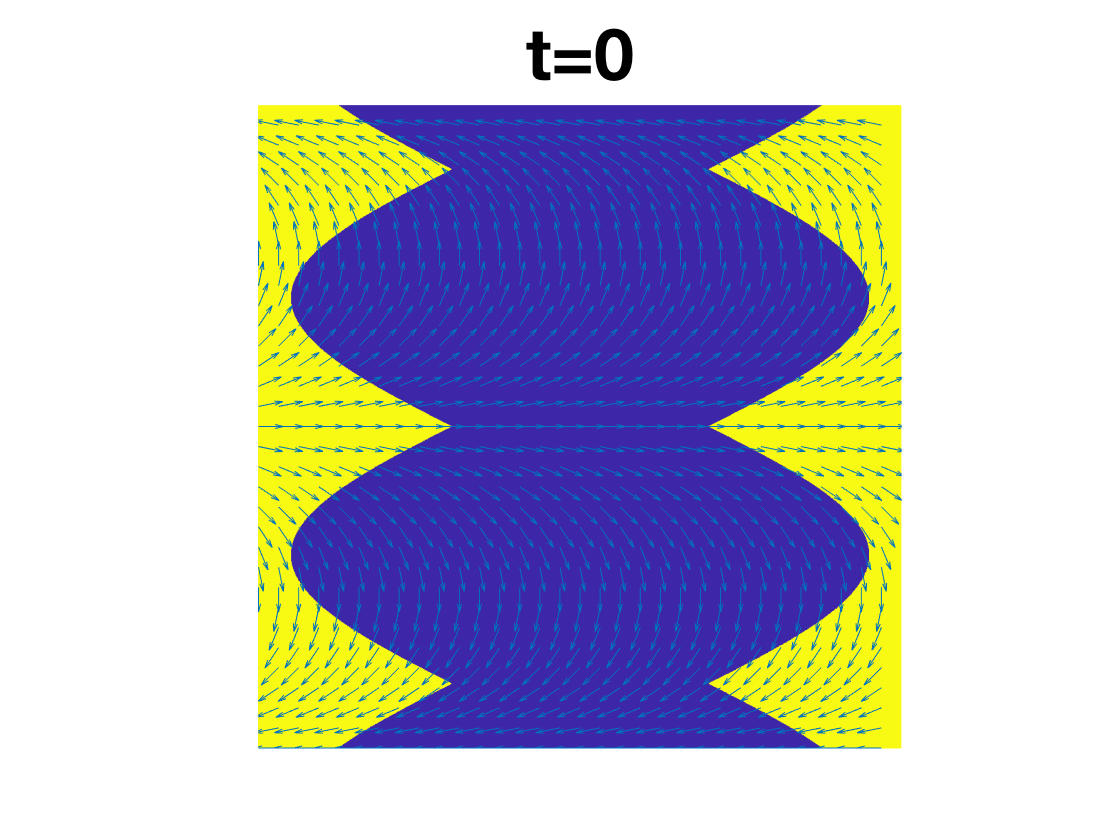}
\includegraphics[scale=0.18,clip,trim= 7cm 1cm 7cm 1cm]{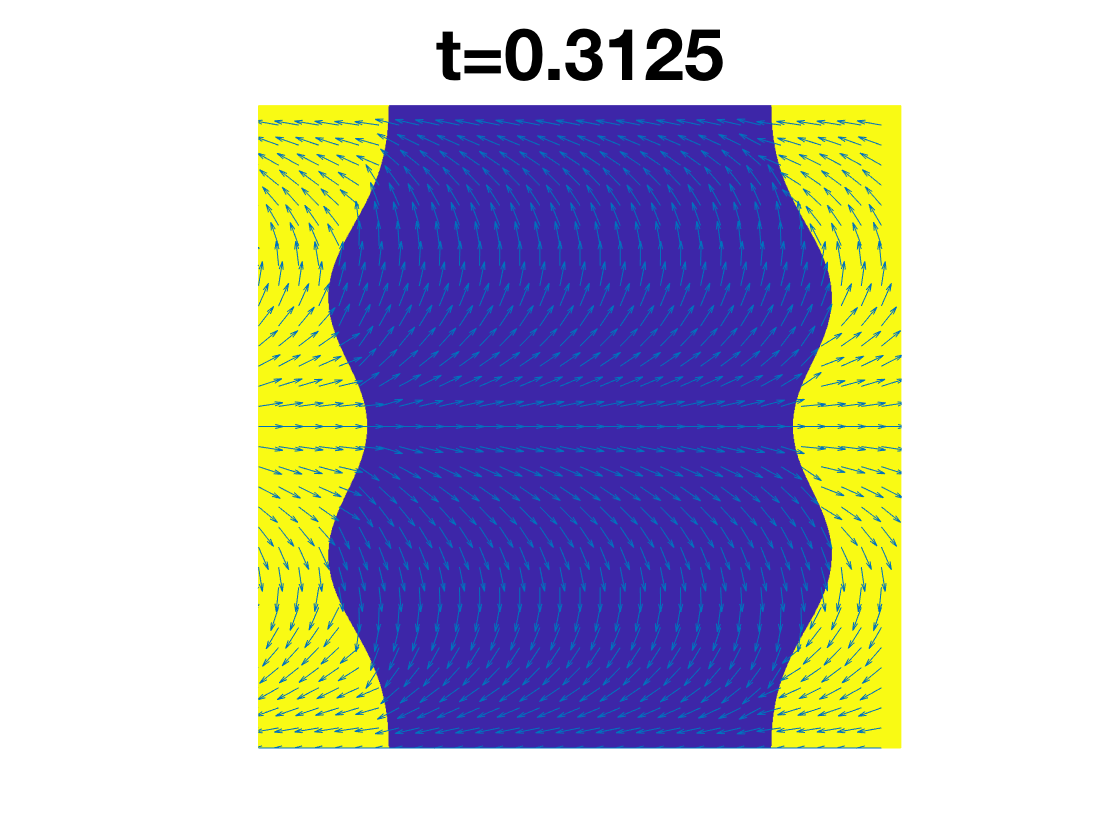}
\includegraphics[scale=0.18,clip,trim= 7cm 1cm 7cm 1cm]{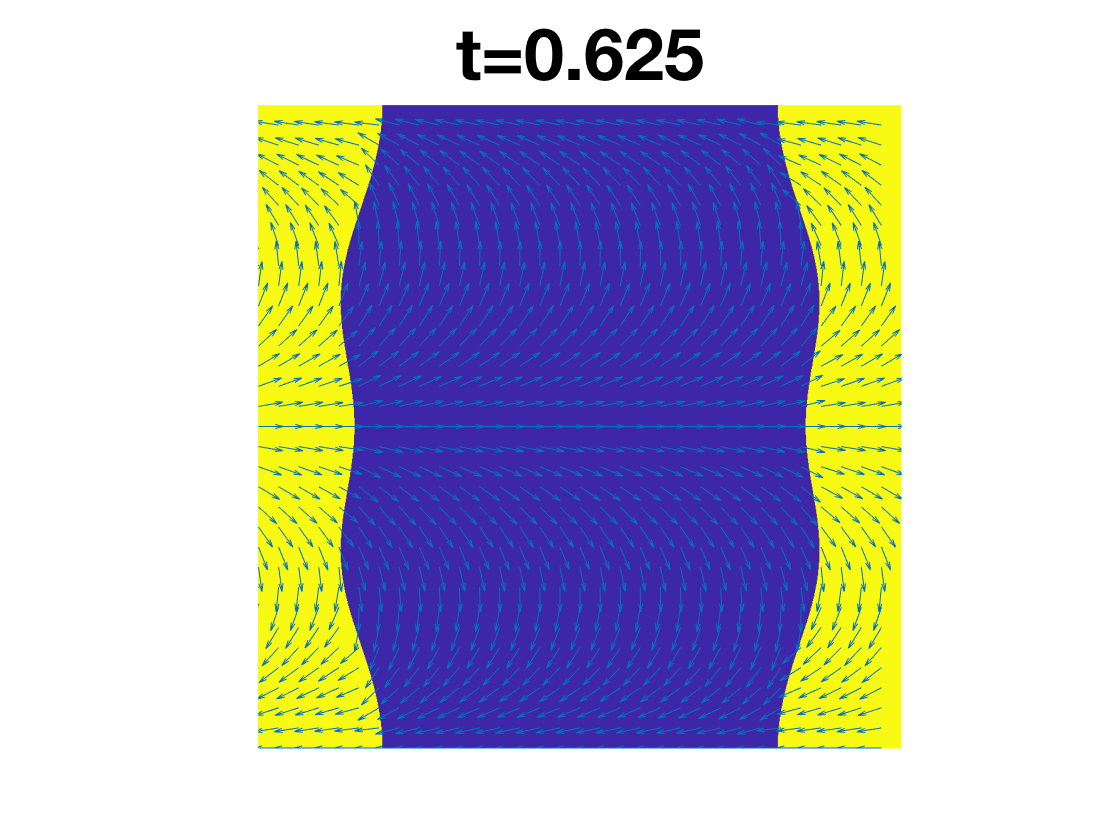} \\ 
\includegraphics[scale=0.18,clip,trim= 7cm 1cm 7cm 1cm]{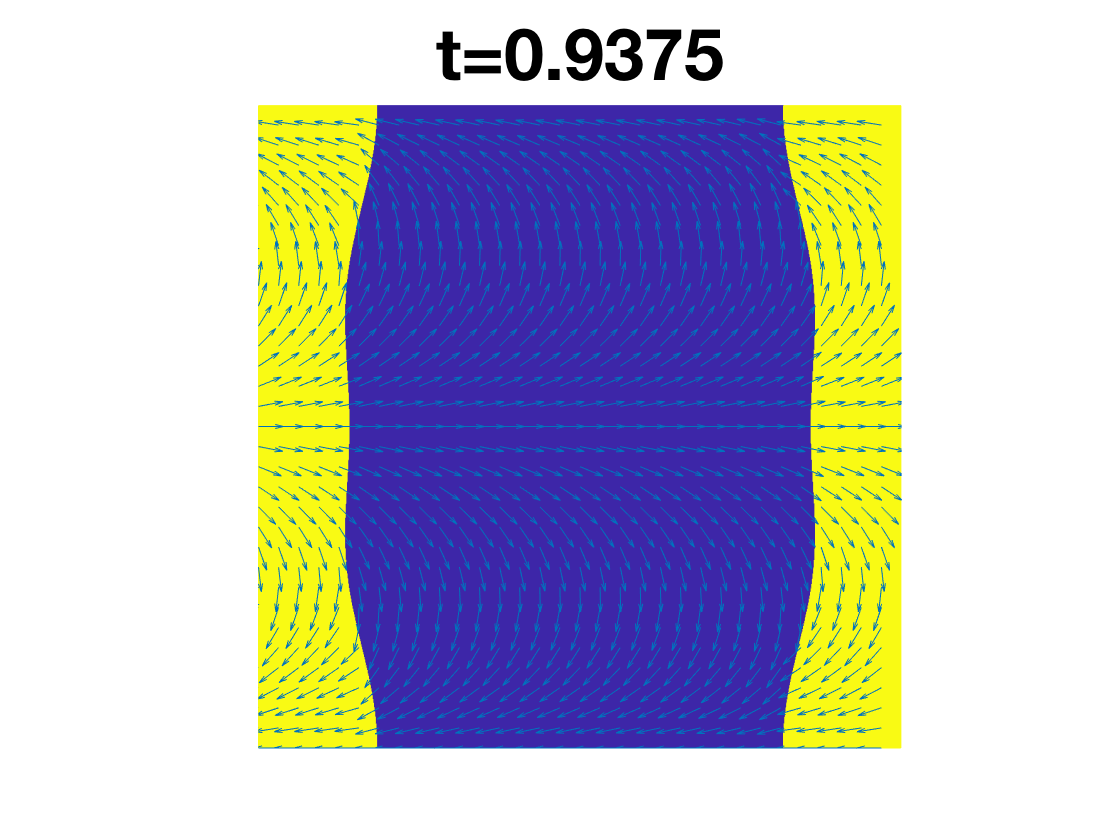}
\includegraphics[scale=0.18,clip,trim= 7cm 1cm 7cm 1cm]{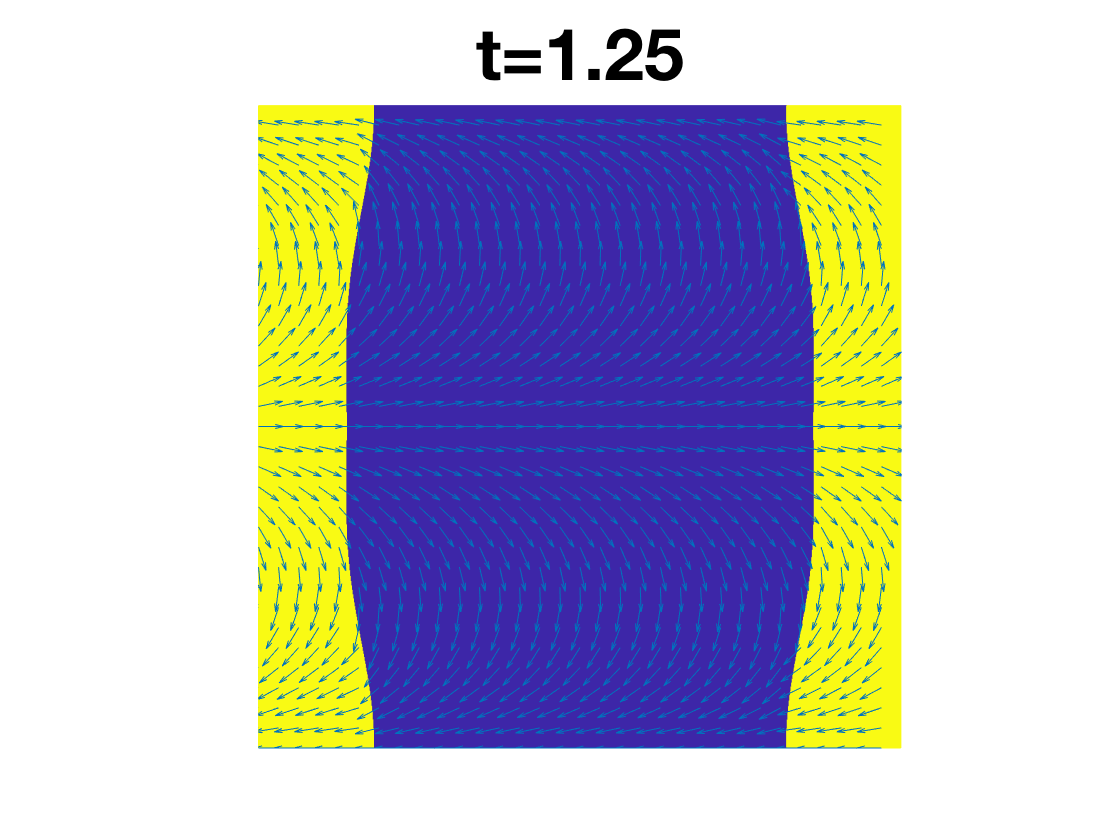}
\includegraphics[scale=0.18,clip,trim= 7cm 1cm 7cm 1cm]{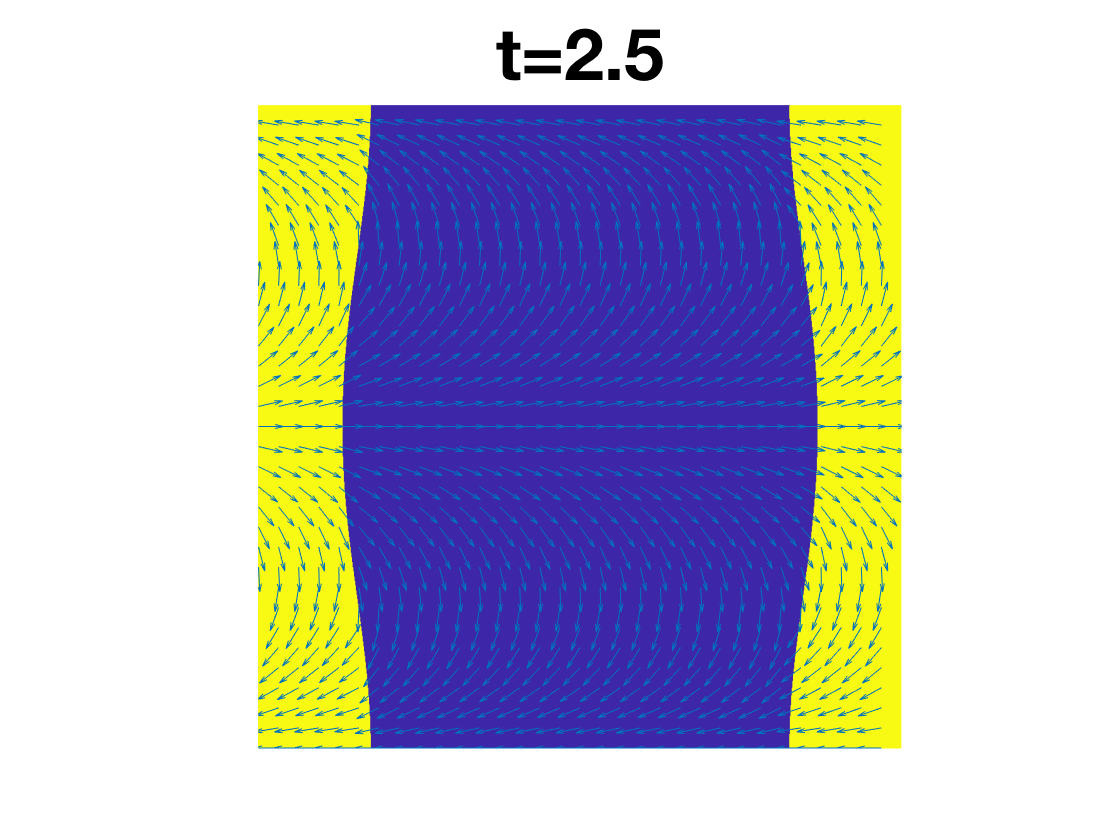} \\
\includegraphics[scale=0.18,clip,trim= 7cm 1cm 7cm 1cm]{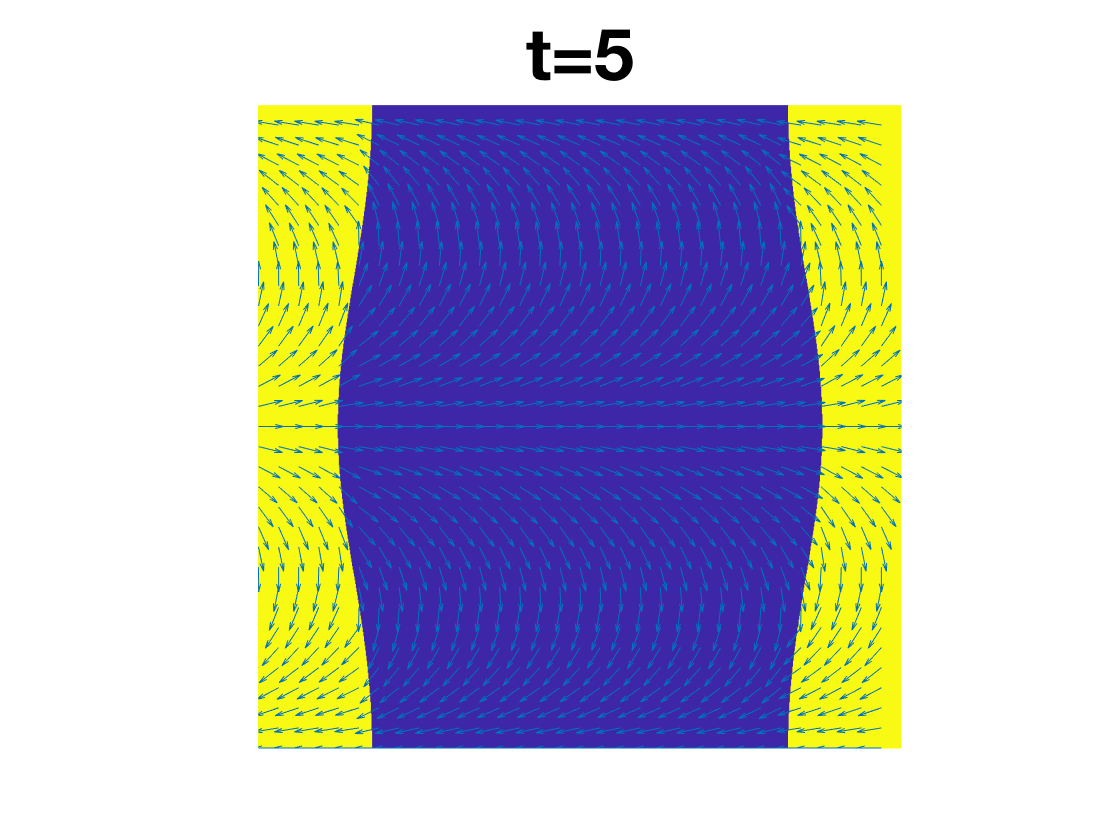}
\includegraphics[scale=0.18,clip,trim= 7cm 1cm 7cm 1cm]{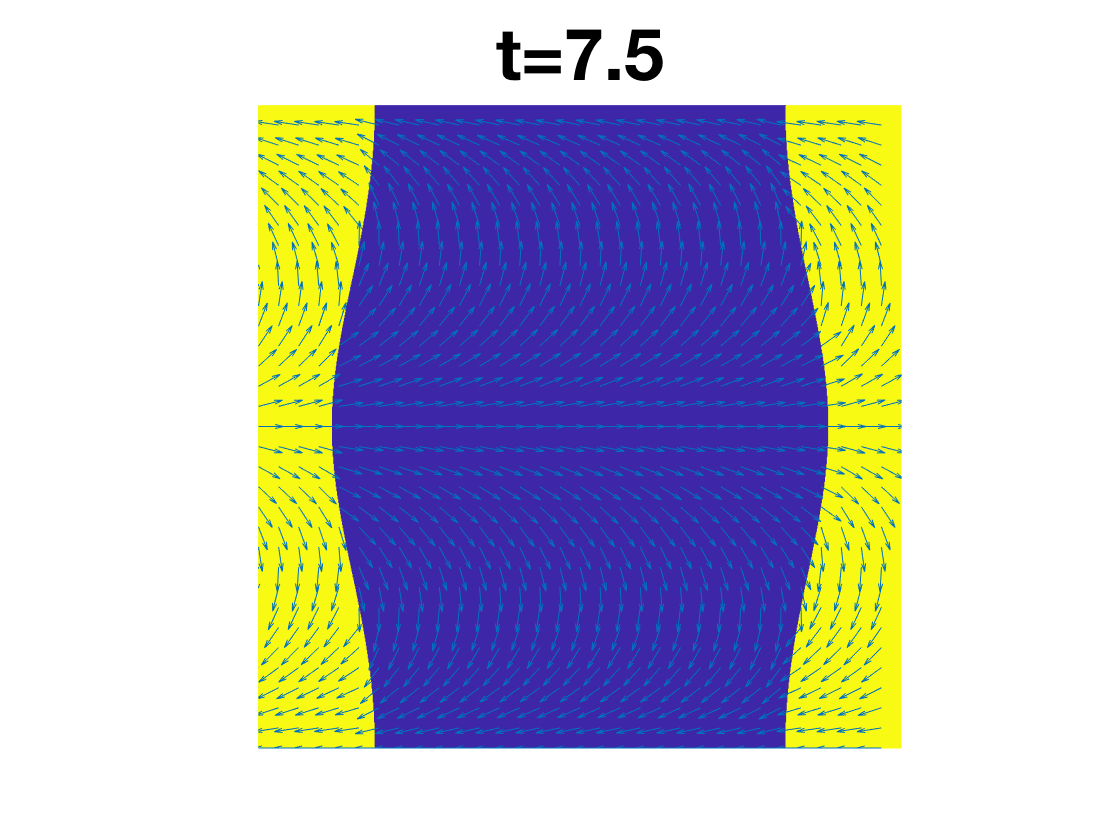}
\includegraphics[scale=0.18,clip,trim= 7cm 1cm 7cm 1cm]{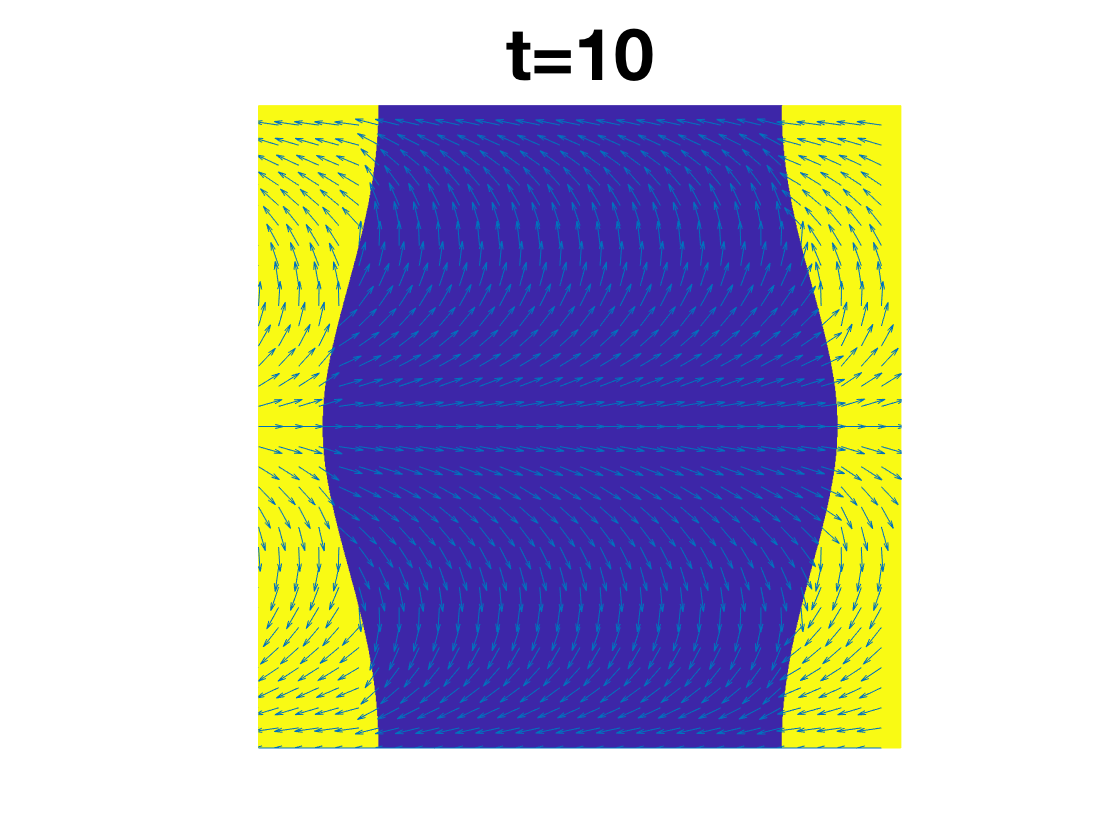} \\
\includegraphics[scale=0.18,clip,trim= 7cm 1cm 7cm 1cm]{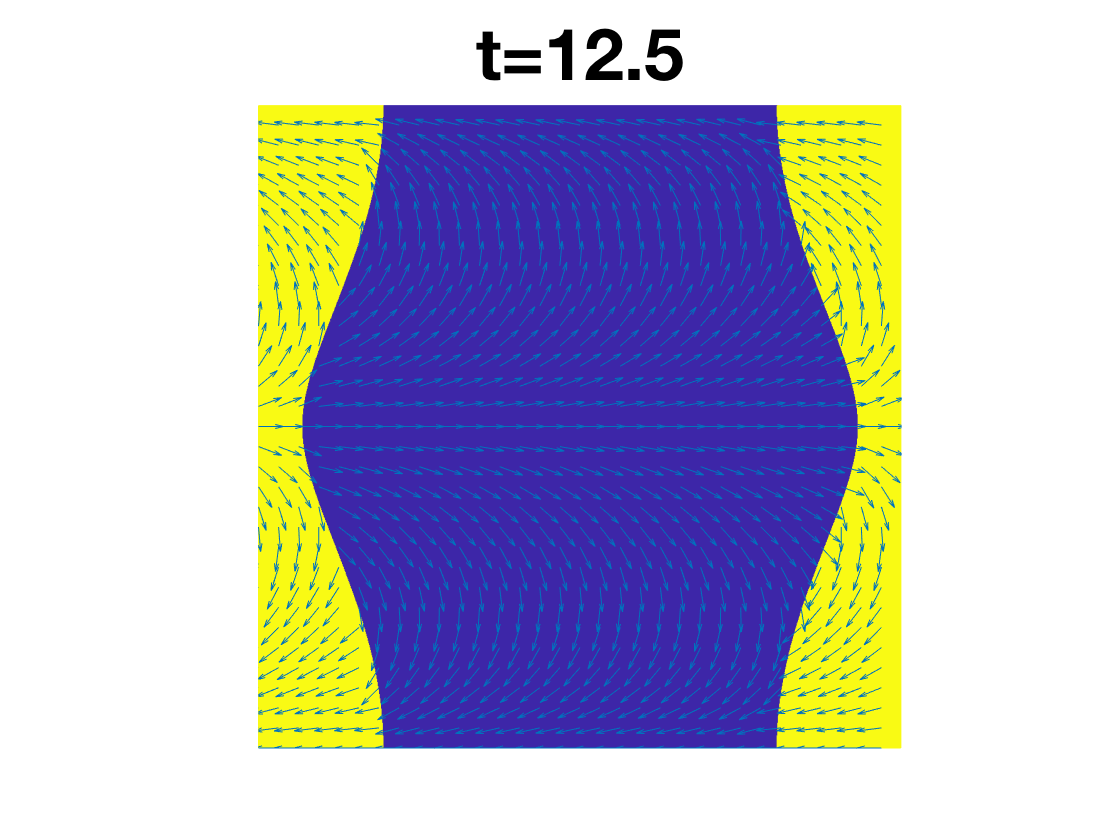}
\includegraphics[scale=0.18,clip,trim= 7cm 1cm 7cm 1cm]{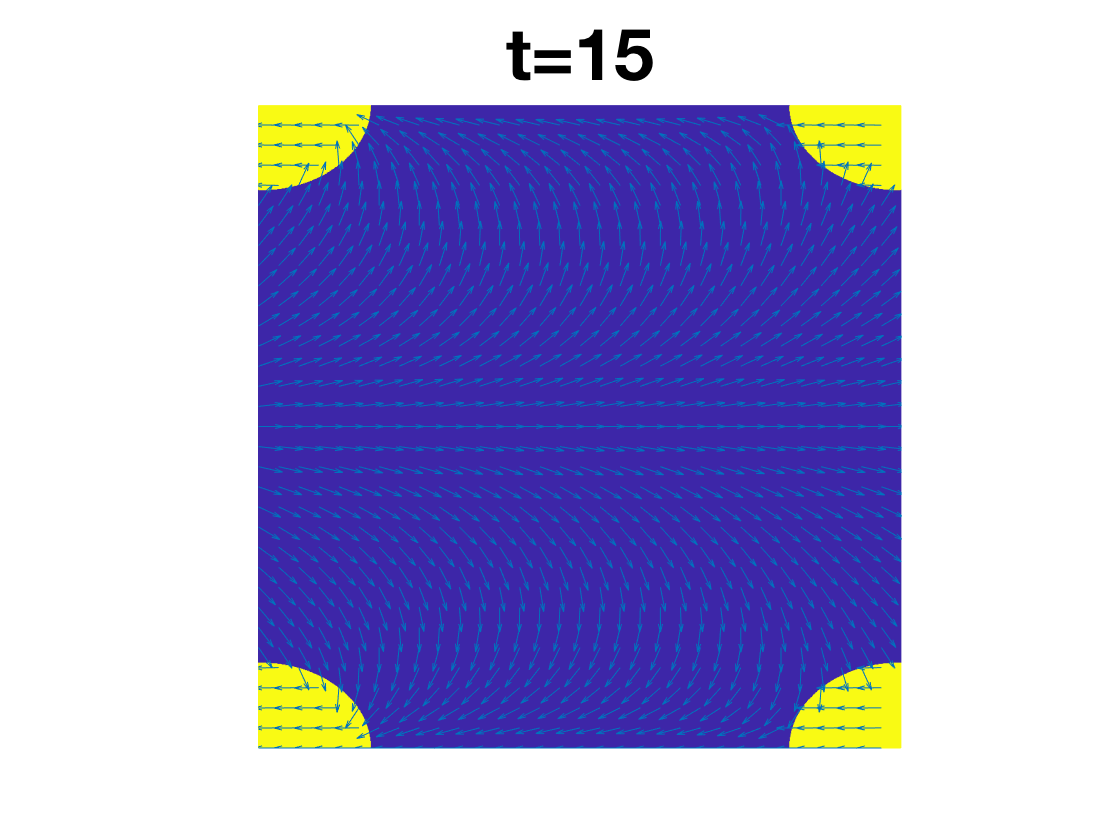}
\includegraphics[scale=0.18,clip,trim= 7cm 1cm 7cm 1cm]{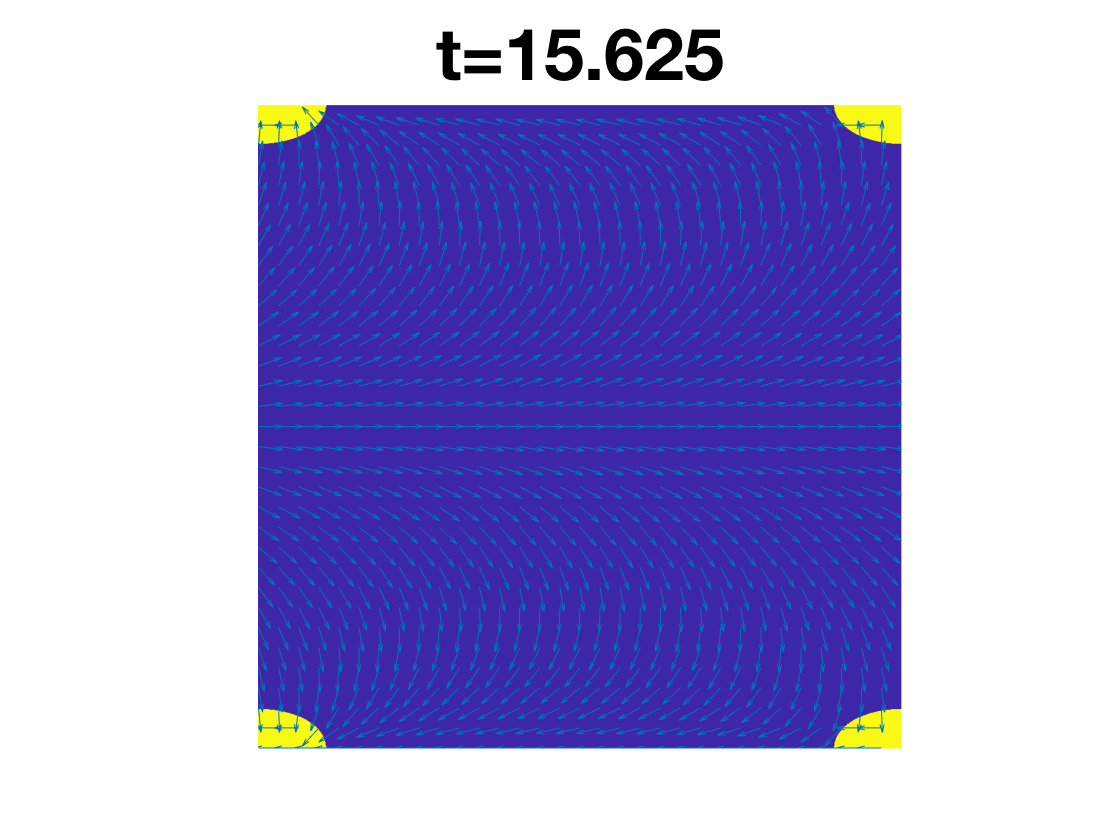}
\caption{Snapshots of the time dynamics on the flat torus where the initial condition has two closed line defects  and winding number 1. See \S \ref{s:FT}.}\label{fig:flat_torus_dynamic_winding_1}
\end{figure}

\subsection{Sphere} \label{s:Sphere}
Let $\Omega = \mathbb S^2$ be a sphere with radius $1$. Since the sphere is a geodesically convex manifold, there is no local minimizer for the energy in \eqref{eq:GL}. In this experiment, we set the grid size $dx = 0.05$, time step size $\tau = 0.005$, accuracy $\varepsilon = 10^{-6}$, the order for quadrature points $p=3$, and the band width $0.5532$  (see Table~\ref{tab:Estimatebw}). After using the closest point representation, the number of degrees of freedom is $136,114$. The number of Fourier modes is $54 \times 54 \times 54$ (see Table~\ref{tab:Fourier_mode}). The initial condition is given by two random matrices in $O(3)$,  one in $SO(3)$ and the other  in $SO^{-}(3)$, 
\begin{align*}
A(x,y,z) = \begin{cases}
\begin{bmatrix}
0.392227 &   0.706046 & 0.046171 \\
   0.655478 &  0.031833 &  0.097132\\ 
   0.171187 & 0.276923 & 0.82346
\end{bmatrix},   & \text{if} \ x<0, \ y<0, \ \text{and} \  z>0 \\
\begin{bmatrix}
0.699077 &  0.547216 &  0.257508 \\
   0.890903 &   0.138624 &   0.840717 \\
  0.959291 &  0.149294&  0.254282
\end{bmatrix}  & \text{otherwise} 
\end{cases}.
\end{align*}
The initial condition is illustrated in Figure~\ref{fig:sphere_dynamic}. Figure~\ref{fig:sphere_dynamic} displays the time dynamics of the line defect on the sphere. The defect evolves toward a circle on the surface as it shrinks. We also observe that the velocity is larger when the circle is smaller. 

\begin{figure}[t!]
\centering
\includegraphics[scale=0.46,clip,trim= 5cm 0cm 5cm 0cm]{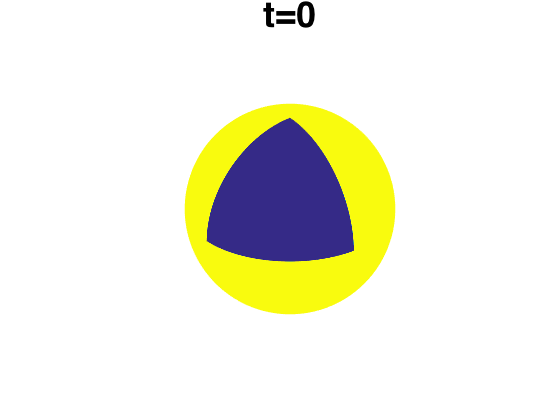}
\includegraphics[scale=0.46,clip,trim= 5cm 0cm 5cm 0cm]{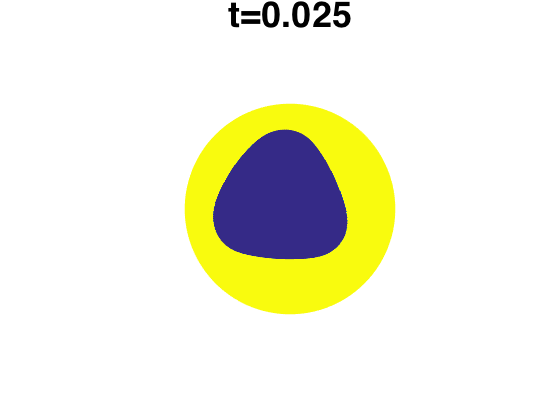}
\includegraphics[scale=0.46,clip,trim= 5cm 0cm 5cm 0cm]{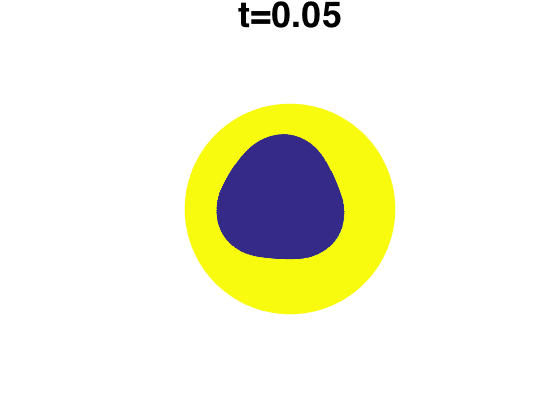} \\ 
\includegraphics[scale=0.46,clip,trim= 5cm 0cm 5cm 0cm]{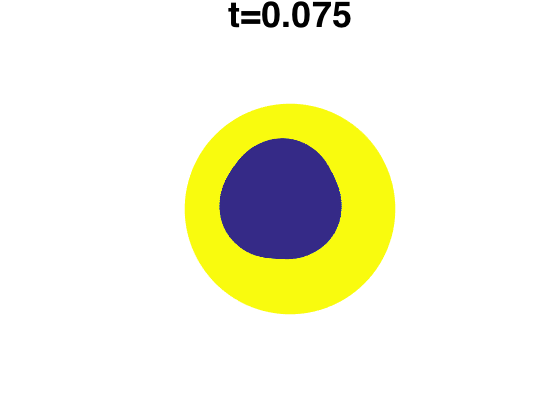}
\includegraphics[scale=0.46,clip,trim= 5cm 0cm 5cm 0cm]{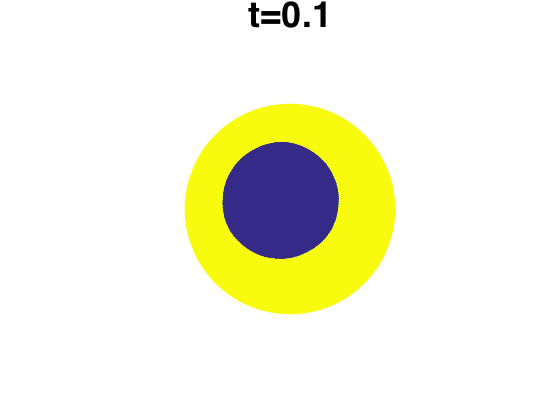}
\includegraphics[scale=0.46,clip,trim= 5cm 0cm 5cm 0cm]{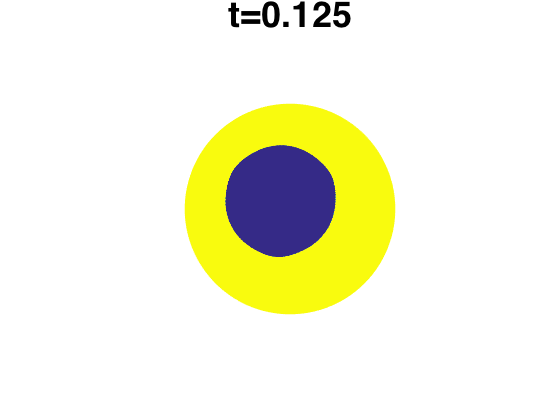} \\
\includegraphics[scale=0.46,clip,trim= 5cm 0cm 5cm 0cm]{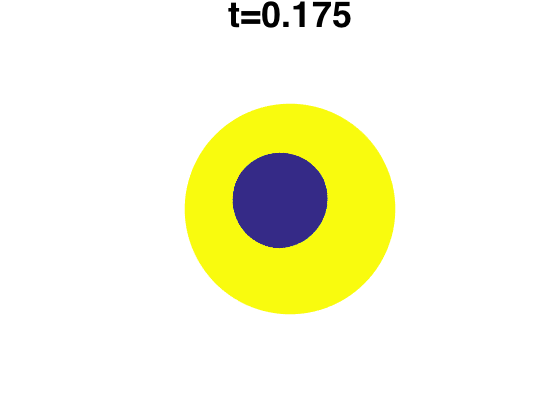}
\includegraphics[scale=0.46,clip,trim= 5cm 0cm 5cm 0cm]{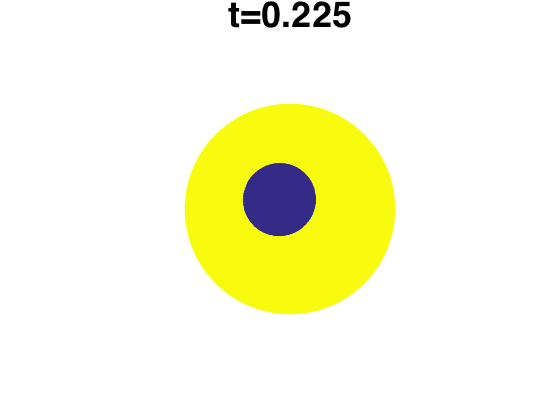}
\includegraphics[scale=0.46,clip,trim= 5cm 0cm 5cm 0cm]{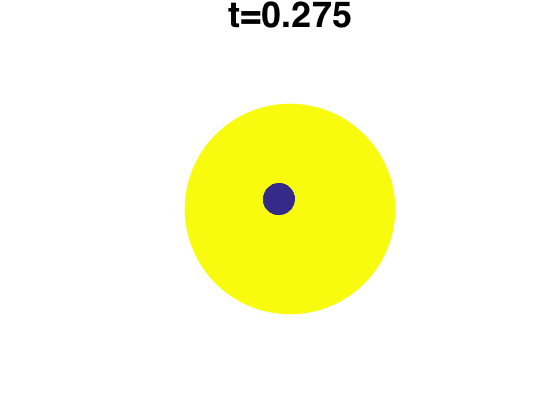}
\caption{Snapshots of the time dynamics of a line defect on a sphere. See \S\ref{s:Sphere}.}\label{fig:sphere_dynamic}
\end{figure}

\subsection{Peanut surface} \label{s:Peanut}
We consider a ``peanut surface'' with the parametric equation depends on $(t,\theta)$ given by:
\begin{align*}
x(t,\theta) &= (3 t-t^3),\\
y(t,\theta) &= \frac{1}{2}\sqrt{(1+x^2)(4-x^2)}\cos(\theta),\\
z(t,\theta) &= \frac{1}{2}\sqrt{(1+x^2)(4-x^2)}\sin(\theta), 
\end{align*} 
with $t \in [-1,1]$ and $\theta \in [0,2\pi]$. This surface is generated by rotating a half peanut curve, illustrated in Figure~\ref{fig:half_peanut_curve}, around the $x-$axis. 

\begin{figure}[t]
\centering
\includegraphics[scale=0.2]{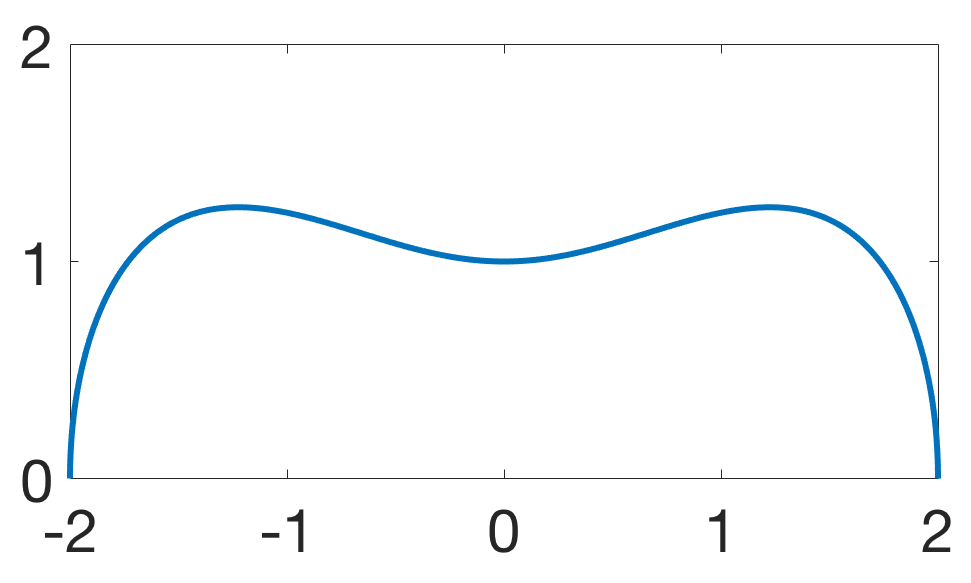}
\caption{Illustration of the half peanut curve. The peanut surface is generated by rotating this half peanut curve about the $x-$axis. See \S\ref{s:Peanut}.} \label{fig:half_peanut_curve}
\end{figure}

For this surface of revolution, we can consider the closest point representation for the half peanut curve in 2-dimensional space and then rotate all the grid points with corresponding closest points to the 3-dimensional space about the $x-$axis. In this experiment, we set the grid size $dx = 0.04$, rotational angle size $d\theta = \frac{\pi}{30}$, time step size $\tau = 0.032$, accuracy $\varepsilon = 10^{-6}$, the order for quadrature points $p=4$, and the band width $1.3994$ (see Table~\ref{tab:Estimatebw}). After using closest point representation, the number of degrees of freedom is $262,144$. The number of Fourier modes is $84 \times 84 \times 84$ (see Table~\ref{tab:Fourier_mode}).

Using the same matrices as in Section~\ref{s:Sphere}, the initial condition is given by 
\begin{align*}
A(x,y,z) = \begin{cases}
\begin{bmatrix}
0.392227 &   0.706046 & 0.046171 \\
   0.655478 &  0.031833 &  0.097132\\ 
   0.171187 & 0.276923 & 0.82346
\end{bmatrix}   & \text{if}  \  x>\sqrt{(y^2+z^2)(y^2+0.1)/1.5} \\
\begin{bmatrix}
0.699077 &  0.547216 &  0.257508 \\
   0.890903 &   0.138624 &   0.840717 \\
  0.959291 &  0.149294&  0.254282
\end{bmatrix} & \text{otherwise} 
\end{cases} . 
\end{align*}
The initial condition is illustrated in Figure~\ref{fig:peanut_initial}. 

We then apply Algorithm~\ref{a:MBO} for this initial condition on the peanut surface. Figure~\ref{fig:peanut_dynamic} displays snapshots of the time dynamics of the line defect. Figure~\ref{fig:peanut_final} displays views of the final state from different angles. The interface converges to a closed geodesic on the surface, which locally minimizes the perimeter. 

\begin{figure}[t!]
\centering
\includegraphics[scale=0.35,clip,trim= 3cm 0cm 3cm 0cm]{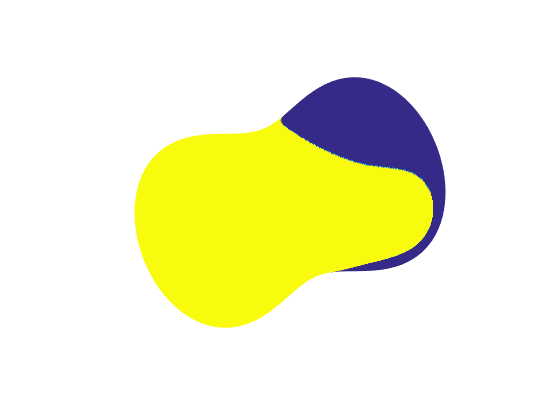}
\includegraphics[scale=0.25,clip,trim= 1cm 0cm 1cm 0cm]{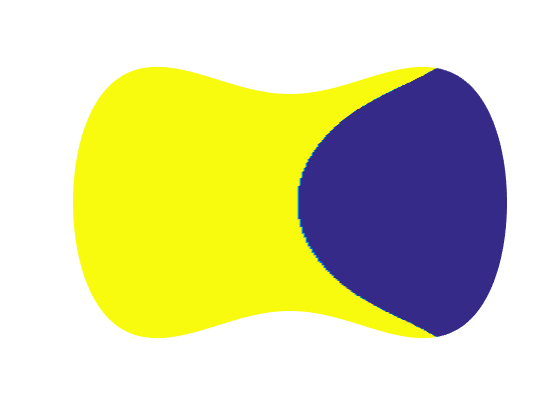}
\includegraphics[scale=0.25,clip,trim= 1cm 0cm 1cm 0cm]{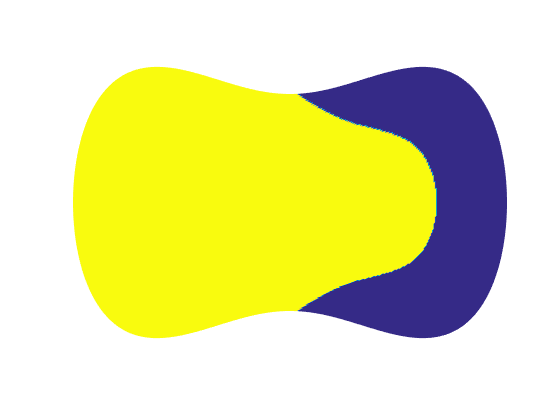}
\caption{{\bf (left)} Initial condition on the peanut, {\bf (center)} a vertical view of the initial condition, and {\bf (right)} a front view of the initial condition. See \S\ref{s:Peanut}.}\label{fig:peanut_initial}
\end{figure}

\begin{figure}[t!]
\begin{center}
\ \vspace{.7cm}

\includegraphics[scale=0.4,clip,trim= 4cm 0cm 4cm 0cm]{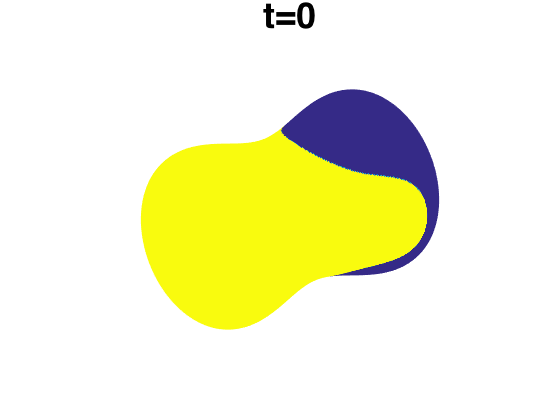}
\includegraphics[scale=0.4,clip,trim= 4cm 0cm 4cm 0cm]{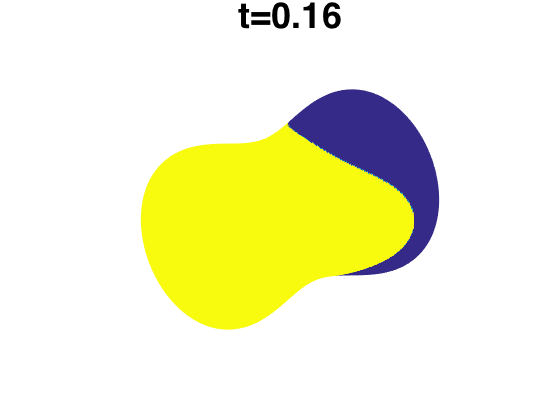}
\includegraphics[scale=0.4,clip,trim= 4cm 0cm 4cm 0cm]{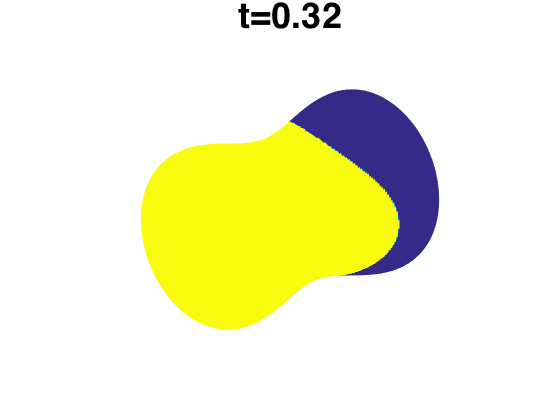}\\
\vspace{.7cm}

\includegraphics[scale=0.4,clip,trim= 4cm 0cm 4cm 0cm]{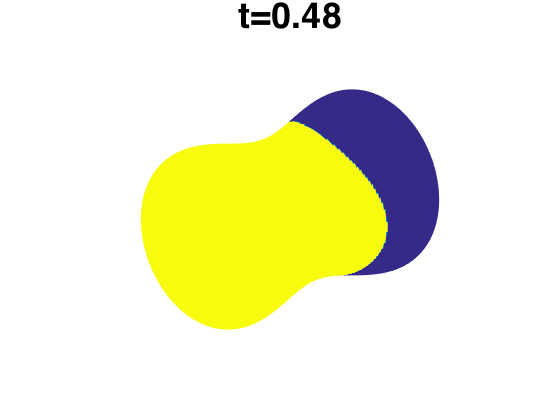}
\includegraphics[scale=0.4,clip,trim= 4cm 0cm 4cm 0cm]{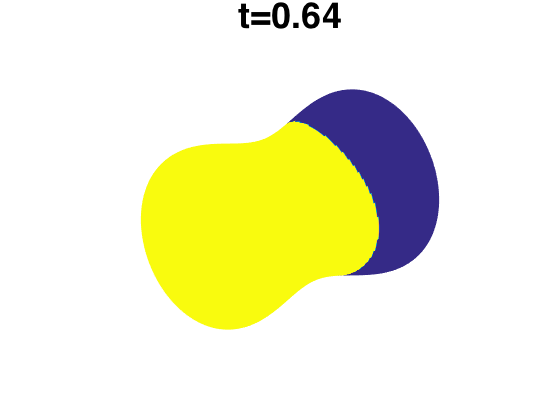}
\includegraphics[scale=0.4,clip,trim= 4cm 0cm 4cm 0cm]{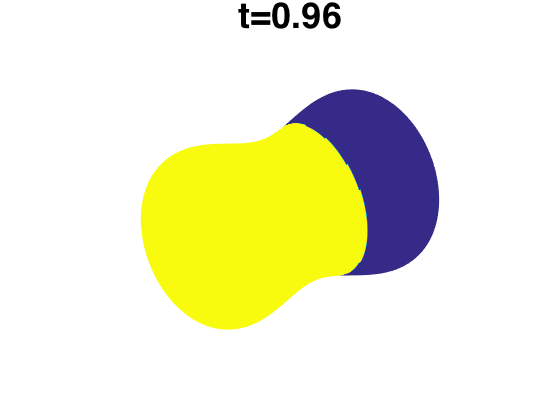}\\
\vspace{.7cm}

\includegraphics[scale=0.4,clip,trim= 4cm 0cm 4cm 0cm]{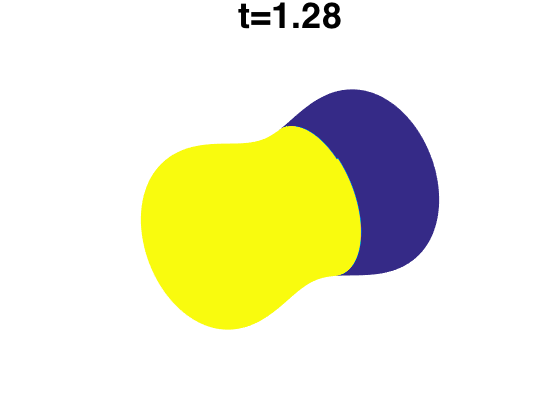}
\includegraphics[scale=0.4,clip,trim= 4cm 0cm 4cm 0cm]{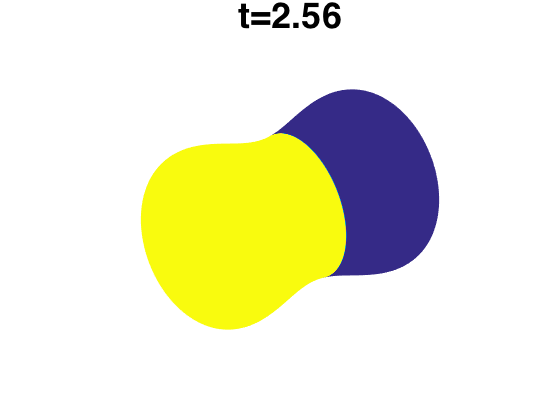}
\includegraphics[scale=0.4,clip,trim= 4cm 0cm 4cm 0cm]{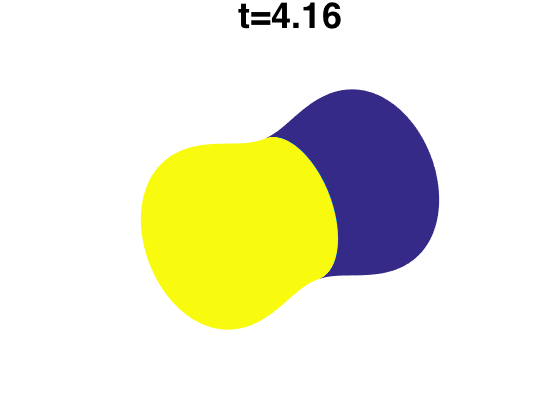}
\caption{Snapshots of the time dynamics of a line defect on the peanut surface. See \S\ref{s:Peanut}.}\label{fig:peanut_dynamic}
\end{center}
\end{figure}

\begin{figure}[t]
\includegraphics[scale=0.3]{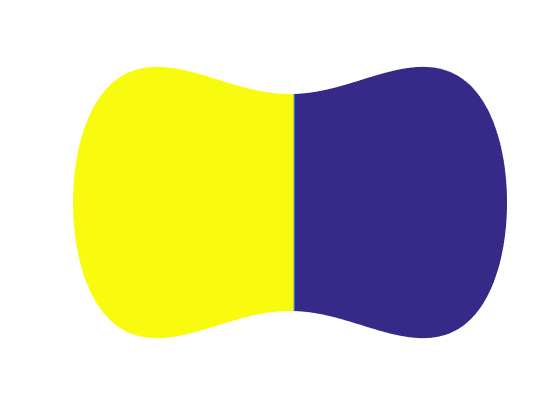}
\includegraphics[scale=0.3]{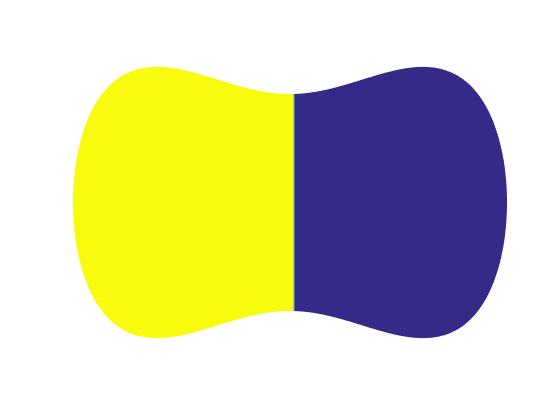}
\caption{{\bf (left)} A vertical view of the equilibrium state.  {\bf (right)} A front view of the equilibrium state. See \S\ref{s:Peanut}.}\label{fig:peanut_final}
\end{figure}

\section{Volume constrained problem} \label{s:VolConst}
In this section, we consider a volume constrained flow, where we constrain
\[ \textrm{vol}( \{ x \colon A(x) \in SO(n) \} ) = V.  \]
To approach this problem numerically, we modify Algorithm \ref{a:MBO} as follows.
We solve the diffusion or surface diffusion equation using the algorithm introduced in Section~\ref{s:Imp}. The projection of $A(\tau,x)$ to $O_n$ is given in Lemma~\ref{l:proj}. Note that, in Algorithm \ref{a:MBO}, $A(\tau,x)$ is projected to $SO(n)$ if $\det(A(\tau,x))>0$ and  $SO^{-}(n)$ if $\det A(\tau,x)<0$.  However, to constrain the volume, one may need to reassign some $A(\tau,x)$ with $\det A(\tau,x)>0$ to a matrix in $SO^{-}(n)$ or  reassign some $A(\tau,x)$ with $\det(A(\tau,x))<0$ by a matrix in $SO(n)$. 
For this purpose, we give the following lemma.
\begin{lem}\label{l:projso-}
Let $A\in \mathbb{R}^{n \times n}$ have a singular value decomposition, $A = U\Sigma V^t$ where $\Sigma$ is a diagonal matrix with diagonal entries $\sigma_i$ in descending order.  Let $D_n$ be the diagonal matrix with diagonal entries $1$ everywhere except in the $n-$th position, where it is $-1$. If $\det A>0$, then $C^{\star}=UD_nV^t$ attains the minimum in 
\begin{align} \label{eq:minso2so-}
dist^2(SO^{-}(n),A) = \min_{C\in SO^{-}(n)} \|C-A\|_F^2 = \sum_{i=1}^n
(\sigma_i-1)^2+4 \sigma_{\min}.
\end{align}
Similarly, if $\det A <0$, then $C^{\star}=UD_nV^t$ attains the minimum in 
\begin{align}
dist^2(SO(n),A) = \min_{C\in SO(n)} \|C-A\|_F^2 = \sum_{i=1}^n
(\sigma_i-1)^2+4 \sigma_{\min}.
\end{align}
\begin{proof}
Assume $\det A >0$ and consider \eqref{eq:minso2so-}. 
Note that $\|C-A\|_F^2 = \|C-U\Sigma V^t\|_F^2 = \|U^tCV-\Sigma \|_F^2 $. 
Since $C \in SO^{-}(n)$, defining $B = U^tCV$, we have $BB^t =  I $ which indicates that $B\in O(n)$. Furthermore, since 
$\det(A) = \det(U)\det(V^t)\det(\Sigma) >0$ 
and $\det(U)\det(V) = 1$, we have $\det(B)=\det(U^t)\det(C)\det(V) = -1$ which implies that $B\in SO^-(n)$. Then, the problem in  \eqref{eq:minso2so-} is equivalent to 
\begin{align} \label{eq:minso2so-2}
\min_{B\in SO^{-}(n)} \|B-\Sigma\|_F^2.
\end{align}
We have 
$\|B-\Sigma\|_F^2 = \|B\|_F^2+\|\Sigma\|_F^2-2\langle B,\Sigma\rangle_F 
= n+\sum_{i\in [n]} \sigma_i^2-2 \sum_{i\in[n]} B_{ii}\sigma_i$. 
Since $B\in SO^-(n)$, $\textrm{tr}(B) = \sum_{i\in[n]} \lambda_i(B)\leq n-2 $ and $B_{ii} \in [-1,1], \forall i\in [n]$. 
Since $\sigma_i$ are non-increasing,  the maximum of $ \sum_{i\in[n]} B_{ii}\sigma_i$  is attained when $B_{ii} = 1, \  \forall i\in [n-1]$ and $B_{nn} = -1$. 
Then, we have that $D_n$ attains the minimum in \eqref{eq:minso2so-2} which implies that 
$UD_nV^t$ attains the minimum in \eqref{eq:minso2so-} which has value $\sum_{i=1}^n
(\sigma_i-1)^2+4 \sigma_{\min}$.

The result for $\det A<0$ is similarly proved.
\end{proof}
\end{lem}

Similar as the problem in \eqref{e:LinProb}, we consider the volume constrained problem 
\begin{align}\label{e:LinProbVolume}
\min_{B \in  L^\infty(\Omega, O_n)}  & \ L^\tau_A(B) \\
\nonumber
\textrm{s.t.} & \ \text{vol}(\{x\colon B(x) \in SO(n)\} ) = V.
\end{align}
Assume $A\in \mathbb{R}^{n \times n}$ has a singular value decomposition, $A = U\Sigma V^t$ where $\Sigma$ is a diagonal matrix with diagonal entries $\sigma_i$ in a descending order and  $D_n$ is a diagonal matrix with diagonal entries $1$ everywhere except in the $n-$th position, where it is $-1$.
We first define
\[ T^+(A) = \left\lbrace \begin{array}{cc}
UV^t,  & \text{if} \ \ \det A >0, \\
UD_nV^t,  & \text{if} \ \ \det A <0
\end{array} \right.
 \]
and
\[ T^-(A) = \left\lbrace \begin{array}{cc}
UD_nV^t,  & \text{if} \ \ \det A >0, \\
UV^t,  & \text{if} \ \ \det A <0.
\end{array} \right.
 \]
Then, we define 
\[\Delta E(x) = \left\langle T^+\left(A(\tau,x)\right)-T^-\left(A(\tau,x)\right), A(\tau,x)  \right\rangle_F,
\]
\[
\Omega^\lambda_+ = \{ x \colon \Delta E(x) \geq \lambda \},
\] 
and $\Omega^\lambda_- =\Omega \setminus \Omega^\lambda_+ $. Similar as that in \cite{ruuth2003simple}, we treat the volume of $\Omega_+^\lambda$ as a function of $\lambda$, {\it i.e.}, 
$ f(\lambda) = \textrm{vol}(\Omega_+^{\lambda})$, 
and identify the value $\lambda$ such that $f(\lambda) = V$. For the matrix case, we need to redefine the matrix valued field on $\Omega^\lambda_+$ and $\Omega^\lambda_-$.  This can be accomplished pointwise by assigning 
$$
A(\tau,x) \mapsto \begin{cases}
T^+\left(A(\tau,x)\right) & \textrm{if} \ \Delta E(\x) \geq \lambda \\
T^-\left(A(\tau,x)\right) & \textrm{if} \ \Delta E(x)< \lambda
\end{cases}.
$$
The following lemma  shows that this choice is optimal.
\begin{lem}\label{l:optimal}
Assume $\lambda_0$ satisfies  $f(\lambda_0) = V$. Then, 
\[
B^{\star} = \begin{cases}
T^+(A(\tau,x))  & \ \textrm{if}   \  \Delta E(x) \geq \lambda_0 \\
T^-(A(\tau,x))  & \ \textrm{if}  \  \Delta E(x) < \lambda_0
\end{cases}.
\]
attains the minimum in \eqref{e:LinProbVolume}. 
\end{lem}
\begin{proof} Write 
$P_n = \{A\in L^\infty(\Omega,O_n) \colon \text{vol}(\{x\colon A(x) \in SO(n)\} ) = V \}$.
Consider any $B \in P_n$, from Lemma~\ref{l:proj} and Lemma~\ref{l:projso-}, we have $L^\tau_A(B) \geq L^\tau_A(\tilde{B})$ where 
\[
\tilde{B}(x)= \begin{cases}
T^+(A(\tau,x))   & \ \textrm{if}  \ \  \det B(x) =1 \\
T^-(A(\tau,x))   & \ \textrm{if}  \ \  \det B(x) =-1
\end{cases}. 
\]
Now, we only need to prove that $L^\tau_A(\tilde{B}) \geq L^\tau_A(B^{\star})$. Denote 
\begin{align*}
\Omega_1 &=\{x\colon \det(\tilde{B}(x))=1 \  \text{and} \ \det(B^{\star}(x))=-1\} \\
\Omega_2 &= \{x\colon \det(\tilde{B}(x))=-1 \  \text{and} \  \det(B^{\star}(x))=1\}. 
\end{align*}
That is, 
\begin{align*}
&B^{\star}(x) = T^-(A(\tau,x)) \quad \text{and} \quad \tilde{B}(x) = T^+(A(\tau,x)) \quad \forall x \in \Omega_1 \\
&B^{\star}(x) = T^+(A(\tau,x)) \quad \text{and} \quad \tilde{B}(x) = T^-(A(\tau,x)), \quad \forall x \in \Omega_2.
\end{align*}  
Also, we have $\Omega_2 \subset \Omega_+^{\lambda_0}$, $\Omega_1 \subset \Omega_-^{\lambda_0}$, and 
$ \textrm{vol}(\Omega_1) = \textrm{vol}(\Omega_2)$.  
Hence, from the definition of $\Delta E(x)$, we have
\begin{align*}
L^\tau_A(\tilde{B}) - L^\tau_A(B^{\star}) & = \frac{2}{\tau} \int_{\Omega} \langle  e^{\Delta \tau} A, B^{\star} - \tilde{B} \rangle_F  \ dx \\
&= \frac{2}{\tau} \int_{\Omega_1} \langle  e^{\Delta \tau} A, B^{\star} - \tilde{B} \rangle_F  \ dx+ \frac{2}{\tau} \int_{\Omega_2} \langle  e^{\Delta \tau} A, B^{\star} - \tilde{B} \rangle_F  \ dx \\
& = -\frac{2}{\tau} \int_{\Omega_1} \langle  e^{\Delta \tau} A,  \tilde{B}-B^{\star} \rangle_F  \ dx+ \frac{2}{\tau} \int_{\Omega_2} \langle  e^{\Delta \tau} A, B^{\star} - \tilde{B} \rangle_F  \ dx \\
& \geq - \lambda_0  \textrm{vol}(\Omega_1) + \lambda_0  \textrm{vol}(\Omega_2) \\
& = 0 .
\end{align*}  
\end{proof}

We refer to the modification to Algorithm \ref{a:MBO} that replaces  step 2 with the reassignment in Lemma \ref{l:optimal} as the \emph{volume-preserving MBO method}.  
The following proposition to show that this algorithm is unconditionally stable.
\begin{prop} \label{p:VolPreserveStab}
For any $\tau>0$, the functional $E^\tau$, defined in \eqref{eq:Lyapunov}, is non-increasing on the volume-preserving iterates $\{ A_s \}_{s=1}^\infty $, {\it i.e.}, $E^\tau(A_{s+1}) \leq E^\tau(A_s)$. 
\end{prop}
Proposition \ref{p:VolPreserveStab} can be proven similarly  to the proof of Proposition \ref{prop:Lyapunov} using the concavity of $E^\tau$ (see Lemma \ref{l:Jprops}) and 
Lemma \ref{l:optimal}.

To identify the level set with appropriate volume, {\it i.e.}, $f(\lambda) = V$, one could use an  iterative method, {\it e.g.} bisection method, Newton's method, fixed point iteration, or Muller's method. 
However, those methods all either require many iterations or are sensitive to the initial guess. In \cite{xu2016efficient}, Xu {\it et. al.} proposed a new algorithm for the volume preserving scalar $\{ \pm 1\}$ problem based on the Quicksort algorithm (see also \cite{jacobsauction,elsey2016threshold}). For a flat torus, the computational domain is discretized by a uniform mesh so that the volume of $\Omega_+^{\lambda}$ can be approximately expressed as $N_+dx^2$ where $N_+$ is the number of grid points where $ \Delta E \geq \lambda$.  Note that when the initial volume $V$ is given, $N_+$ is fixed. If we sort the value of $\Delta E(x)$ into descending order, we only need to choose the first $N_+$ values and assign the corresponding grid points to the set $\Omega_+$.  Since $A(\tau,x)$ is a matrix-valued function, when $\lambda \neq 0$, there are points $x \in \Omega$ where we need to reassign matrices $A(\tau,x)$ with negative determinate to $SO(n)$ or visa versa. This can be done by $T^+$ or $T^-$.

To illustrate the algorithm, we conduct a numerical experiment on the flat torus. The initial condition has a closed line defect and all parameters are chosen the same as in Figure~\ref{fig:flat_torus_dynamic_shrinking}. We set $\tau = 8dx = 0.0078125$. Figure~\ref{fig:flat_torus_dynamic_volume_preserving} displays snapshots of the time dynamics of the closed line defect on the flat torus preserving volume. The closed line defect evolves towards a circle.

\begin{figure}[ht]
\centering
\includegraphics[scale=0.18,clip,trim= 7cm 1cm 7cm 1cm]{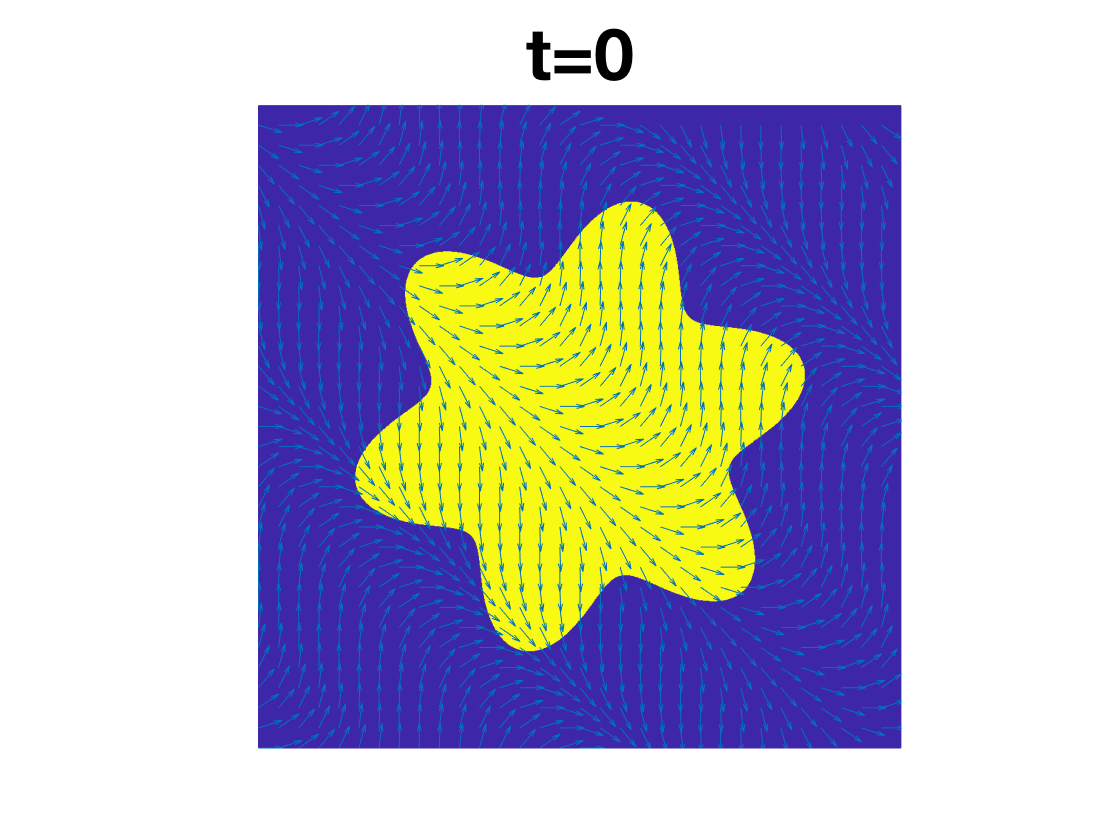}
\includegraphics[scale=0.18,clip,trim= 7cm 1cm 7cm 1cm]{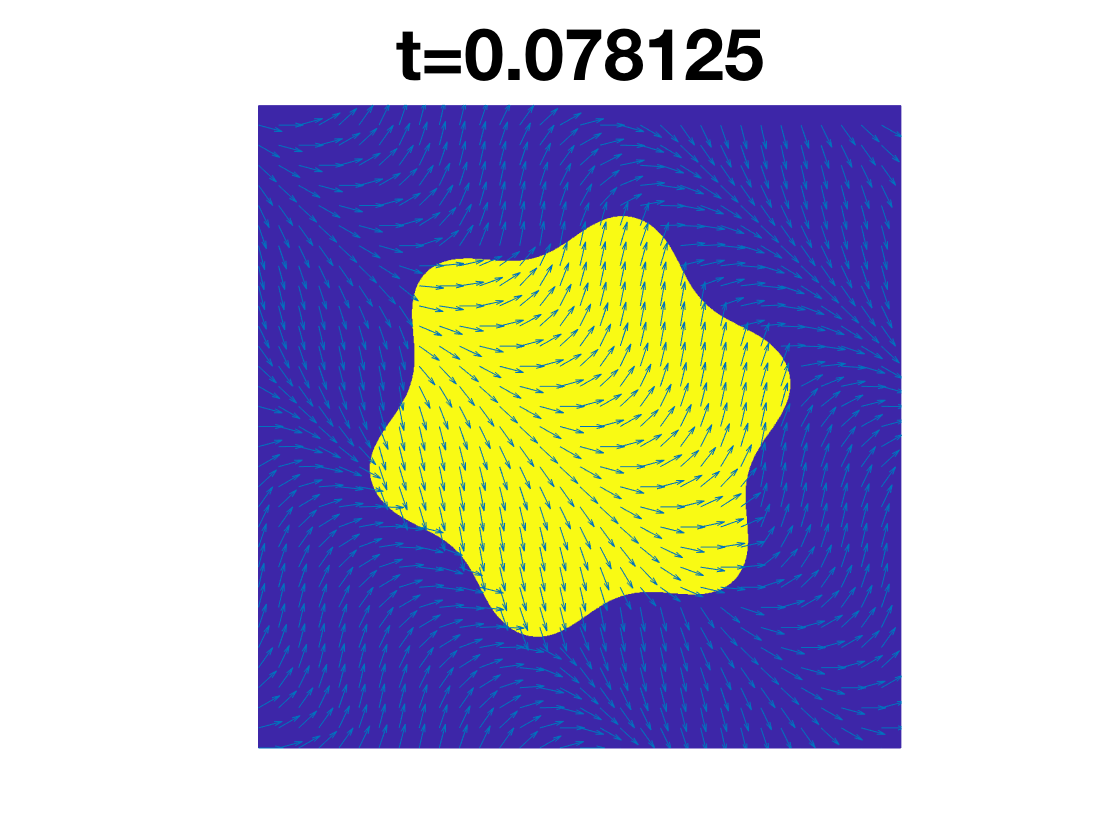}
\includegraphics[scale=0.18,clip,trim= 7cm 1cm 7cm 1cm]{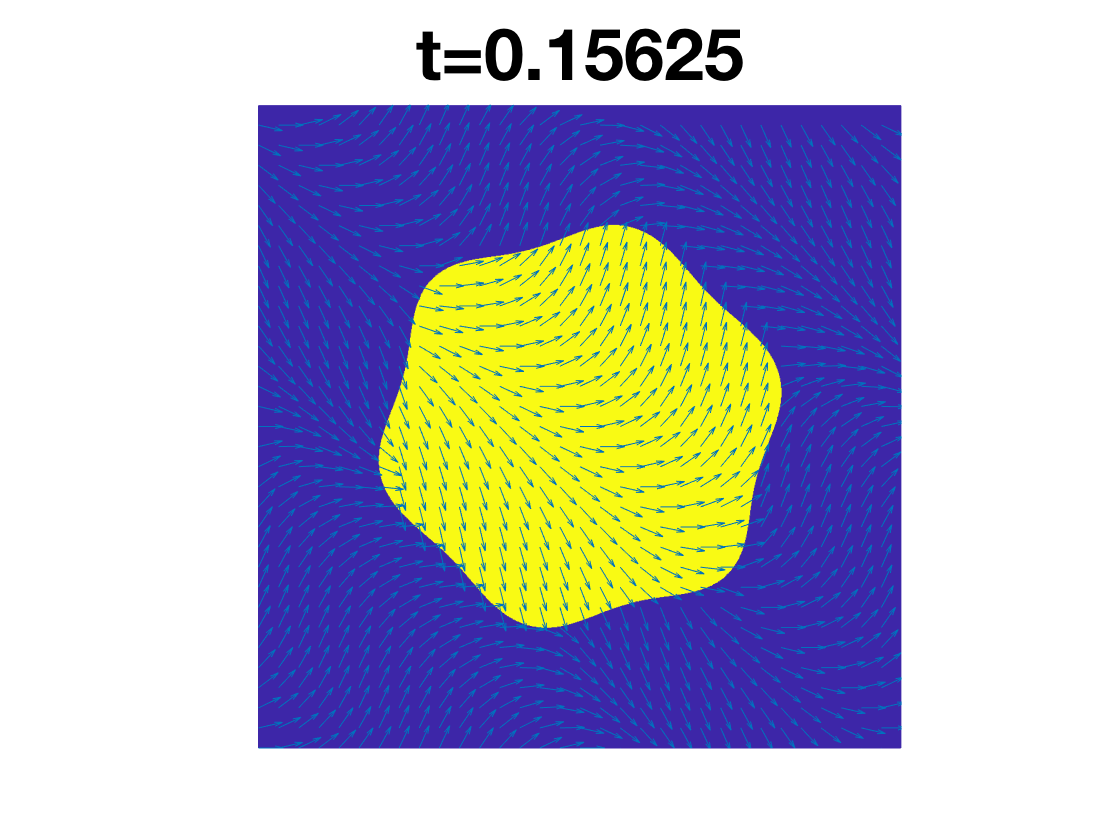}
\includegraphics[scale=0.18,clip,trim= 7cm 1cm 7cm 1cm]{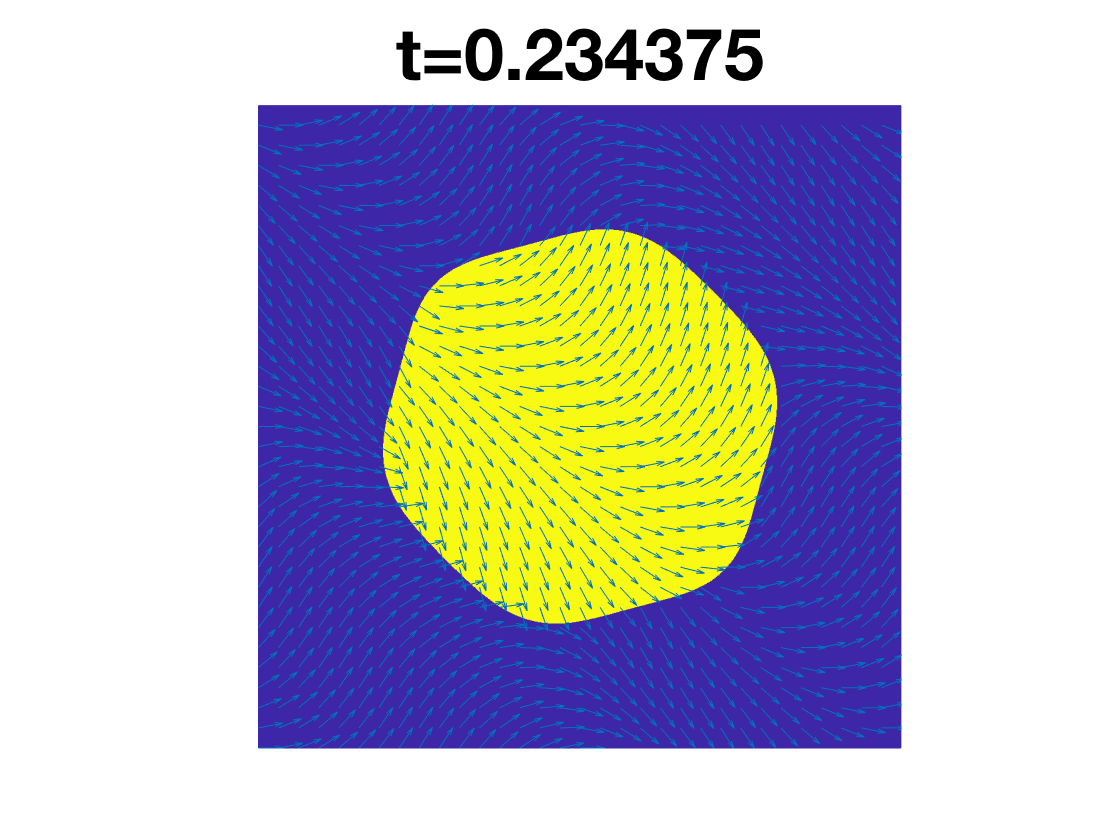}
\includegraphics[scale=0.18,clip,trim= 7cm 1cm 7cm 1cm]{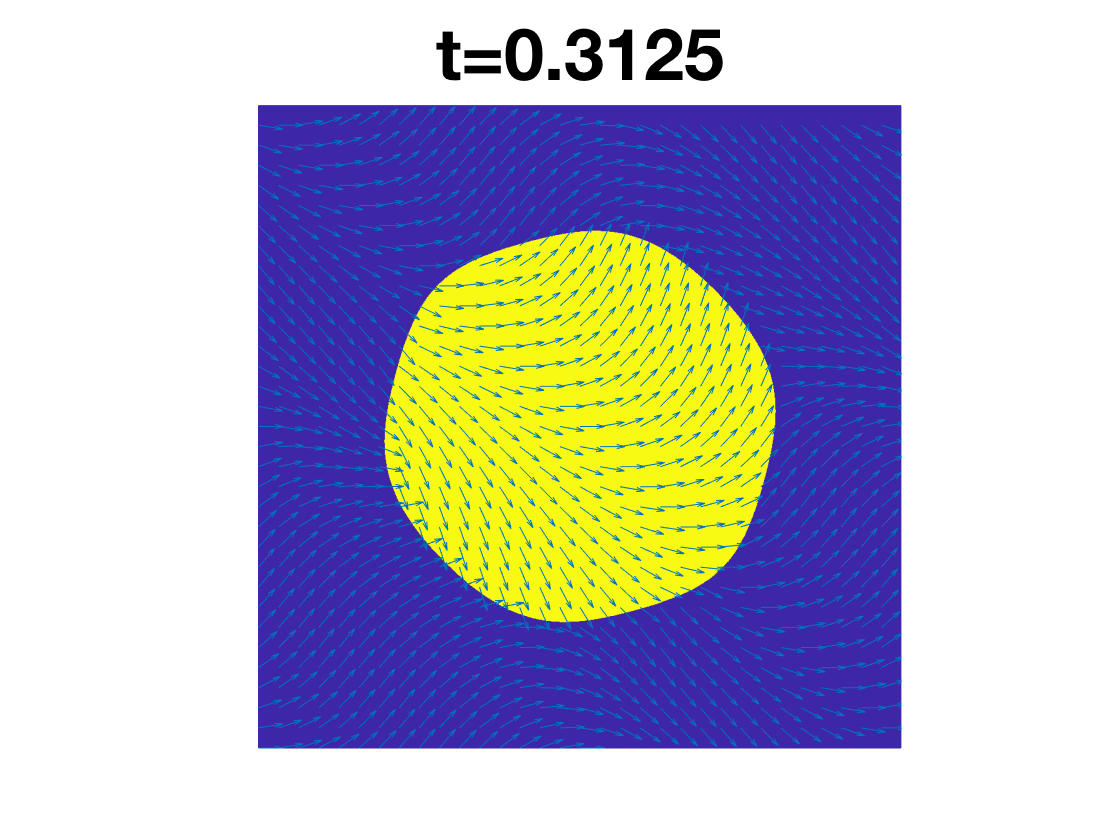}
\includegraphics[scale=0.18,clip,trim= 7cm 1cm 7cm 1cm]{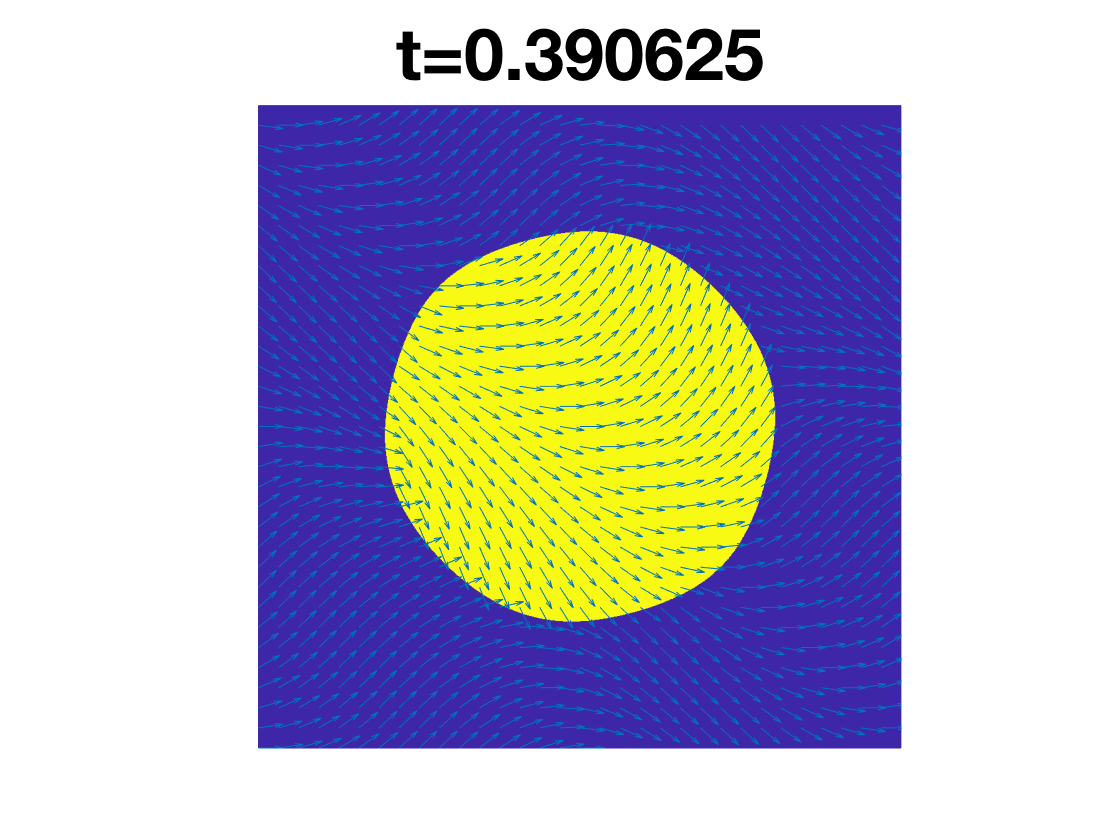}
\includegraphics[scale=0.18,clip,trim= 7cm 1cm 7cm 1cm]{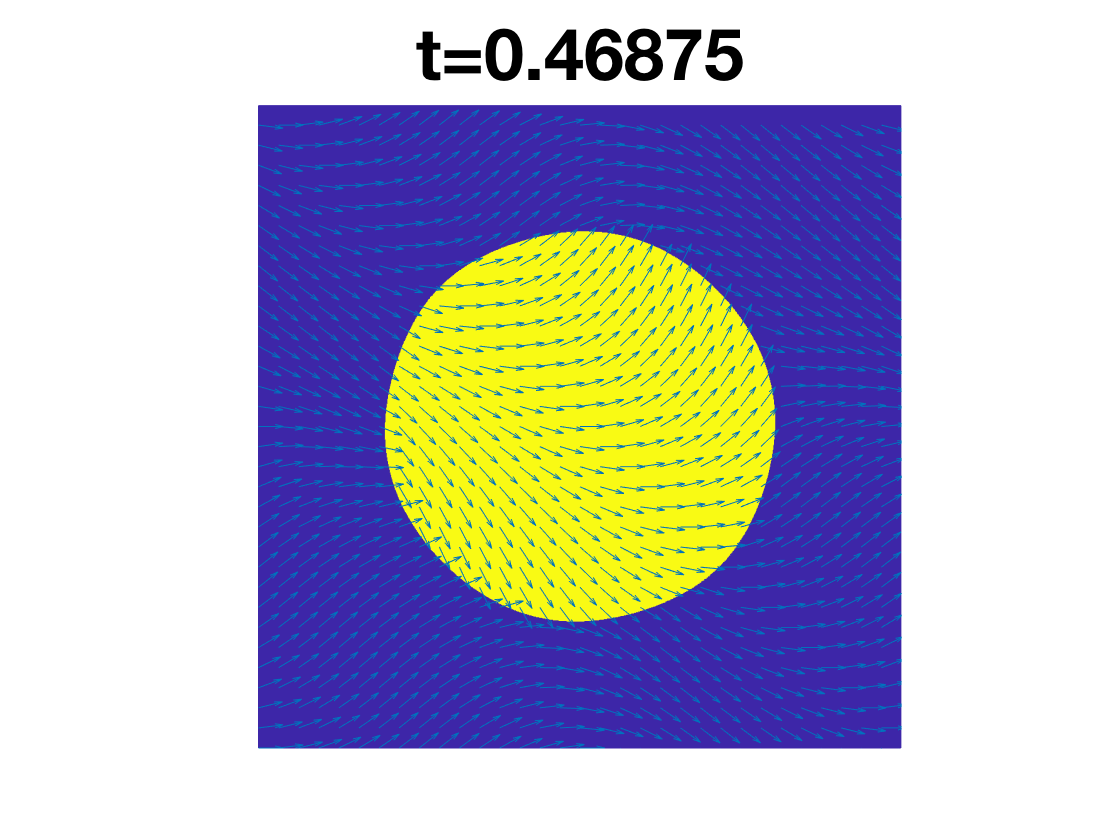}
\includegraphics[scale=0.18,clip,trim= 7cm 1cm 7cm 1cm]{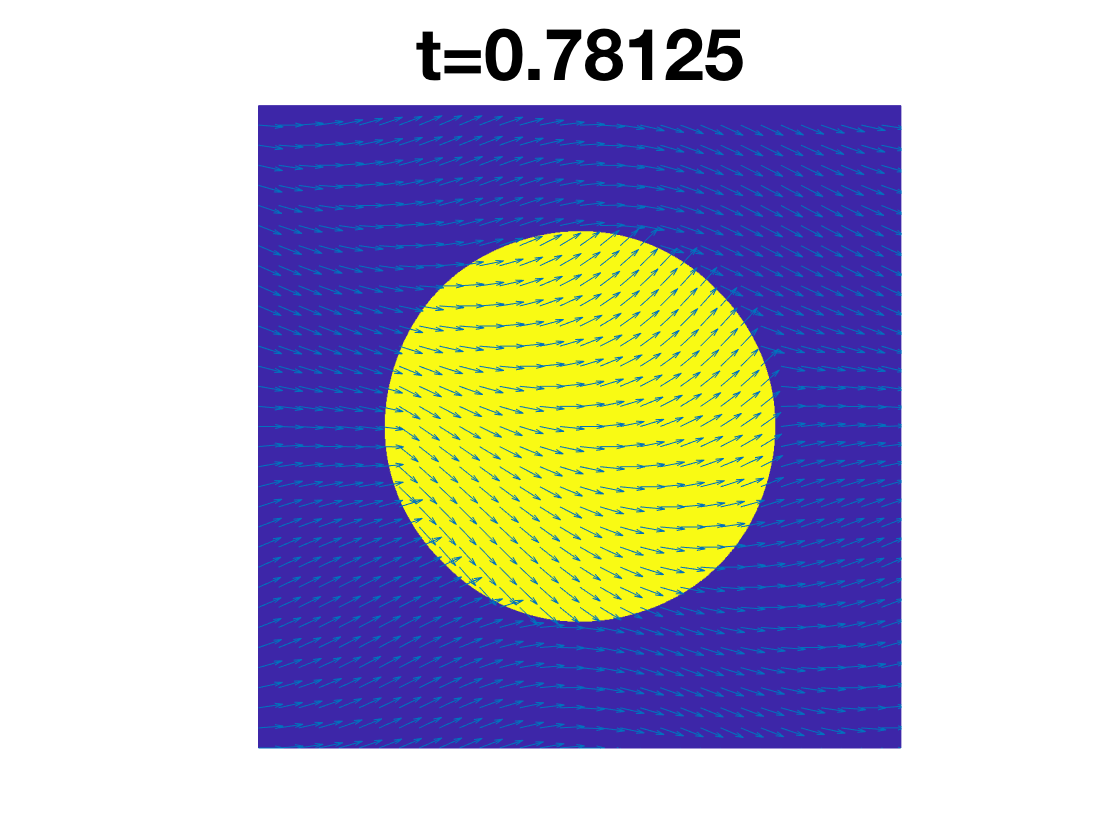}
\includegraphics[scale=0.18,clip,trim= 7cm 1cm 7cm 1cm]{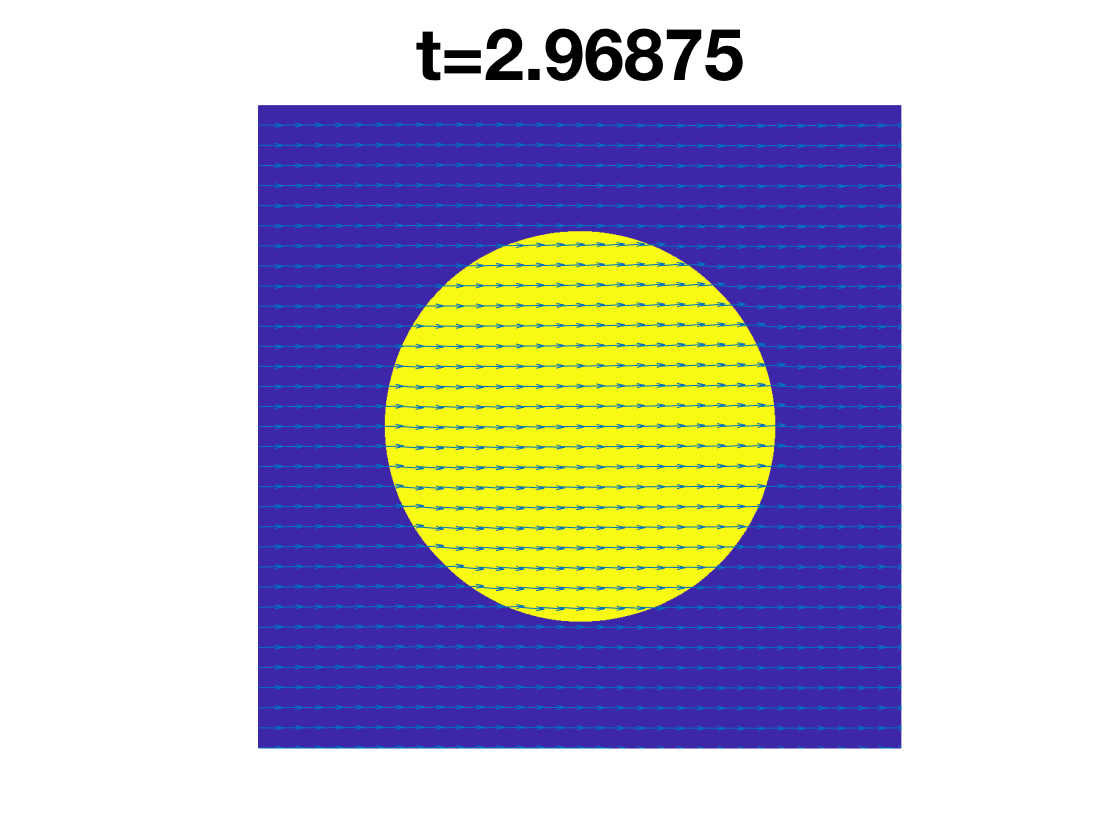}
\caption{Snapshots of the time dynamics for a closed line defect on the flat torus with a volume constraint. See \S\ref{s:VolConst}.}\label{fig:flat_torus_dynamic_volume_preserving}
\end{figure}

\medskip

Finally, we generalize the algorithm in \cite{xu2016efficient} to surfaces. When the surface is not uniformly discretized, there is no $N_+$ in the uniform set as before. We store the integral surface volume $w$ at each point on the surface. After $\Delta E(x)$ is computed, we still use the Quicksort algorithm to sort $\Delta E(x)$ into descending order and store the sort index in a set $\mathcal{S}$ which specifies how the elements of $\Delta E(x)$ were rearranged to obtain the sorted vector. Then we set $V_0 = 0$, and add the contribution of volume from each $x$ to $V_0$ according to the descending order one by one until $V_0\geq V$. The procedure is summarized in Algorithm~\ref{a:quicksort}.

\begin{algorithm}[t!]
\DontPrintSemicolon
\KwIn{Let $A(\tau,x)$ be computed by Algorithm~\ref{a:nufft_surface}, 
the integral volume $w$ at each surface point be given, 
and desired surface volume $V$ be given. }
\KwOut{A parameter $\lambda$ so that $\textrm{vol}(\Omega_+^{\lambda}) = V$.}
 {\bf 1.} Use the Quicksort algorithm to sort $\Delta E(x)$ into a descending order and store the sorted index in $\mathcal{S}$.\\
 {\bf 2.} Set $V_0=0$ and $i=1$.\;
\While{$V_0<V$}{
$V_0 = V_0+w(\mathcal{S}(i))$\\
$i = i+1$
 }
 {\bf 3.} $\lambda =\dfrac{\Delta E(x)(\mathcal{S}(i-1)) +\Delta E(x)(\mathcal{S}(i))}{2}$. 
\caption{A Quicksort based algorithm for volume preserving on surfaces. } 
\label{a:quicksort}
\end{algorithm}

To illustrate the algorithm on surfaces, we consider the volume preserving problem on a sphere. 
The initial condition has a closed line defect and is chosen to be the same as that in Figure~\ref{fig:sphere_dynamic}. 
Figure~\ref{fig:sphere_volume_preserving} displays snapshots of the time dynamics.
The line defect is seen to evolve towards a circle. 
In this experiment, we set 
the grid size $dx = 0.05$, 
time step size $\tau = 0.01$, 
accuracy $\varepsilon = 10^{-6}$, 
the order for quadrature points $p=3$, 
and the band width $0.7823$  (see Table~\ref{tab:Estimatebw}). 
After using the closest point representation, the number of degrees of freedom is $206,026$. 
The number of Fourier modes is $81 \times 81 \times 81$  (see Table~\ref{tab:Fourier_mode}).

\begin{figure}[ht]
\centering
\includegraphics[scale=0.23,clip,trim= 10cm 1cm 10cm 1cm]{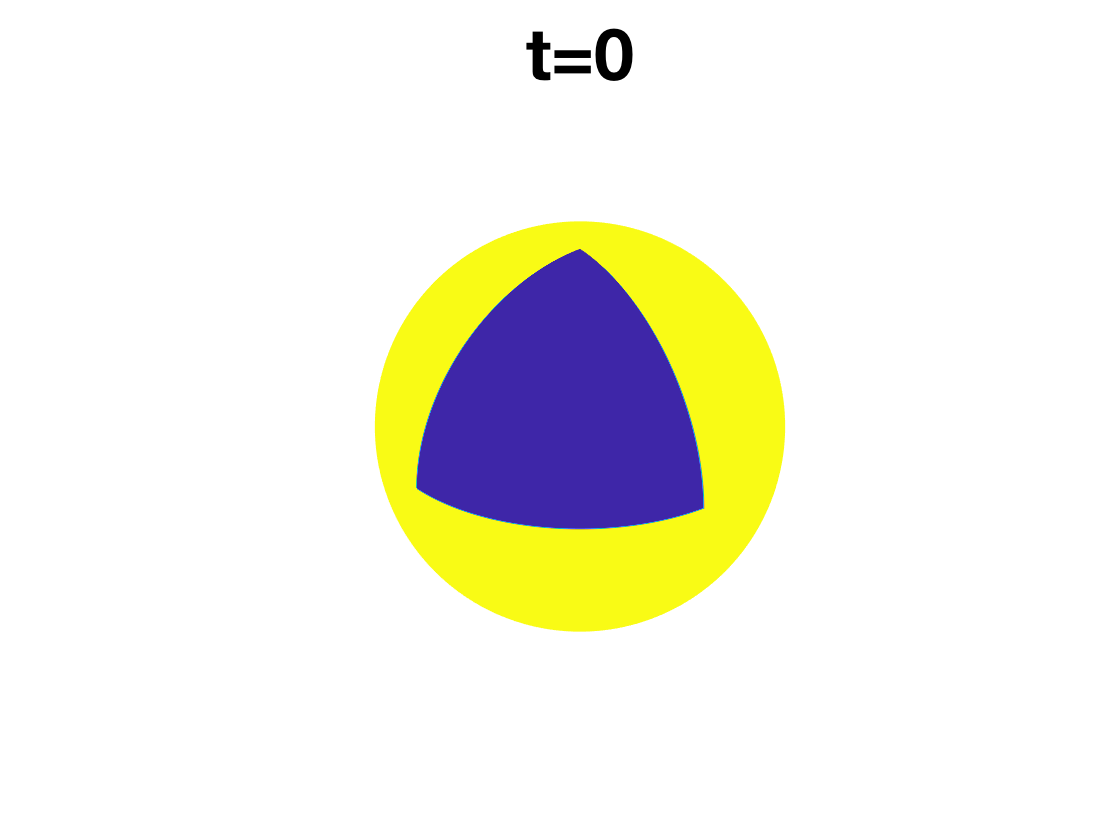}
\includegraphics[scale=0.23,clip,trim= 10cm 1cm 10cm 1cm]{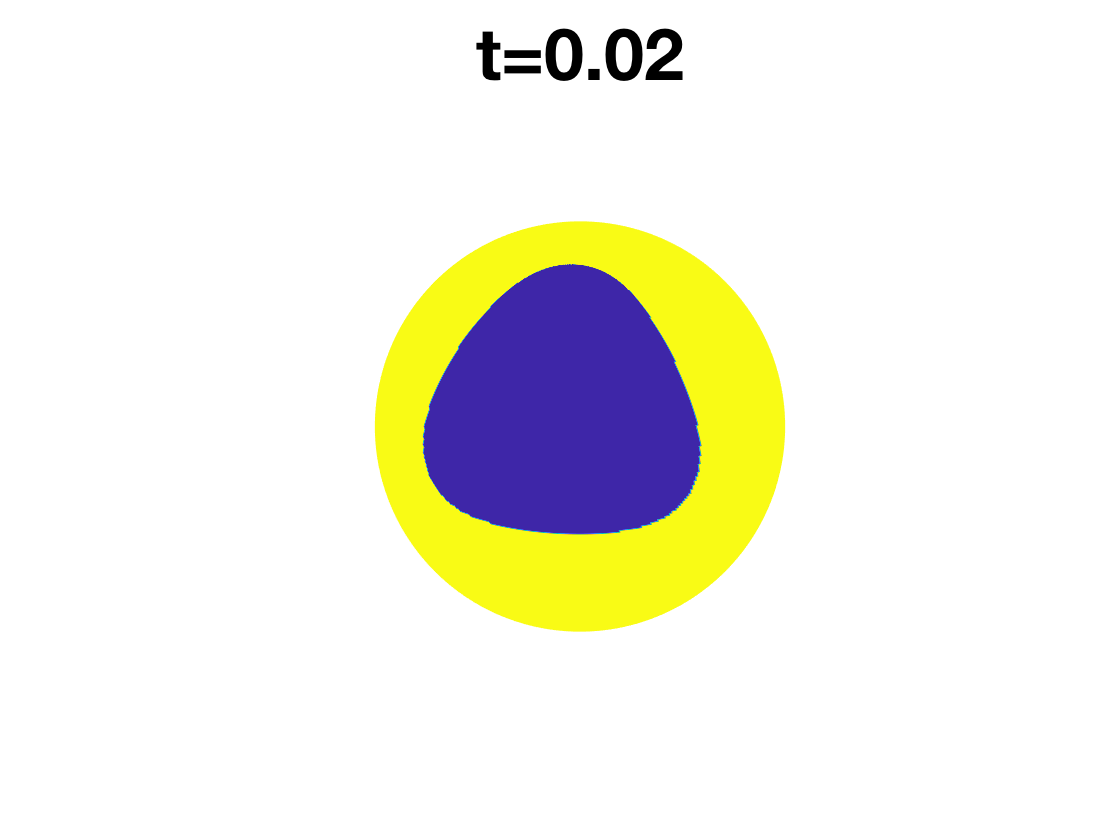}
\includegraphics[scale=0.23,clip,trim= 10cm 1cm 10cm 1cm]{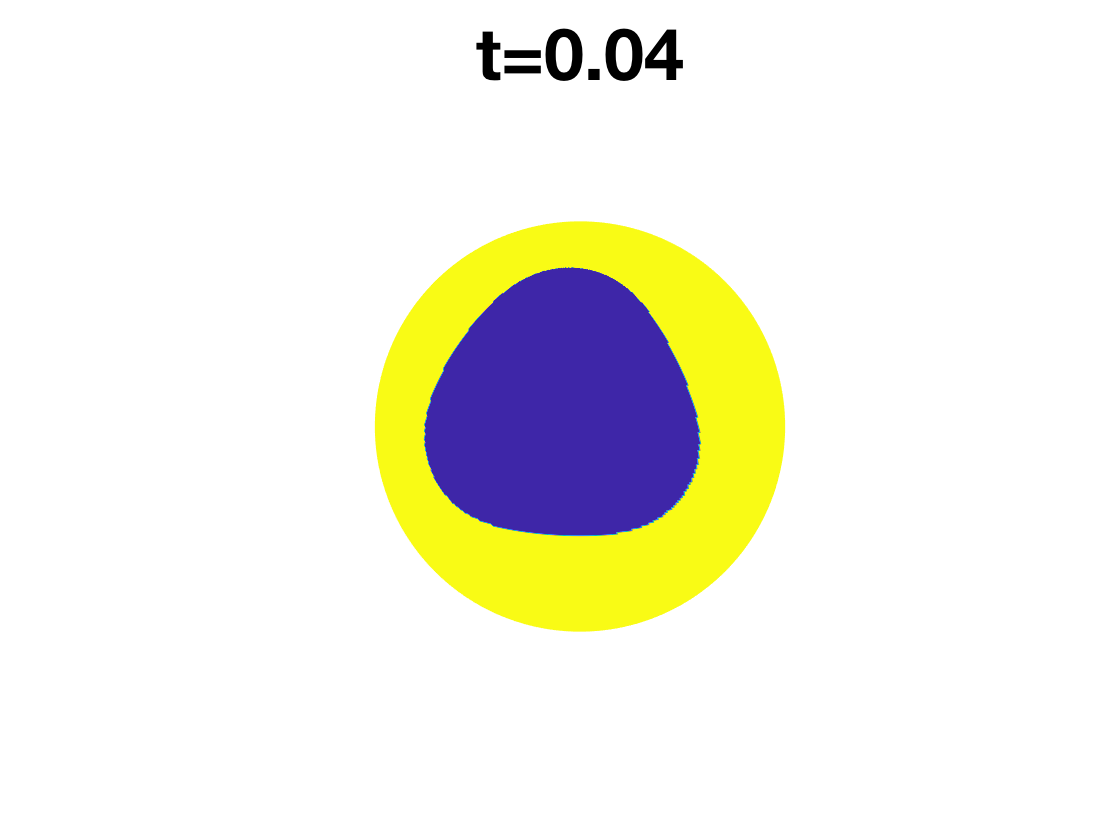}
\includegraphics[scale=0.23,clip,trim= 10cm 1cm 10cm 1cm]{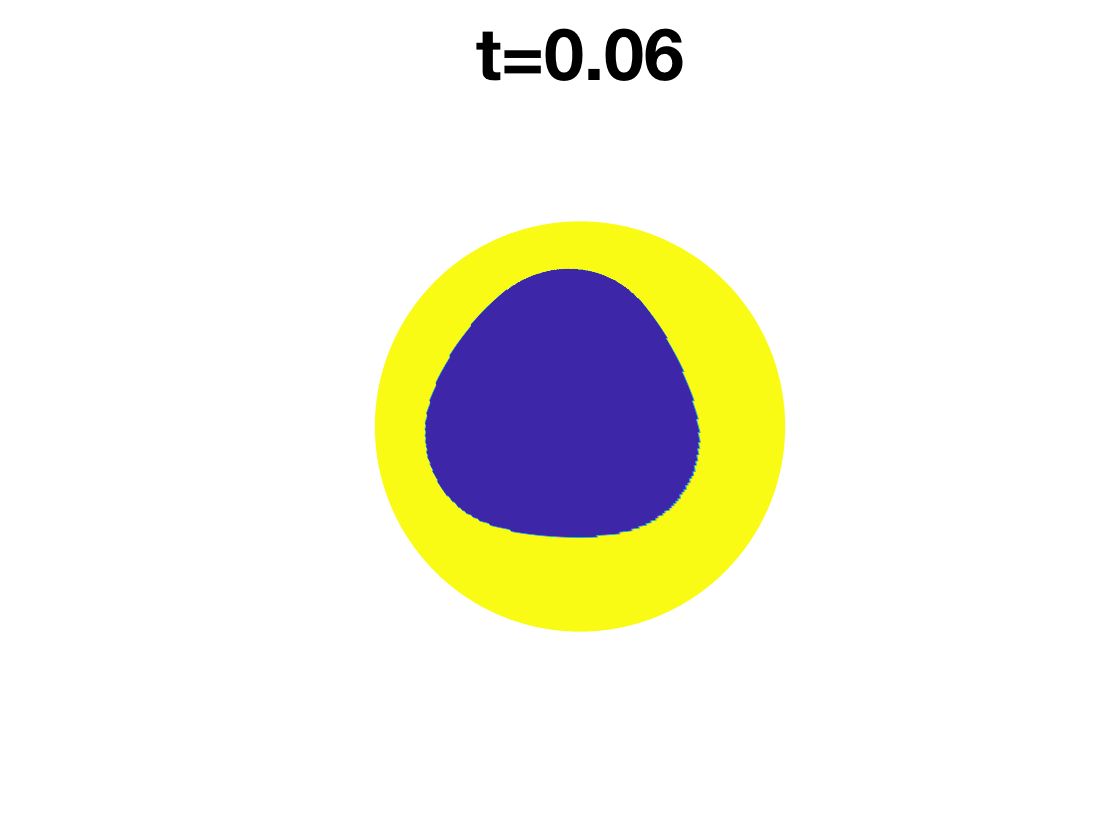}
\includegraphics[scale=0.23,clip,trim= 10cm 1cm 10cm 1cm]{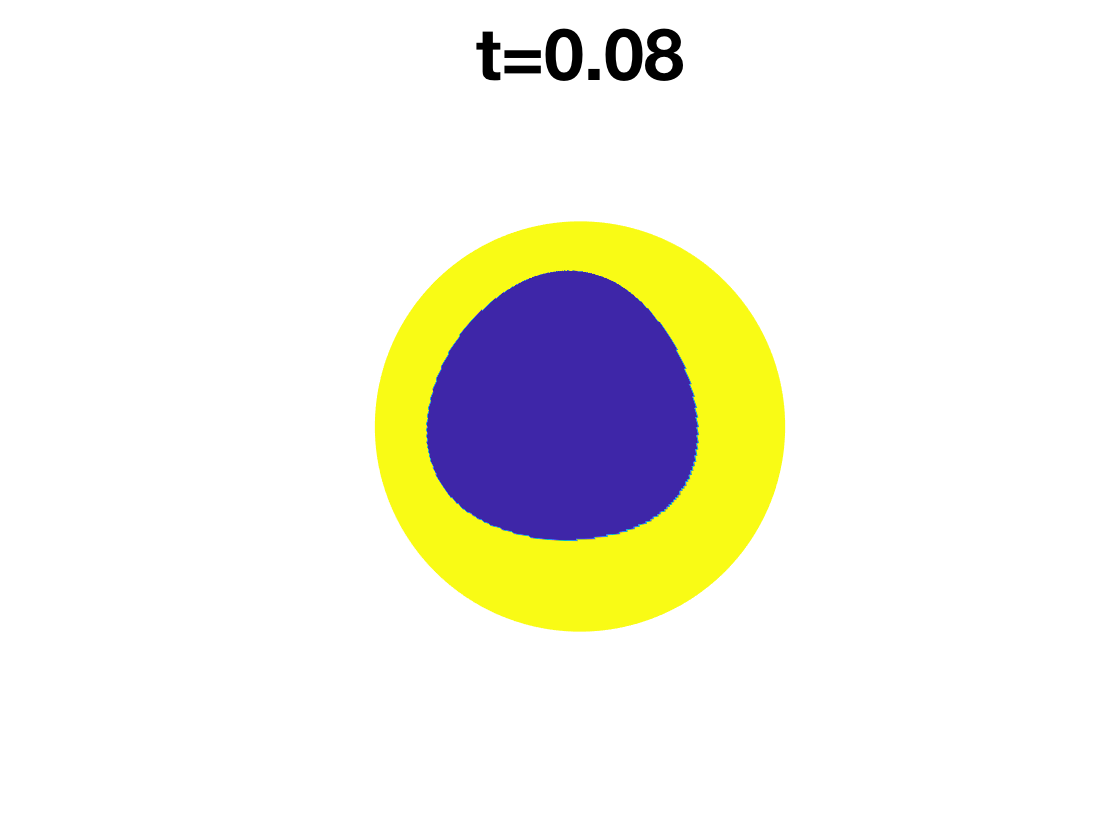}
\includegraphics[scale=0.23,clip,trim= 10cm 1cm 10cm 1cm]{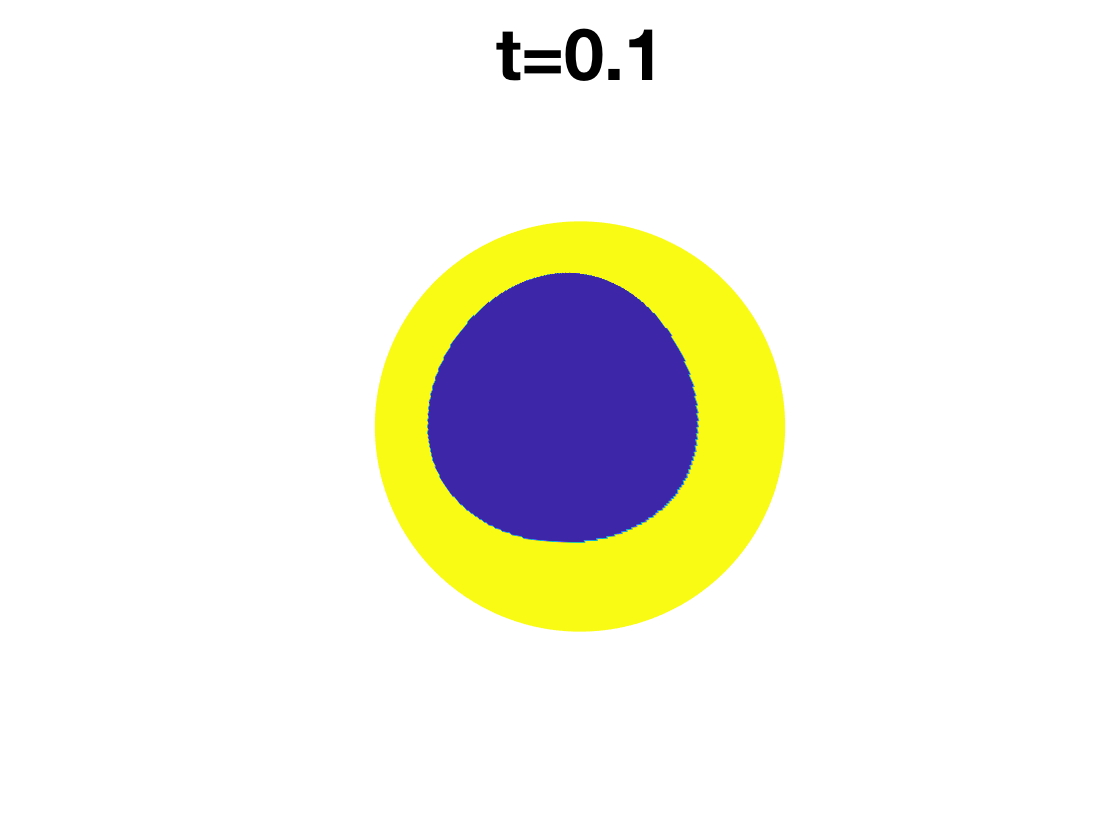}
\includegraphics[scale=0.23,clip,trim= 10cm 1cm 10cm 1cm]{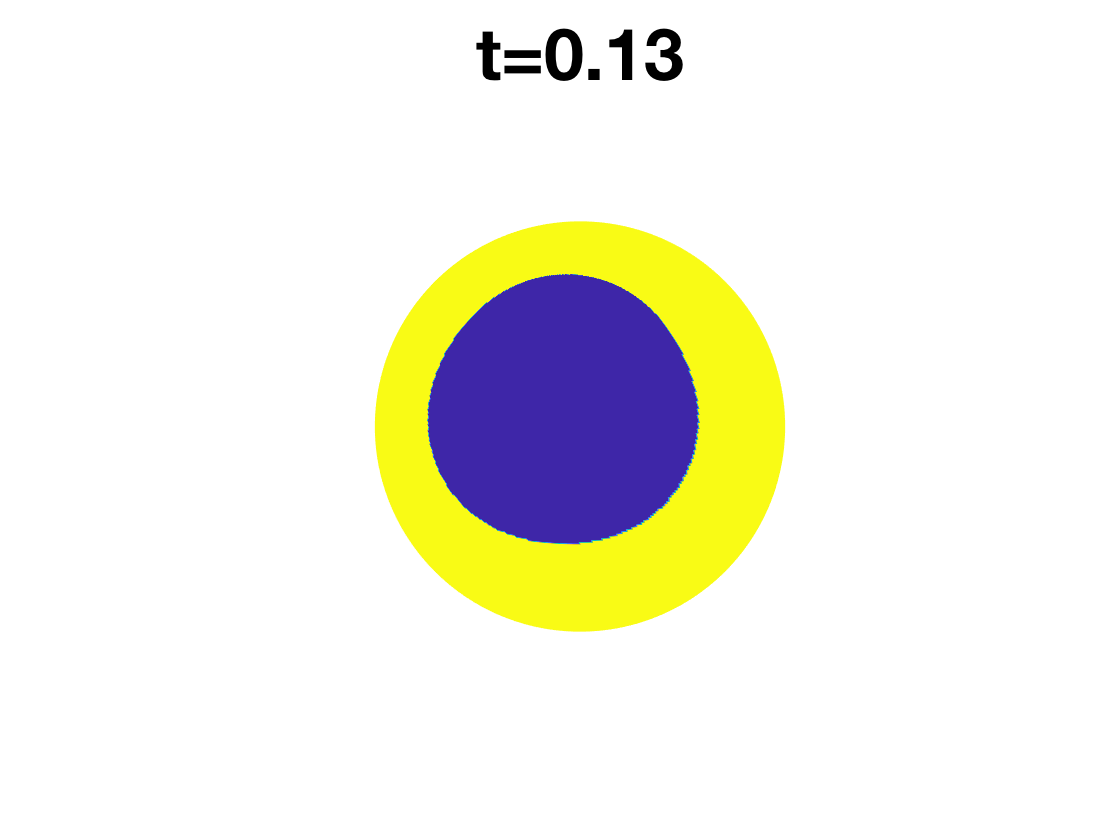}
\includegraphics[scale=0.23,clip,trim= 10cm 1cm 10cm 1cm]{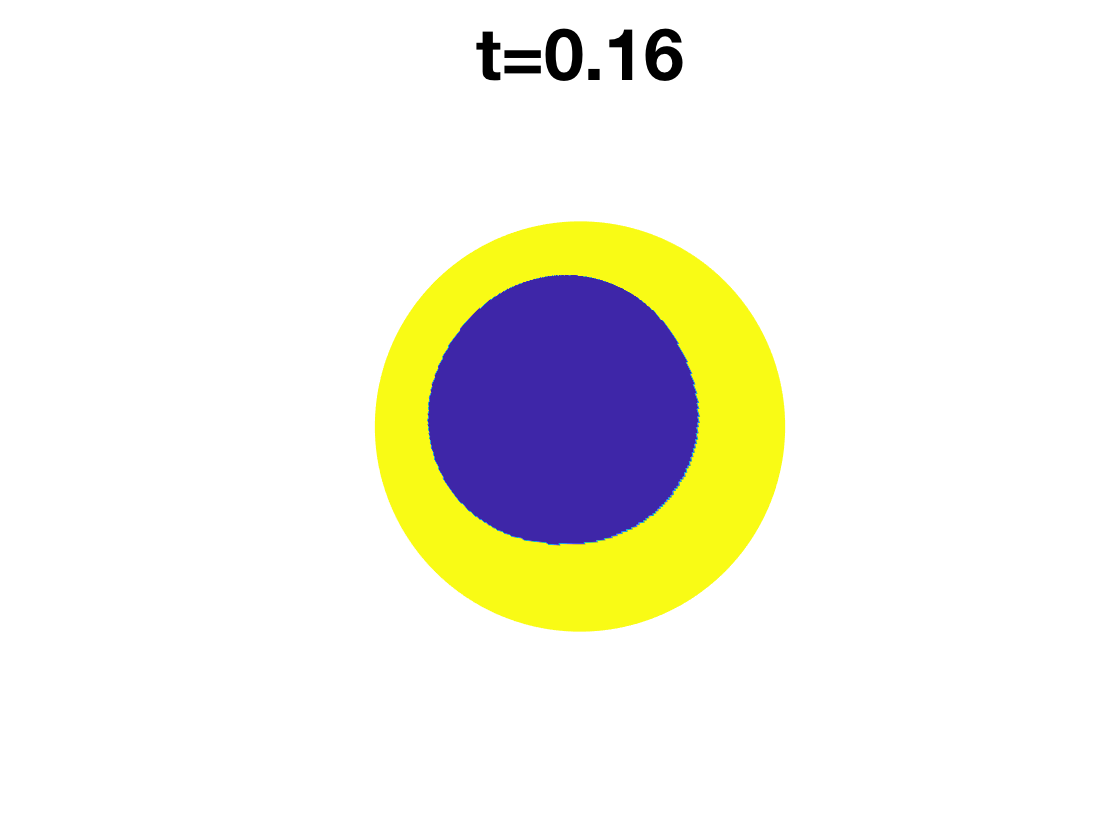}
\includegraphics[scale=0.23,clip,trim= 10cm 1cm 10cm 1cm]{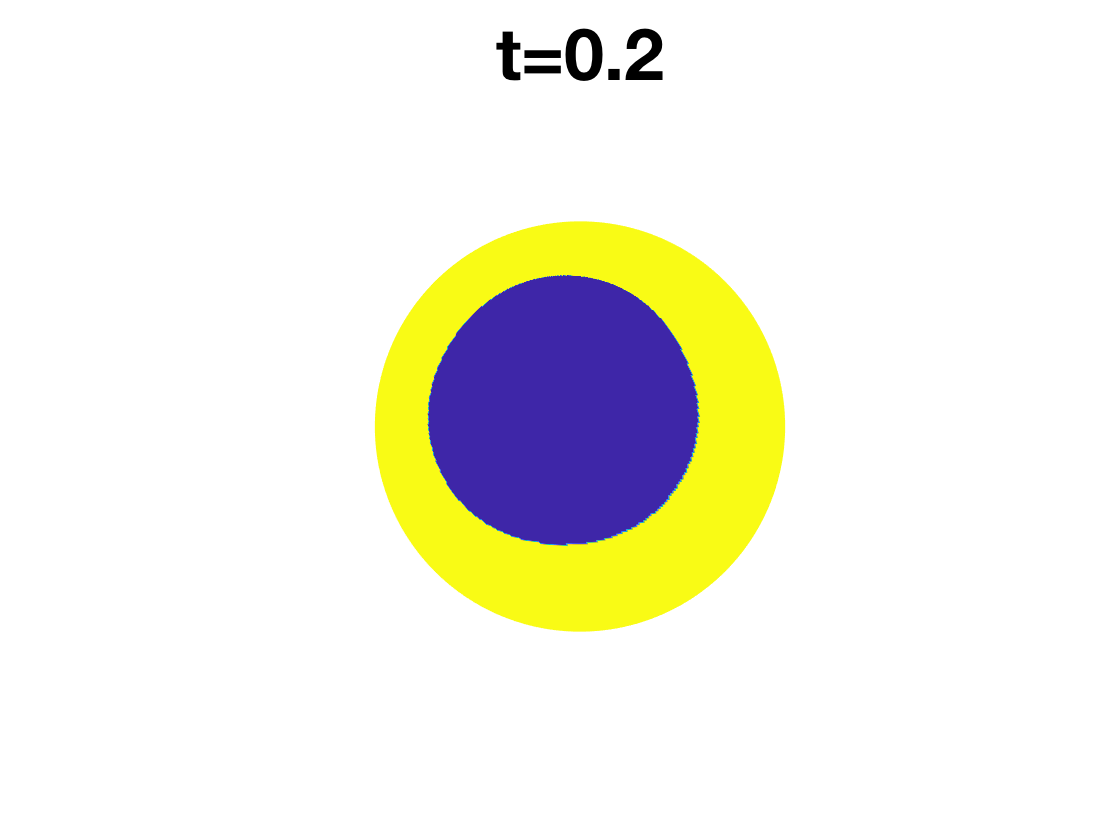}
\caption{Snapshots of the time dynamics of a closed line defect on the sphere surface with a volume constraint. See \S\ref{s:VolConst}.}\label{fig:sphere_volume_preserving}
\end{figure}

\section{Discussion} \label{s:disc}
We introduced a generalization of the MBO diffusion generated motion for orthogonal matrix-valued fields. 
In Section \ref{s:stab}, we proved the stability and convergence of this method by extending the Lyapunov function of Esedoglu and Otto. We implemented this algorithm using the closest point method and non-uniform FFT and used it to perform a variety of numerical  experiments on flat tori and closed surfaces.  There are a variety of open questions that we hope to address in future work. 

In this paper, we only considered a single matrix-valued field that has two ``phases'' given by when the determinant is positive or negative. It would be very interesting to extend this work to the mutli-phase problem as was accomplished for $n=1$ in \cite{esedoglu2015threshold}. 

In this paper, we solved the heat diffusion equation for a matrix-valued function entry by entry at each step by evaluating the convolutions between Green's function and initial condition. When $n=2$ or $n=3$, it is still acceptable. However, for $n$ large, the algorithm will be inefficient at each step, requiring $n^2$ convolutions at each step. 
Since the matrix-valued heat diffusion equation is entry-wise, it is natural to implement the code in parallel so that the solution can be evaluated simultaneously for each entry.

In Section \ref{s:FT}, we considered the case for $n=2$ on a two-dimensional flat torus. We made several observations regarding the index of this field, and it would be good to further analyze this case, possibly using the index of the initial condition to predict the steady state solution. 

It is well-known that boundary conditions in the complex Ginzburg-Landau problem can lead to solutions with point defects, a.k.a., vortices \cite{Bethuel_1994}. Interesting numerical examples can be found in, {\it e.g.}, \cite{Viertel2017}. 
Since we prove in Lemma \ref{l:n=2Equiv} that this energy generalizes the complex Ginzburg-Landau energy, it would be interesting to know what types of singularities are allowed in the higher dimensional fields when boundary conditions are imposed.

\subsection*{Acknowledgements} The authors would like to thank Elena Cherkaev, Matthew Jacobs, Sarang Joshi, Stan Osher, and Dan Spirn for useful discussions during the preparation of this manuscript. D. Wang thanks Xiao-Ping Wang for the continuous encouragement and helpful discussions.  

\clearpage
\printbibliography

\end{document}